\documentclass{amsart}
\usepackage{url}
\usepackage{color}
\usepackage{times}
\usepackage{cite}
\usepackage{enumerate,latexsym}
\usepackage{amsmath,amssymb}
\usepackage{amsthm}
\usepackage{verbatim}
\usepackage{mathrsfs}
\newtheorem{thm}{Theorem}[section]
\newtheorem*{theorem*}{Theorem}
\newtheorem*{acknowledgement*}{Acknowledgement}
\newtheorem{cor}[thm]{Corollary}

\newtheorem{lem}[thm]{Lemma}
\newtheorem{prop}[thm]{Proposition}
\theoremstyle{definition}

\theoremstyle{remark}
\newtheorem{rem}[thm]{Remark}

\numberwithin{equation}{section}

\title[Asymptotically conical self-expanders of mean curvature flow]{The space of asymptotically conical self-expanders of mean curvature flow}

\author{Jacob Bernstein}
\address{Department of Mathematics, Johns Hopkins University, 3400 N. Charles Street, Baltimore, MD 21218}
\email{bernstein@math.jhu.edu}

\author{Lu Wang}
\address{Department of Mathematics, University of Wisconsin-Madison, 480 Lincoln Drive, Madison, WI 53706}
\email{luwang@math.wisc.edu}

\thanks{The first author was partially supported by the NSF Grant DMS-1609340. The second author was partially supported by the NSF Grants DMS-1406240 and  DMS-1834824, an Alfred P. Sloan Research Fellowship, the office of the Vice Chancellor for Research and Graduate Education at the University of Wisconsin-Madison with funding from the Wisconsin Alumni Research Foundation and a Vilas Early Investigator Award.}

\begin{document}
\begin{abstract}
We show that the space of asymptotically conical self-expanders of the mean curvature flow is a smooth Banach manifold. An immediate consequence is that non-degenerate self-expanders -- that is, those self-expanders that admit no non-trivial normal Jacobi fields that fix the asymptotic cone -- are generic in a certain sense.
\end{abstract}
\maketitle

\section{Introduction} \label{IntroSec}
A \emph{hypersurface}, i.e., a properly embedded codimension-one submanifold, $\Sigma\subset\mathbb{R}^{n+1}$, is a \emph{self-expander} if
\begin{equation} \label{ExpanderEqn}
\mathbf{H}_\Sigma-\frac{\mathbf{x}^\perp}{2}=\mathbf{0}.
\end{equation}
Here 
$$
\mathbf{H}_\Sigma=\Delta_\Sigma\mathbf{x}=-H_\Sigma\mathbf{n}_\Sigma=-\mathrm{div}_\Sigma(\mathbf{n}_\Sigma)\mathbf{n}_\Sigma
$$
is the mean curvature vector, $\mathbf{n}_\Sigma$ is the unit normal, and $\mathbf{x}^\perp$ is the normal component of the position vector. Self-expanders arise naturally in the study of mean curvature flow. Indeed, $\Sigma$ is a self-expander if and only if the family of homothetic hypersurfaces
$$
\left\{\Sigma_t\right\}_{t>0}=\left\{\sqrt{t}\, \Sigma\right\}_{t>0}
$$
is a \emph{mean curvature flow} (MCF), that is, a solution to the flow
$$
\left(\frac{\partial \mathbf{x}}{\partial t}\right)^\perp=\mathbf{H}_{\Sigma_t}.
$$
In particular, self-expanders may be thought of as the forward in time analog of self-shrinkers.  While self-shrinkers model the behavior of a MCF as it develops a singularity, self-expanders are expected to model the behavior of a MCF as it emerges from a conical singularity. They are also expected to model the long time behavior of the flow. The interested reader may refer to \cite{AngenentIlmanenChopp}, \cite{CMGen}, \cite{ClutterbuckSchnurer}, \cite{Ding}, \cite{Ecker}, \cite{Huisken}, \cite{IlmanenLec}, \cite{MazzeoSaez},\cite{SchnurerSchulze}, and references therein.
Finally, self-expanders arise variationally as stationary points, with respect to compactly supported variations, of the functional
$$
E[\Sigma]=\int_{\Sigma} e^{\frac{|\mathbf{x}|^2}{4}}\, d\mathcal{H}^n
$$
where $\mathcal{H}^n$ is the $n$-dimensional Hausdorff measure.

Throughout the paper, $n,k\geq 2$ are integers and $\alpha\in (0,1)$. Let $\Gamma$ be a $C^{k,\alpha}_{*}$ asymptotically conical $C^{k,\alpha}$-hypersurface in $\mathbb{R}^{n+1}$ and let $\mathcal{L}(\Gamma)$   the link of the asymptotic cone of $\Gamma$.  For instance, if $\lim_{\rho\to 0^+} \rho \Gamma=\mathcal{C}$ in $C^{k, \alpha}_{loc} (\mathbb{R}^{n+1}\setminus\{\mathbf{0}\})$ for $\mathcal{C}$ a cone, then $\Gamma$ is $C^{k,\alpha}_{*}$-asymptotically conical with asymptotic cone $\mathcal{C}$. For technical reasons, the actual definition is slightly weaker -- see Section \ref{ACHSec} for the details. We wish to view the space of all $C^{k,\alpha}_{*}$-asymptotically conical self-expanders as an infinite dimensional geometric object and then understand some of its global features. However, this space is cumbersome to work with and so, inspired by work of White \cite{WhiteEI}, we take a slightly different point of view.

Specifically, let $\mathcal{ACH}^{k,\alpha}_n(\Gamma)$ be the space of $C^{k,\alpha}_*$-asymptotically conical embeddings of $\Gamma$ into $\mathbb{R}^{n+1}$.  Roughly speaking,  $\mathbf{f}\in \mathcal{ACH}^{k,\alpha}_n(\Gamma)$ if $\mathbf{f}$ is an embedding, $\Sigma=\mathbf{f}(\Gamma)$ is a $C^{k,\alpha}_{*}$-asymptotically conical $C^{k,\alpha}$-hypersurface and the parametrization of $\Sigma$ by $\mathbf{f}$ is well behaved at infinity.  This last condition means that an appropriate blow down of $\mathbf{f}$ yields a nice parametrization of the asymptotic cone of $\Sigma$. Two elements $\mathbf{f}_1,\mathbf{f}_2\in\mathcal{ACH}^{k,\alpha}_n(\Gamma)$ are equivalent if $\mathbf{f}_1=\mathbf{f}_2\circ\phi$ for some diffeomorphism, $\phi$, of $\Gamma$ that fixes the infinity of $\Gamma$; see Section \ref{ACESubsec} for the exact definitions of $\mathcal{ACH}^{k,\alpha}_n(\Gamma)$ and of the equivalence. Denote by $[\mathbf{f}]$ the equivalence class of $\mathbf{f}$. Observe, that it is possible for two elements $\mathbf{f}_1, \mathbf{f}_2\in  \mathcal{ACH}^{k,\alpha}_n(\Gamma)$ to parametrize the same hypersurface, $\Sigma$, but to be inequivalent. This happens when the blow downs of $\mathbf{f}_1$ and $\mathbf{f}_2$ give different parameterizations of the asymptotic cone of $\Sigma$.

\begin{thm} \label{StructureThm}
For $\Gamma\in\mathcal{ACH}^{k,\alpha}_n$ let 
$$
\mathcal{ACE}^{k,\alpha}_n(\Gamma)=\left\{[\mathbf{f}]\colon \mbox{$\mathbf{f}\in\mathcal{ACH}^{k,\alpha}_n(\Gamma)$ and $\Sigma=\mathbf{f}(\Gamma)$ satisfies \eqref{ExpanderEqn}}\right\}.
$$
Then the following statements hold:
\begin{enumerate}
\item \label{ManifoldItem} $\mathcal{ACE}^{k,\alpha}_n(\Gamma)$ is a smooth Banach manifold modeled on $C^{k,\alpha}(\mathcal{L}(\Gamma);\mathbb{R}^{n+1})$, with a countable cover by coordinate domains $\mathcal{O}_i$. 
\item \label{ProjectionItem} The projection map $\Pi\colon\mathcal{ACE}^{k,\alpha}_n(\Gamma)\to C^{k,\alpha}(\mathcal{L}(\Gamma);\mathbb{R}^{n+1})$ which assigns to $[\mathbf{f}]$ the trace at infinity, $\mathrm{tr}_\infty^1[\mathbf{f}]$, of $\mathbf{f}$ is a smooth map of Fredholm index $0$. 
\item \label{LocProjectionItem} Each $\Pi|_{\mathcal{O}_i}$ has a coordinate representation given by the map $(z,\kappa)\mapsto (z,\psi_i(z,\kappa))$ from a neighborhood of $0 \in \mathcal{Z}_i\oplus\mathcal{K}_i$ to itself, where $\mathcal{K}_i$ is the kernel of $D\Pi([\mathbf{f}_i])$ for some $[\mathbf{f}_i]\in\mathcal{O}_i$ and $\mathcal{Z}_i$ is the complement of $\mathcal{K}_i$ in $C^{k,\alpha}(\mathcal{L}(\Gamma);\mathbb{R}^{n+1})$. 
\item  \label{KernelItem} The kernel of $D\Pi([\mathbf{f}])$ is isomorphic to the space of normal Jacobi fields of $\mathbf{f}(\Gamma)$ that fix the asymptotic cone.
\end{enumerate}
\end{thm}

By work of Smale \cite{Smale}, an immediate consequence of Theorem \ref{StructureThm} is that for a ``generic" $C^{k,\alpha}$-regular cone, $\mathcal{C}$, all self-expanders asymptotic to $\mathcal{C}$ are non-degenerate in that they admit no non-trivial normal Jacobi fields that fix the asymptotic cone -- see Section \ref{JacobiSec}. 

\begin{cor} \label{GenConeCor}
Given $\Gamma\in\mathcal{ACH}^{k,\alpha}_n$ there is a nowhere dense set $\mathcal{S}\subset C^{k,\alpha}(\mathcal{L}(\Gamma); \mathbb{R}^{n+1})$ so that if $\varphi \in C^{k,\alpha}(\mathcal{L}(\Gamma); \mathbb{R}^{n+1})\backslash \mathcal{S}$ and $[\mathbf{f}]\in \mathcal{ACE}^{k,\alpha}_n(\Gamma)$ has $\Pi( [\mathbf{f}])=\varphi$, then the space of normal Jacobi fields of $\Sigma=\mathbf{f}(\Gamma)$ that fix the asymptotic cone of $\Sigma$ is trivial. That is, the space $\mathcal{K}$ defined in \eqref{KspaceEqn} is trivial.
\end{cor}

Similarly, it follows from Theorem \ref{StructureThm} that whenever the projection map $\Pi$ is proper, then there is a well defined mod 2 degree -- see \cite{BWProperness} for natural situations in which $\Pi$ is proper. 

\begin{cor} \label{Mod2DegreeCor}
Given $\Gamma\in\mathcal{ACH}^{k,\alpha}_n$ if $\mathcal{U}$ is an open subset of $\mathcal{ACE}^{k,\alpha}_n(\Gamma)$ and $\mathcal{V}$ is a connected open subset of $C^{k,\alpha}(\mathcal{L}(\Gamma);\mathbb{R}^{n+1})$ such that $\Pi|_{\mathcal{U}}\colon\mathcal{U}\to \mathcal{V}$ is proper, then $\Pi|_{\mathcal{U}}$ has a well defined mod 2 degree.  
\end{cor}

Theorem \ref{StructureThm} is inspired by the main result of \cite{WhiteEI} in which White considers the space of all regular immersions of compact $n$-dimensional manifolds with boundary into $\mathbb{R}^{N}$, $N>n$.  He shows that the critical points in this space of any reasonable elliptic integrand, e.g.,  the area functional,  form a Banach manifold and that the natural boundary restriction map is smooth and Fredholm of index $0$.  White derives many interesting consequences from this result including developing an integer degree theory -- see \cite{WhiteMD} for further applications. The results of \cite{WhiteEI} were preceded by earlier work on this problem for minimal surfaces by B\"{o}hme-Tromba \cite{BohmeTromba} and Tomi-Tromba \cite{TomiTromba}.  Those authors took a different approach, one allowing for branched immersions, and the interested reader should consult those papers and their references. 

Theorem \ref{StructureThm} may be thought of as an extension of \cite{WhiteEI} to the space of asymptotically conical self-expanders and we follow the main outline of White's proof. However, almost all the technical details are different. We also do not fully carry out the program of \cite{WhiteEI}. For example, we do not develop an integer degree theory for the map $\Pi$, instead we address this topic in a separate paper \cite{BWIntegerDegree}. To elaborate on the differences, while White is able to repeatedly appeal to standard results for elliptic boundary value problems, we must develop the relevant estimates in the non-compact setting ourselves. In addition to the issues caused by working on a non-compact domain, there are new challenges posed by  the lack of ellipticity of the linearizion of \eqref{ExpanderEqn} ``at infinity". For instance, we establish the Fredholm properties of this operator using basic facts about the heat kernel on Euclidean space. We remark that an analogous Fredholm property for the linearization of asymptotically conical expanding Ricci solitons has also been derived by Deruelle \cite{Deruelle} though our setup and method of deriving the relevant estimates are different from his. 

There is a rich literature on the use of global analysis methods to study elliptic variational problems.  For instance, the problem of studying the structure of the space of compact minimal surfaces where the ambient metric is allowed to vary has been considered by White in \cite{WhiteVM} for closed surfaces and also by Maximo-Nunes-Smith \cite{MaximoNunesSmith} for free boundary annuli. In addition, there are two papers where an integer degree theory is constructed for non-compact minimal surfaces. Namely, \cite{AlexakisMazzeo} where such a theory is sketched for non-compact surfaces in $\mathbb{H}^3$ which extend to surfaces with boundary on the ideal boundary and  \cite{FMMR} where the theory is used to study certain non-compact annuli in $\mathbb{H}^2\times \mathbb{R}$. 

We organize the paper as follows. In Section \ref{NotationSec} we fix notation. In Section \ref{ACHSec} we introduce $C^{k,\alpha}_{*}$-asymptotic conical hypersurfaces/embeddings and the trace at infinity. In Section \ref{VariationSec} we derive the first and second variation formula for the functional $E$. Because of the low regularity at infinity of $C^{k,\alpha}_*$-asymptotically conical self-expanders, inspired by work of White \cite{WhiteEI}, we associate a so-called $\mathbf{v}$-Jacobi operator, $L_{\mathbf{v}}$, which is a variant of the usual Jacobi operator, to the second variation of $E$ where $\mathbf{v}$ is a $C^{k,\alpha}$ vector field transverse to the given self-expander. In Section \ref{FredholmSec} we prove that $L_{\mathbf{v}}$ is Fredholm of index $0$ between appropriate function spaces. In Section \ref{JacobiSec} we adapt  to our setting an asymptotic expansion of almost eigenfunctions of certain drifted Laplacians obtained recently by the first author \cite{Bernstein}. This will be used to justify the legitimacy of integration by parts in several places and also serves as a substitute for the Calder\'{o}n unique continuation theorem employed in the proof of \cite[Theorem 3.2]{WhiteEI}. In Section \ref{StructureSec} we prove a so-called smooth dependence theorem, i.e., Theorem \ref{SmoothDependThm}, of which Theorem \ref{StructureThm} is an easy corollary. Theorem \ref{SmoothDependThm} may be thought of as an analog of \cite[Theorem 3.2]{WhiteEI} in the asymptotically conical self-expander setting. 

\section{Notation} \label{NotationSec}

\subsection{Basic notions} \label{NotionSubsec}
Denote a (open) ball in $\mathbb{R}^n$ of radius $R$ and center $x$ by $B_R^n(x)$ and the closed ball by $\bar{B}^n_R(x)$. We often omit the superscript $n$ when its value is clear from context. We also omit the center when it is the origin. 

For an open set $U\subset\mathbb{R}^{n+1}$, a \emph{hypersurface in $U$}, $\Sigma$, is a smooth, properly embedded, codimension-one  submanifold of $U$.
 We also consider hypersurfaces of lower regularity and given an integer $k\geq 2$ and $\alpha\in (0,1)$ we define a \emph{$C^{k,\alpha}$-hypersurface in $U$} to be a properly embedded, codimension-one $C^{k,\alpha}$ submanifold of $U$. 
 When needed, we distinguish between a point $p\in\Sigma$ and its \emph{position vector} $\mathbf{x}(p)$.

Consider the hypersurface $\mathbb{S}^n\subset\mathbb{R}^{n+1}$, the unit $n$-sphere in $\mathbb{R}^{n+1}$. A \emph{hypersurface in $\mathbb{S}^n$}, $\sigma$, is a closed, embedded, codimension-one smooth submanifold of $\mathbb{S}^n$ and \emph{$C^{k,\alpha}$-hypersurfaces in $\mathbb{S}^n$} are defined likewise. Observe, that $\sigma$ is a closed codimension-two submanifold of $\mathbb{R}^{n+1}$ and so we may associate to each point $p\in\sigma$ its position vector $\mathbf{x}(p)$. Clearly, $|\mathbf{x}(p)|=1$.

A \emph{cone} is a set $\mathcal{C}\subset\mathbb{R}^{n+1}\setminus\{\mathbf{0}\}$ that is dilation invariant around the origin. That is, $\rho\mathcal{C}=\mathcal{C}$ for all $\rho>0$. The \emph{link} of the cone is the set $\mathcal{L}[\mathcal{C}]=\mathcal{C}\cap\mathbb{S}^{n}$. The cone is \emph{regular} if its link is a smooth hypersurface in $\mathbb{S}^{n}$ and \emph{$C^{k,\alpha}$-regular} if its link is a $C^{k,\alpha}$-hypersurface in $\mathbb{S}^n$. For any hypersurface $\sigma\subset\mathbb{S}^n$ the \emph{cone over $\sigma$}, $\mathcal{C}[\sigma]$, is the cone defined by 
$$
\mathcal{C}[\sigma]=\left\{\rho p\colon p\in\sigma, \rho>0\right\}\subset\mathbb{R}^{n+1}\setminus\{\mathbf{0}\}.
$$
Clearly, $\mathcal{L}[\mathcal{C}[\sigma]]=\sigma$. 

\subsection{Function spaces} \label{FunctionSubsec}
Let $\Sigma$ be a properly embedded, $C^{k,\alpha}$ submanifold of an open set $U\subset\mathbb{R}^{n+1}$. There is a natural Riemannian metric, $g_\Sigma$, on $\Sigma$ of class $C^{k-1,\alpha}$ induced from the Euclidean one. As we always take $k\geq 2$, the Christoffel symbols of this metric,  in appropriate coordinates, are well defined and of regularity $C^{k-2,\alpha}$.  Let $\nabla_\Sigma$ be the covariant derivative on $\Sigma$. Denote by $d_\Sigma$ the geodesic distance on $\Sigma$ and by $B^\Sigma_R(p)$ the (open) geodesic ball in $\Sigma$ of radius $R$ and center $p\in\Sigma$. For $R$ small enough so that $B_{R}^\Sigma(p)$ is strictly geodesically convex and $q\in B^\Sigma_R(p)$, denote by $\tau^\Sigma_{p,q}$ the parallel transport along the unique minimizing geodesic in $B^\Sigma_R(p)$ from $p$ to $q$. 

For the rest of this section, let $\Omega$ be a domain in $\Sigma$, $l$ an integer in $[0,k]$, $\beta\in (0,1)$ and $d\in\mathbb{R}$. Suppose $l+\beta\leq k+\alpha$. We first consider the following norm for functions on $\Omega$:
$$
\Vert f\Vert_{l; \Omega}=\sum_{i=0}^l \sup_{\Omega} |\nabla_\Sigma^i f|.
$$
We then let
$$
C^l(\Omega)=\left\{f\in C_{loc}^l(\Omega)\colon \Vert f\Vert_{l; \Omega}<\infty\right\}.
$$
We next define the H\"{o}lder semi-norms for functions $f$ and tensor fields $T$ on $\Omega$: 
$$
[f]_{\beta; \Omega} =\sup_{\substack{p,q\in\Omega \\ q\in B^\Sigma_{\delta}(p)\setminus\{p\}}} \frac{|f(p)-f(q)|}{d_\Sigma(p,q)^\beta} 
\mbox{ and } 
[T]_{\beta; \Omega} =\sup_{\substack{p,q\in\Omega \\ q\in B^\Sigma_{\delta}(p)\setminus\{p\}}} \frac{|T(p)-(\tau^\Sigma_{p,q})^* T(q)|}{d_\Sigma(p,q)^\beta},
$$
where $\delta=\delta(\Sigma,\Omega)>0$ so that for all $p\in\Omega$, $B^\Sigma_\delta(p)$ is strictly geodesically convex. We further define the norm for functions on $\Omega$:
$$
\Vert f\Vert_{l, \beta; \Omega}=\Vert f\Vert_{l; \Omega}+[\nabla_\Sigma^l f]_{\beta; \Omega},
$$
and let 
$$
C^{l, \beta}(\Omega)=\left\{f\in C_{loc}^{l, \beta}(\Omega)\colon \Vert f\Vert_{l, \beta; \Omega}<\infty\right\}.
$$

We also define the following weighted norms for functions on $\Omega$:
$$
\Vert f\Vert_{l; \Omega}^{(d)}=\sum_{i=0}^l\sup_{p\in\Omega} \left(|\mathbf{x}(p)|+1\right)^{-d+i} |\nabla_\Sigma^i f(p)|.
$$
We then let 
$$
C^{l}_d(\Omega)=\left\{f\in C^l_{loc}(\Omega)\colon \Vert f\Vert_{l; \Omega}^{(d)}<\infty\right\}.
$$
We further define the following weighted H\"{o}lder semi-norms for functions $f$ and tensor fields $T$ on $\Omega$:
\begin{align*}
[f]_{\beta; \Omega}^{(d)} & =\sup_{\substack{p,q\in\Omega \\ q\in B^\Sigma_{\delta_p}(p)\setminus\{p\}}} \left((|\mathbf{x}(p)|+1)^{-d+\beta}+(|\mathbf{x}(q)|+1)^{-d+\beta}\right) \frac{|f(p)-f(q)|}{d_\Sigma(p,q)^\beta}, \mbox{ and}, \\
[T]_{\beta; \Omega}^{(d)} & =\sup_{\substack{p,q\in\Omega \\ q\in B^\Sigma_{\delta_p}(p)\setminus\{p\}}} \left((|\mathbf{x}(p)|+1)^{-d+\beta}+(|\mathbf{x}(q)|+1)^{-d+\beta}\right) \frac{|T(p)-(\tau^\Sigma_{p,q})^* T(q)|}{d_\Sigma(p,q)^\beta},
\end{align*}
where $\eta=\eta(\Omega,\Sigma)\in \left(0,\frac{1}{4}\right)$ so that for any $p\in\Sigma$, letting $\delta_p=\eta (|\mathbf{x}(p)|+1)$, $B_{\delta_p}^\Sigma(p)$ is strictly geodesically convex. Finally, we define the norm for functions on $\Omega$:
$$
\Vert f\Vert_{l, \beta; \Omega}^{(d)}=\Vert f\Vert_{l; \Omega}^{(d)}+[\nabla_\Sigma^l f]_{\beta; \Omega}^{(d-l)},
$$
and we let
$$
C^{l,\beta}_d(\Omega)=\left\{f\in C^{l,\beta}_{loc}(\Omega)\colon \Vert f\Vert_{l, \beta; \Omega}^{(d)}<\infty\right\}.
$$

The norms $\Vert\cdot\Vert_{l, \beta; \Omega}$ and $\Vert\cdot\Vert_{l, \beta; \Omega}^{(d)}$ are equivalent for different choices of $\delta$ and $\eta$. We also often omit $\Omega$ when it is clear from context. It is a standard exercise to verify that all the spaces defined above are Banach spaces. It is also straightforward to extend them to $\mathbb{R}^M$-valued maps and to tensor fields.  

\begin{rem}
The spaces $C^{l,\beta}_d(\Omega)$ are the interpolation spaces $(C^{l}_d(\Omega), C^{l+1}_d(\Omega))_{\beta,\infty}$ -- compare \cite[Remark 2.1]{Deruelle}. When $\Omega$ is asymptotically conical, our weighted H\"{o}lder norms can be related to appropriate unweighted ones on compact domains -- see Item \eqref{ListProp1} of Proposition \ref{ListProp} -- so the verification of basic properties for weighted H\"{o}lder continuous functions become standard.
\end{rem}

Finally, we introduce the convention that $C^{l,0}_{loc}=C^{l}_{loc}$, $C^{l,0}=C^l$ and $C^{l,0}_d=C^l_d$ and that $C^{0, \beta}_{loc}=C^\beta_{loc}$,  $C^{0,\beta}=C^\beta$ and $C^{0,\beta}_d=C^\beta_d$. The notation for the corresponding norms is abbreviated in the same fashion.

\subsection{Homogeneous functions and homogeneity at infinity} \label{HomoFuncSubsec}
Fix a $C^{k,\alpha}$-regular cone $\mathcal{C}$ with its link $\mathcal{L}$. By our definition $\mathcal{C}$ is a $C^{k,\alpha}$-hypersurface in $\mathbb{R}^{n+1}\setminus\{\mathbf{0}\}$. For $R>0$ let $\mathcal{C}_R=\mathcal{C}\setminus\bar{B}_R$. There is an $\eta=\eta(\mathcal{L},R)>0$ so that for any $p\in\mathcal{C}_R$, $B^{\mathcal{C}}_{\delta_p}(p)$ is strictly geodesically convex, where $\delta_p=\eta(|\mathbf{x}(p)|+1)$. We also fix an integer $l\in [0,k]$ and $\beta\in [0,1)$ with $l+\beta\leq k+\alpha$.

A map $\mathbf{f}\in C^{l,\beta}_{loc}(\mathcal{C}; \mathbb{R}^M)$ is \emph{homogeneous of degree $d$} if $\mathbf{f}(\rho p)=\rho^d \mathbf{f}(p)$ for all $p\in\mathcal{C}$ and $\rho>0$. Given a map $\varphi\in C^{l,\beta}(\mathcal{L}; \mathbb{R}^M)$ the \emph{homogeneous extension of degree $d$ of $\varphi$} is the map $\mathscr{E}_d^{\mathrm{H}}[\varphi]\in C^{l,\beta}_{loc}(\mathcal{C}; \mathbb{R}^M)$ defined by 
$$
\mathscr{E}_d^{\mathrm{H}}[\varphi](p)=|\mathbf{x}(p)|^d \varphi(|\mathbf{x}(p)|^{-1}p).
$$
Conversely, given a homogeneous $\mathbb{R}^M$-valued map of degree $d$, $\mathbf{f}\in C^{l,\beta}_{loc}(\mathcal{C}; \mathbb{R}^M)$, let $\varphi=\mathrm{tr}[\mathbf{f}]\in C^{l,\beta}(\mathcal{L}; \mathbb{R}^M)$, the \emph{trace} of $\mathbf{f}$, be the restriction of $\mathbf{f}$ to $\mathcal{L}$. Clearly, $\mathbf{f}$ is the homogeneous extension of degree $d$ of $\varphi$.

A map $\mathbf{g}\in C^{l,\beta}_{loc}(\mathcal{C}_R; \mathbb{R}^M)$ is \emph{asymptotically homogeneous of degree $d$} if 
$$
\lim_{\rho\to 0^+} \rho^d \mathbf{g}(\rho^{-1}p)=\mathbf{f}(p) \mbox{ in $C^{l,\beta}_{loc}(\mathcal{C}; \mathbb{R}^M)$}
$$
for some $\mathbf{f}\in C^{l,\beta}_{loc}(\mathcal{C}; \mathbb{R}^M)$ that is homogeneous of degree $d$. For such a $\mathbf{g}$ we define the \emph{trace at infinity} of $\mathbf{g}$ by $\mathrm{tr}^d_\infty[\mathbf{g}]=\mathrm{tr}[\mathbf{f}]$. Observe, that if $\mathbf{g}\in C_{loc}^{l,\beta}(\mathcal{C}_R; \mathbb{R}^M)$ is asymptotically homogeneous of degree $d$, then $\mathbf{g}\in C^{l,\beta}_d (\mathcal{C}_R; \mathbb{R}^M)$ but the reverse need not be true. We define
$$
C^{l,\beta}_{d,\mathrm{H}}(\mathcal{C}_R; \mathbb{R}^M)=\left\{\mathbf{g}\in C^{l,\beta}_d(\mathcal{C}_R; \mathbb{R}^M)\colon \mbox{$\mathbf{g}$ is asymptotically homogeneous of degree $d$}\right\}.
$$
It is straightforward to verify that $C^{l,\beta}_{d,\mathrm{H}}(\mathcal{C}_R; \mathbb{R}^M)$ is a closed subspace of $C^{l,\beta}_d(\mathcal{C}_R; \mathbb{R}^M)$ and that
$$
\mathrm{tr}^d_\infty\colon C^{l,\beta}_{d,\mathrm{H}}(\mathcal{C}_R; \mathbb{R}^M)\to C^{l,\beta}(\mathcal{L}; \mathbb{R}^M)
$$
is a bounded linear map. Hence, $C^{l,\beta}_{d,0}(\mathcal{C}_R; \mathbb{R}^{M})=\ker\left( \mathrm{tr}^d_\infty\right)$ is a closed subspace of $C^{l,\beta}_{d,\mathrm{H}}(\mathcal{C}_R; \mathbb{R}^M)$.  Moreover, the construction of $\mathrm{tr}_\infty^d$ and the Arzel\`{a}-Ascoli theorem ensure that for $l+\beta\leq l^\prime+\beta^\prime\leq k+\alpha$,
 \begin{equation} \label{ClosedCondEqn}
\mathbf{g}\in C^{l,\beta}_{d,\mathrm{H}} (\mathcal{C}_R; \mathbb{R}^M)\cap C^{l^\prime, \beta^\prime}_{d}(\mathcal{C}_R; \mathbb{R}^M)\Rightarrow \mathrm{tr}^d_\infty[\mathbf{g}]\in C^{l^\prime, \beta^\prime}_{d}(\mathcal{C}_R; \mathbb{R}^M).
 \end{equation}
Finally, observe that $\mathbf{x}|_{\mathcal{C}_R}\in C^{k,\alpha}_{1,\mathrm{H}}(\mathcal{C}_R; \mathbb{R}^{n+1})$ and $\mathrm{tr}_\infty^1[\mathbf{x}|_{\mathcal{C}_R}]=\mathbf{x}|_{\mathcal{L}}$.

\subsection{Transverse sections} \label{TransverseSubsec}
The unit normal vector field of a (two-sided) $C^{k,\alpha}$-hypersurface is, in general, of class $C^{k-1,\alpha}$. In order to avoid difficulties introduced by this loss of derivative, we adopt a different notion. Let $\Sigma$ be a $C^{k,\alpha}$-hypersurface in some open set of $\mathbb{R}^{n+1}$. We say a $C^{k,\alpha}$ vector field $\mathbf{v}\colon\Sigma\to\mathbb{R}^{n+1}$ along $\Sigma$, is a \emph{transverse section} if
\begin{itemize}
\item $|\mathbf{v}|=1$;
\item $\mathbf{v}(p)$ does not lie in $T_p\Sigma$.
\end{itemize}
Given a transverse section, $\mathbf{v}$, on $\Sigma$ we define a $C^{k,\alpha}$ map $\gamma_{\mathbf{v}}\colon\Sigma\times\mathbb{R}\to\mathbb{R}^{n+1}$ by
$$
(p,s)\mapsto\gamma_{\mathbf{v}}(p,s)=\mathbf{x}(p)+s\mathbf{v}(p).
$$
An open neighborhood, $U$, of $\Sigma$ is a \emph{$\mathbf{v}$-regular neighborhood} if there is an open neighborhood $W$ of $\Sigma\times\{0\}$ so that $\gamma_{\mathbf{v}}(W)=U$ and $\gamma_{\mathbf{v}}$ restricts to a $C^{k,\alpha}$ diffeomorphism of $W$ onto $U$. If, in addition, $W$ is of the form $\Sigma\times (-\epsilon,\epsilon)$, we denote $U$ by $N_\epsilon(\Sigma,\mathbf{v})$ and say that $\Sigma$ is \emph{$\epsilon$-regular with respect to $\mathbf{v}$}. It is straightforward to check that $D\gamma_{\mathbf{v}}$ at $(p,0)$ is invertible. As a consequence, such $\mathbf{v}$-regular neighborhoods exist. If $\Sigma$ is compact, by the inverse function theorem, it is $\epsilon$-regular with respect to $\mathbf{v}$ for some $\epsilon=\epsilon(\Sigma,\mathbf{v})>0$.

Fix a transverse section, $\mathbf{v}$, on $\Sigma$. Given a function $f\in C_{loc}^{k,\alpha}(\Sigma)$ let 
$$
\Sigma_f=\left\{\gamma_{\mathbf{v}}(p,f(p))\colon p\in\Sigma\right\},
$$
and we call such a set a \emph{$\mathbf{v}$-graph of function $f$}. Clearly, if $U$ is a $\mathbf{v}$-regular neighborhood of $\Sigma$ and $(p,f(p))\in\gamma_{\mathbf{v}}^{-1}(U)$ for all $p\in\Sigma$, then $\Sigma_f$ is a $C^{k,\alpha}$-hypersurface in $U$.

Given a $\mathbf{v}$-regular neighborhood, $U$, of $\Sigma$ define a $C^{k,\alpha}$ projection map $\pi_{\mathbf{v}}\colon U\to\Sigma$ given by $\pi_{\mathbf{v}}=\pi_1\circ\gamma_{\mathbf{v}}^{-1}$, where $\pi_1$ is the natural projection from $\Sigma\times\mathbb{R}$ onto $\Sigma$ defined by $\pi_1(p,s)=p$. In particular, $\pi_{\mathbf{v}}\circ\gamma_{\mathbf{v}}=\pi_1$. Let $l\geq 0$ be an integer and $\beta\in [0,1)$ so that $l+\beta\leq k+\alpha$. Given a map $\mathbf{f}\in C^{l,\beta}_{loc}(\Sigma; \mathbb{R}^M)$ we define the \emph{$\mathbf{v}$-extension of $\mathbf{f}$}  (to $U$) to be the map $\mathscr{E}_{\mathbf{v}}[\mathbf{f}]\in C^{l,\beta}_{loc}(U; \mathbb{R}^M)$ given by $\mathscr{E}_{\mathbf{v}}[\mathbf{f}]=\mathbf{f}\circ\pi_{\mathbf{v}}$.

\subsection{Homogeneous extensions around cones} \label{HomoExtSubsec}
Fix a cone, $\mathcal{C}\subset\mathbb{R}^{n+1}$, of class $C^{k,\alpha}$ with its link $\mathcal{L}$. Let $\hat{\mathcal{C}}$ be an open cone in $\mathbb{R}^{n+1}$ containing $\mathcal{C}$. Assume that $l\in [0,k]$ is an integer and $\beta\in [0,1)$ with $l+\beta\leq k+\alpha$. We wish to define a natural extension of an element of $C^{l,\beta}(\mathcal{L}; \mathbb{R}^M)$ to a homogeneous element of $C^{l,\beta}_{loc}(\hat{\mathcal{C}}; \mathbb{R}^M)$.

To that end, let $\mathbf{v}$ be a homogeneous (of degree $0$) transverse section on $\mathcal{C}$. As $\mathcal{L}$ is compact, we may suppose $\hat{\mathcal{C}}$ is $\mathbf{v}$-regular. Given a map $\varphi\in C^{l,\beta}(\mathcal{L}; \mathbb{R}^M)$ the \emph{$\mathbf{v}$-homogeneous extension of degree $d$ of $\varphi$}  (to $\hat{\mathcal{C}}$) is the map $\mathscr{E}_{\mathbf{v},d}^{\mathrm{H}}[\varphi]\in C_{loc}^{l,\beta}(\hat{\mathcal{C}}; \mathbb{R}^M)$ defined by 
$$
\mathscr{E}_{\mathbf{v},d}^{\mathrm{H}}[\varphi]=\mathscr{E}_{\mathbf{v}}\circ\mathscr{E}_{d}^{\mathrm{H}}[\varphi].
$$
Clearly, for all $\rho>0$,
$$
\mathscr{E}_{\mathbf{v},d}^{\mathrm{H}}[\varphi](\rho p)=\rho^d \mathscr{E}_{\mathbf{v},d}^{\mathrm{H}}[\varphi](p),
$$
and thus
$$
\nabla_{\mathbf{x}}\mathscr{E}_{\mathbf{v},d}^{\mathrm{H}}[\varphi]=d\mathscr{E}_{\mathbf{v},d}^{\mathrm{H}}[\varphi]
$$
where $\nabla$ is the flat connection on $\mathbb{R}^{n+1}$. In particular, if we denote by $\hat{\mathbf{v}}=\mathscr{E}_{\mathbf{v}}[\mathbf{v}]$, equivalently, $\hat{\mathbf{v}}=\mathscr{E}_{\mathbf{v},0}^{\mathrm{H}}[\mathbf{v}|_{\mathcal{L}}]$, one has $\nabla_{\mathbf{x}}\hat{\mathbf{v}}=\mathbf{0}$. Moreover, 
$$
\left\Vert\mathscr{E}_{\mathbf{v},d}^{\mathrm{H}}[\varphi]\right\Vert_{l, \beta; \hat{\mathcal{C}}_R}^{(d)} \leq C \Vert\varphi\Vert_{l, \beta}
$$
where $C$ depends only on $\hat{\mathcal{C}}, \mathbf{v}, l,\beta, d$ and $R$. 

\section{Asymptotically conical hypersurfaces} \label{ACHSec}
A $C^{k,\alpha}$-hypersurface, $\Sigma\subset\mathbb{R}^{n+1}$, is \emph{$C^{k,\alpha}_{*}$-asymptotically conical} if there is a $C^{k,\alpha}$-regular cone, $\mathcal{C}\subset\mathbb{R}^{n+1}$, and a homogeneous transverse section, $\mathbf{v}$, on $\mathcal{C}$ such that $\Sigma$ outside some compact set is given by the $\mathbf{v}$-graph of a function in $C^{k,\alpha}_1\cap C^{k}_{1,0}(\mathcal{C}_R)$ for some $R>1$.  Observe that by the Arzel\`{a}-Ascoli theorem one has that, for every $\beta\in [0, \alpha)$,
$$
\lim_{\rho\to 0^+} \rho\Sigma = \mathcal{C} \mbox{ in } C^{k, \beta}_{loc}(\mathbb{R}^{n+1}\setminus \{\mathbf{0}\}).
$$
Clearly, the asymptotic cone, $\mathcal{C}$, is uniquely determined by $\Sigma$ and so we denote it by $\mathcal{C}(\Sigma)$. Let $\mathcal{L}(\Sigma)$ denote the link of $\mathcal{C}(\Sigma)$ and, for $R>0$, let $\mathcal{C}_R(\Sigma)=\mathcal{C}(\Sigma)\setminus \bar{B}_R$. Denote the space of $C^{k,\alpha}_{*}$-asymptotically conical $C^{k,\alpha}$-hypersurfaces in $\mathbb{R}^{n+1}$ by $\mathcal{ACH}^{k,\alpha}_n$.

We claim the above definition is independent of the choice of transverse sections. To see this, let $\mathbf{w}$ be another homogeneous transverse section on $\mathcal{C}(\Sigma)$. Let $K$ be a compact set of $\Sigma$ that may be enlarged if needed. Denote by $\Sigma^\prime=\Sigma\setminus K$. Since $\pi_{\mathbf{v}}$ restricts to a $C^{k,\alpha}$ diffeomorphism of $\Sigma^\prime$ onto $\mathcal{C}_R(\Sigma)$, we denote its inverse by $\theta_{\mathbf{v}; \Sigma^\prime}$. We then write
$$
\pi_{\mathbf{w}}|_{\Sigma^\prime}=\left(\pi_{\mathbf{w}}\circ\theta_{\mathbf{v}; \Sigma^\prime}\right)\circ\pi_{\mathbf{v}}|_{\Sigma^\prime}.
$$
Notice that $\pi_{\mathbf{w}}\circ\theta_{\mathbf{v};\Sigma^\prime}$ restricts to a $C^{k,\alpha}$ diffeomorphism of $\mathcal{C}_R(\Sigma)$ onto the ends of $\mathcal{C}(\Sigma)$ and that
$$
\mathbf{x}|_{\mathcal{C}(\Sigma)}\circ\pi_{\mathbf{w}}\circ\theta_{\mathbf{v};\Sigma^\prime} \in C^{k,\alpha}_1 \cap C^k_{1,\mathrm{H}}(\mathcal{C}_R(\Sigma); \mathbb{R}^{n+1}).
$$
Moreover, 
\begin{equation} \label{TraceTransitEqn}
\mathrm{tr}_\infty^1[\mathbf{x}|_{\mathcal{C}(\Sigma)}\circ\pi_{\mathbf{w}}\circ\theta_{\mathbf{v};\Sigma^\prime}]=\mathbf{x}|_{\mathcal{L}(\Sigma)}.
\end{equation}
Thus, $\pi_{\mathbf{w}}$ restricts to a $C^{k,\alpha}$ diffeomorphism of $\Sigma^\prime$ onto the ends of $\mathcal{C}(\Sigma)$, 
$$
\mathbf{x}|_\Sigma\circ(\pi_{\mathbf{w}}|_{\Sigma^\prime})^{-1} \in C^{k,\alpha}_1\cap C^{k}_{1,\mathrm{H}}(\mathcal{C}_R(\Sigma); \mathbb{R}^{n+1}) \mbox{ and } \mathrm{tr}_\infty^1[\mathbf{x}|_\Sigma\circ(\pi_{\mathbf{w}}|_{\Sigma^\prime})^{-1}]=\mathbf{x}|_{\mathcal{L}(\Sigma)}.
$$
Hence $\Sigma^\prime$ can be written as the $\mathbf{w}$-graph of the function
$$
\left(\mathbf{x}|_\Sigma\circ(\pi_{\mathbf{w}}|_{\Sigma^\prime})^{-1}-\mathbf{x}|_{\mathcal{C}(\Sigma)}\right)\cdot\mathbf{w}\in C^{k,\alpha}_1\cap C^k_{1,0}(\mathcal{C}_R(\Sigma)).
$$

Let us now summarize some properties of weighted H\"{o}lder continuous functions on asymptotically conical hypersurfaces.

\begin{prop} \label{ListProp}
Fix $\Sigma\in \mathcal{ACH}_n^{k,\alpha}$ and $\Gamma\in\mathcal{ACH}^{k,\alpha}_m$. Let $l$ be an integer in $[0,k]$, $\beta\in [0,1)$ so that $l+\beta\leq k+\alpha$, and $d, d^\prime\in\mathbb{R}$. The following statements hold:
\begin{enumerate}
\item \label{ListProp1} For any $R\geq 1$ and $\rho_2>\rho_1>0$, the norm for functions in $C^{l,\beta}_d(\Sigma)$ defined by
$$
\Vert  f\Vert _{l,\beta; \Sigma\cap \bar{B}_R}+\sup_{\rho>R}\rho^{-d}\Vert f\circ\mathscr{D}_\rho\Vert_{l,\beta; (\rho^{-1} \Sigma)\cap (\bar{B}_{\rho_2}\setminus B_{\rho_1})}
$$
is equivalent to the norm $\Vert\cdot\Vert ^{(d)}_{l,\beta}$. Here $\mathscr{D}_\rho (x)=\rho x$ for $x\in\mathbb{R}^{n+1}$. 
\item \label{ListProp3} If $f\in C^{l,\beta}_d(\Sigma)$ and $\inf_\Sigma (|\mathbf{x}|+1)^{-d}|f|=\delta>0$, then $\frac{1}{f}\in C^{l,\beta}_{-d}(\Sigma)$ with its norm bounded by a constant depending only on $n,l,\beta,d,\delta$ and $\Vert f\Vert^{(d)}_{l,\beta}$.
\item \label{ListProp2} If $f\in C^{l, \beta}_d(\Sigma)$ and $g\in C^{l, \beta}_{d^\prime}(\Sigma)$, then $fg \in C^{l, \beta}_{d+d^\prime}(\Sigma)$ and
$$
\Vert fg\Vert _{l,\beta}^{(d+d^\prime)}\leq \nu_0 \Vert f \Vert _{l,\beta}^{(d)} \Vert g\Vert _{l,\beta}^{(d^\prime)},
$$
where $\nu_0$ depends only on $n,l,\beta,d$ and $d^\prime$.
\item \label{ListProp4} When $l\geq 1$, if $\mathbf{f}\in C^{l,\beta}_{1}(\Sigma; \mathbb{R}^{m+1})$ satisfies $\mathbf{f}(\Sigma)\subset\Gamma$ and there is a $\delta\in (0,1)$ and $R>1$ so that $|\mathbf{f}(p)|\geq \delta |\mathbf{x}(p)|$ on $\Sigma\setminus\bar{B}_R$, then, for all $g\in C^{l,\beta}_{d}(\Gamma)$,
$$
\Vert g\circ \mathbf{f}\Vert ^{(d)}_{l,\beta} \leq \nu_1 \Vert g\Vert _{l,\beta}^{(d)},
$$
where $\nu_1>1$ depends on $m,n,l,\beta,d,\delta,R$ and $\Vert\mathbf{f}\Vert^{(1)}_{l,\beta}$.
\item \label{ListProp5} When $l\geq 1$, if $g\in C^{l,\beta^\prime}_d(\Gamma)$ for some $\beta^\prime\in (\beta,1)$ with $l+\beta^\prime\leq k+\alpha$, and if, for $i\in\{1,2\}$, $\mathbf{f}_i\in C^{l,\beta}_{1}(\Sigma; \mathbb{R}^{m+1})$ satisfies $\mathbf{f}_i(\Sigma)\subset\Gamma$ and that there is a $\delta\in (0,1)$ and $R>1$ so that $|\mathbf{f}_i(p)|\geq\delta |\mathbf{x}(p)|$ on $\Sigma\setminus\bar{B}_R$, then
$$
\Vert g\circ \mathbf{f}_2-g\circ \mathbf{f}_1\Vert _{l,\beta}^{(d)}\leq \nu_2 \left(\Vert\mathbf{f}_2-\mathbf{f}_1\Vert _{l,\beta}^{(1)}\right)^{\beta^\prime-\beta},
$$
where $\nu_2>1$ depends only on $m,n,l,\beta,\beta^\prime, d,\delta,R,\Vert \mathbf{f}_1\Vert_{l,\beta}^{(1)},\Vert \mathbf{f}_2\Vert_{l,\beta}^{(1)}$ and $\Vert g\Vert_{l,\beta^\prime}^{(d)}$. 
\end{enumerate}
\end{prop}

\begin{rem}
The above continues to hold for $\Sigma_R=\Sigma\backslash \bar{B}_R$ where $\Sigma$ is either a $C^{k,\alpha}$-regular cone or in $\mathcal{ACH}^{k,\alpha}_n$. Where appropriate, the above also extend to $\mathbb{R}^M$-valued maps and to sections of tensor bundles.
\end{rem}

\subsection{Traces at infinity} \label{TraceSubsec}
Fix an element $\Sigma\in\mathcal{ACH}_n^{k,\alpha}$. Let $l$ be an integer in $[0,k]$ and $\beta\in [0,1)$ such that $l+\beta<k+\alpha$. A map $\mathbf{f}\in C_{loc}^{l,\beta}(\Sigma; \mathbb{R}^M)$ is \emph{asymptotically homogeneous of degree $d$} if $\mathbf{f}\circ\theta_{\mathbf{v};\Sigma^\prime}\in C^{l,\beta}_{d,\mathrm{H}}(\mathcal{C}_R(\Sigma); \mathbb{R}^M)$ , where $\mathbf{v}, \Sigma^\prime, \theta_{\mathbf{v};\Sigma^\prime}$ are referred in the previous discussions, and we define the \emph{trace at infinity of $\mathbf{f}$} to be
$$
\mathrm{tr}_\infty^d[\mathbf{f}]=\mathrm{tr}_\infty^d[\mathbf{f}\circ\theta_{\mathbf{v}; \Sigma^\prime}] \in C^{l,\beta}(\mathcal{L}(\Sigma); \mathbb{R}^M).
$$
In view of Item \eqref{ListProp4} of Proposition \ref{ListProp} and \eqref{TraceTransitEqn}, that whether $\mathbf{f}$ is asymptotically homogeneous of degree $d$ and the definition of $\mathrm{tr}_{\infty}^d$ are independent of the choice of homogeneous transverse sections on $\mathcal{C}(\Sigma)$. Clearly, $\mathbf{x}|_{\Sigma}$ is asymptotically homogeneous of degree $1$ and $\mathrm{tr}_\infty^{1}[\mathbf{x}|_\Sigma]=\mathbf{x}|_{\mathcal{L}(\Sigma)}$.

We next define the space
$$
C_{d,\mathrm{H}}^{l,\beta}(\Sigma; \mathbb{R}^M)=\left\{\mathbf{f}\in C_{d}^{l,\beta}(\Sigma; \mathbb{R}^M)\colon \mbox{$\mathbf{f}$ is asymptotically homogeneous of degree $d$}\right\}.
$$
One can check that $C_{d,\mathrm{H}}^{l,\beta}(\Sigma; \mathbb{R}^M)$ is a closed subspace of $C_d^{l,\beta}(\Sigma; \mathbb{R}^M)$, and the map 
$$
\mathrm{tr}_{\infty}^d\colon C_{d,\mathrm{H}}^{l,\beta}(\Sigma; \mathbb{R}^M)\to C^{l,\beta}(\mathcal{L}(\Sigma); \mathbb{R}^M)
$$
is a bounded linear map. We further define the set $C^{l,\beta}_{d,0}(\Sigma;\mathbb{R}^M)\subset C_{d,\mathrm{H}}^{l,\beta}(\Sigma; \mathbb{R}^M)$ to be the kernel of $\mathrm{tr}_\infty^d$.  As before this is a closed subspace.

\subsection{Asymptotically conical embeddings} \label{ACESubsec}
Fix an element $\Gamma\in \mathcal{ACH}^{k,\alpha}_n$. We define the space of $C^{k,\alpha}_{*}$-asymptotically conical embeddings of $\Gamma$ into $\mathbb{R}^{n+1}$ to be
$$
\mathcal{ACH}^{k,\alpha}_n(\Gamma)=\left\{\mathbf{f}\in C_{1}^{k,\alpha}\cap C^k_{1,\mathrm{H}}(\Gamma; \mathbb{R}^{n+1})\colon\mbox{$\mathbf{f}$ and $\mathscr{E}_{1}^{\mathrm{H}}\circ\mathrm{tr}_\infty^1[\mathbf{f}]$ are embeddings}\right\}.
$$
Clearly, $\mathcal{ACH}^{k,\alpha}_n(\Gamma)$ is an open set of the Banach space $C^{k,\alpha}_{1}\cap C^{k}_{1,\mathrm{H}}(\Gamma; \mathbb{R}^{n+1})$ with the $\Vert\cdot \Vert_{k,\alpha}^{(1)}$ norm. By \eqref{ClosedCondEqn} and the hypotheses on $\mathbf{f}$, $\mathrm{tr}_\infty^1[\mathbf{f}]\in C^{k,\alpha}(\mathcal{L}(\Gamma);\mathbb{R}^{n+1})$ and so
$$
\mathcal{C}[\mathbf{f}]=\mathscr{E}_{1}^{\mathrm{H}}\circ\mathrm{tr}_\infty^1[\mathbf{f}]\colon \mathcal{C}(\Gamma)\to \mathbb{R}^{n+1}\setminus\{\mathbf{0}\}
$$
is a $C^{k,\alpha}$ embedding.  As this map is homogeneous of degree $1$, it parameterizes a $C^{k,\alpha}$-regular cone.

%Given a $\Gamma\in\mathcal{ACH}^{k,\alpha}_n$, $\Gamma\setminus B_R$ for some $R$ sufficiently large is $C^{k,\alpha}$ diffeomorphic to $\mathcal{L}(\Gamma)\times [0,1)$. Thus, we may add $\mathcal{L}(\Gamma)$ to $\Gamma$ at infinity so to make a compact $C^{k,\alpha}$ manifold with boundary, denoted by $\bar{\Gamma}$ and called an \emph{end compactification of $\Gamma$}. Every element of $\mathcal{ACH}_n^{k,\alpha}(\Gamma)$ parameterizes an element of $\mathcal{ACH}_n^{k,\alpha}$. Conversely, any element of $\mathcal{ACH}_n^{k,\alpha}$ that has an end compactification diffeomorphic to $\bar{\Gamma}$ can be parameterized by some element of $\mathcal{ACH}_n^{k,\alpha}(\Gamma)$. 

\begin{prop} \label{AsympConicalProp}
Given $\Gamma\in\mathcal{ACH}^{k,\alpha}_n$ and $\mathbf{f}\in\mathcal{ACH}^{k,\alpha}_n(\Gamma)$, let $\Sigma=\mathbf{f}(\Gamma)$. The following statements hold:
\begin{enumerate}
\item \label{ACHItem} $\Sigma\in\mathcal{ACH}_n^{k,\alpha}$ and $\mathcal{C}(\Sigma)=\mathcal{C}[\mathbf{f}](\mathcal{C}(\Gamma))$. 
\item \label{ACInverseItem} The inverse, denoted by $\mathbf{f}^{-1}$, of $\mathbf{f}$ restricted the range to its image is an element of $\mathcal{ACH}^{k,\alpha}_n(\Sigma)$ and $\mathcal{C}[\mathbf{f}^{-1}]=\mathcal{C}[\mathbf{f}]^{-1}$.
\item \label{ACHCCOMPItem} For $\mathbf{g}\in\mathcal{ACH}_n^{k,\alpha}(\Sigma)$, $\mathbf{g}\circ \mathbf{f}\in\mathcal{ACH}_n^{k,\alpha}(\Gamma)$ and $\mathcal{C}[\mathbf{g}\circ\mathbf{f}]=\mathcal{C}[\mathbf{g}]\circ\mathcal{C}[\mathbf{f}]$. 
%\item \label{ACEItem} If $\bar{\Sigma}$ is $C^{k,\alpha}$ diffeomorphic to $\bar{\Gamma}$, there is a $\mathbf{g}\in\mathcal{ACH}^{k,\alpha}_n(\Gamma)$ such that $\Sigma=\mathbf{g}(\Gamma)$. 
\end{enumerate}
\end{prop}

\begin{proof}
As $\Gamma\in\mathcal{ACH}^{k,\alpha}_n$, there is a compact set $K\subset\Gamma$ so that $\Gamma^\prime=\Gamma\setminus K$ is the $\mathbf{v}$-graph of a function in  $C^{k,\alpha}_1\cap C^{k}_{1,0}(\mathcal{C}_R(\Gamma))$ for some $R>1$, where $\mathbf{v}$ is a homogeneous transverse section on $\mathcal{C}(\Gamma)$. Thus, $\pi_{\mathbf{v}}$ restricts to a $C^{k,\alpha}$ diffeomorphism of $\Gamma^\prime$ onto $\mathcal{C}_R(\Gamma)$, and we denote its inverse by $\theta_{\mathbf{v}; \Gamma^\prime}$.  

As $\mathbf{f}\in\mathcal{ACH}^{k,\alpha}_n(\Gamma)$, it follows that
$$
\mathbf{f}\circ\theta_{\mathbf{v}; \Gamma^\prime}\in C^{k,\alpha}_1\cap C^{k}_{1,\mathrm{H}}(\mathcal{C}_R(\Gamma); \mathbb{R}^{n+1}),
$$
and so, as noted above, $\mathrm{tr}_\infty^1[\mathbf{f}]\in C^{k,\alpha}(\mathcal{L}(\Gamma); \mathbb{R}^{n+1})$ and $\mathcal{C}[\mathbf{f}]$ parametrizes a $C^{k,\alpha}$-regular cone, denoted by $\mathcal{C}$. Observe that 
$$
\lim_{\rho\to 0^+} \rho \mathbf{f}\circ\theta_{\mathbf{v}; \Gamma^\prime}(\rho^{-1} \cdot)=\mathcal{C}[\mathbf{f}] \mbox{ in $C^k_{loc}(\mathcal{C}(\Gamma); \mathbb{R}^{n+1})$}.
$$
Let $\mathcal{L}$ be the link of $\mathcal{C}$ and $\mathbf{w}$ a homogeneous transverse section on $\mathcal{C}$. Thus, $\pi_{\mathbf{w}}$ restricts to a $C^{k,\alpha}$ diffeomorphism of $\Sigma^\prime=\Sigma\setminus K^\prime$ onto $\mathcal{C}_{R^\prime}$ for some compact set $K^\prime\subset\Sigma$ and $R^\prime>1$. Let us denote its inverse by $\theta_{\mathbf{w};\Sigma^\prime}$. Moreover,
$$
\mathbf{x}|_{\Sigma}\circ \theta_{\mathbf{w};\Sigma^\prime}\in C^{k,\alpha}_1\cap C^{k}_{1,\mathrm{H}}(\mathcal{C}_{R^\prime}; \mathbb{R}^{n+1}),
$$
and its trace at infinity is $\mathbf{x}|_{\mathcal{L}}$. Hence, it follows that $\Sigma\in\mathcal{ACH}^{k,\alpha}_n$ and our construction ensures $\mathcal{C}(\Sigma)=\mathcal{C}[\mathbf{f}](\mathcal{C}(\Gamma))$, proving Item \eqref{ACHItem}.

By our hypotheses on $\Gamma$ and $\mathbf{f}$, both $\mathbf{f}$ and its differential at any point of $\Gamma$ are injective, and, for some $\delta>0$, $|\mathbf{f}(p)|>\delta |\mathbf{x}(p)|$ whenever $|\mathbf{x}(p)|$ is sufficiently large. Thus, by the inverse function theorem, $\mathbf{f}$ restricted the range to its image is invertible and its inverse $\mathbf{f}^{-1}\in C^{k,\alpha}_{1}(\Sigma;\mathbb{R}^{n+1})$. Moreover,
\begin{align*}
\mathcal{C}[\mathbf{f}^{-1}] & =\lim_{\rho\to 0^+} \rho \mathbf{f}^{-1}\circ\theta_{\mathbf{w};\Sigma^\prime}(\rho^{-1}\cdot)=\lim_{\rho\to 0^+} \rho\mathbf{f}^{-1}\circ\theta_{\mathbf{w};\Sigma^\prime}(\rho^{-1}\cdot)\circ\mathcal{C}[\mathbf{f}]\circ\mathcal{C}[\mathbf{f}]^{-1} \\
& =\lim_{\rho\to 0^+} \rho\mathbf{f}^{-1}\circ\theta_{\mathbf{w};\Sigma^\prime}(\rho^{-1}\pi_{\mathbf{w}}(\rho\mathbf{f}\circ\theta_{\mathbf{v};\Gamma^\prime}(\rho^{-1}\cdot)))\circ\mathcal{C}[\mathbf{f}]^{-1} \\
& =\lim_{\rho\to 0^+} \rho\theta_{\mathbf{v};\Gamma^\prime}(\rho^{-1}\cdot)\circ \mathcal{C}[\mathbf{f}]^{-1}=\mathcal{C}[\mathbf{f}]^{-1}.
\end{align*}
Here the convergence is in $C^k_{loc}(\mathcal{C}(\Sigma))$. Thus we have shown $\mathbf{f}^{-1}\in \mathcal{ACH}^{k,\alpha}_n(\Sigma)$, proving Item \eqref{ACInverseItem}.

As $k\geq 1$, it follows from Item \eqref{ListProp4} of Proposition \ref{ListProp} that $\mathbf{g}\circ \mathbf{f}\in C^{k,\alpha}_{1}(\Gamma)$. The hypothesis that each $\mathbf{f}$ and $\mathbf{g}$ is an embedding implies that $\mathbf{g}\circ \mathbf{f}$ is an embedding. Moreover,
\begin{align*}
\mathcal{C}[\mathbf{g}\circ \mathbf{f}]& =\lim_{\rho\to 0^+} \rho (\mathbf{g}\circ\mathbf{f})\circ\theta_{\mathbf{v}; \Gamma^\prime}(\rho^{-1}\cdot)= \lim_{\rho\to 0^+} \rho\mathbf{g}\circ\theta_{\mathbf{w}; \Sigma^\prime}(\rho^{-1} \pi_{\mathbf{w}}(\rho\mathbf{f}\circ\theta_{\mathbf{v};\Gamma^\prime}(\rho^{-1} \cdot)))\\
&=\left(\lim_{\rho\to 0^+}  \rho\mathbf{g}(\theta_{\mathbf{w}; \Sigma^\prime}(\rho^{-1}\cdot))\right)\circ \pi_{\mathbf{w}}\circ\left( \lim_{\rho\to 0^+} \rho \mathbf{f}(\theta_{\mathbf{v};\Gamma^\prime}(\rho^{-1} \cdot))\right)\\
&=\mathcal{C}[\mathbf{g}]\circ\pi_{\mathbf{w}}\circ \mathcal{C}[\mathbf{f}]=\mathcal{C}[\mathbf{g}]\circ \mathcal{C}[\mathbf{f}]. 
\end{align*}
Here the convergence is in $C^{k}_{loc}(\mathcal{C}(\Gamma))$. Hence, $\mathcal{C}[\mathbf{g}\circ \mathbf{f}]$ is the composition of embeddings and so is also an embedding. That is, $\mathbf{g}\circ \mathbf{f}\in \mathcal{ACH}^{k,\alpha}_n(\Gamma)$ and we have proved Item \eqref{ACHCCOMPItem}.
\end{proof}

Finally, we introduce a natural equivalence relation on $\mathcal{ACH}_{n}^{k,\alpha}(\Gamma)$. First, say a $C^{k,\alpha}$ diffeomorphism $\phi:\Gamma\to \Gamma$ \emph{fixes infinity} if $\mathbf{x}|_{\Gamma}\circ \phi\in \mathcal{ACH}^{k,\alpha}_n(\Gamma)$ and
$$
\mathrm{tr}^1_{\infty}[\mathbf{x}|_{\Gamma}\circ \phi]=\mathbf{x}|_{\mathcal{L}(\Gamma)}.
$$
%Observe, that such a $\phi$ naturally extends to a $C^{k,\alpha}$ diffeomorphism from $\bar{\Gamma}$ to $\bar{\Gamma}$ that restricts to the identity map on $\partial \Gamma$. 
Two elements $\mathbf{f}, \mathbf{g}\in\mathcal{ACH}_{n}^{k,\alpha}(\Gamma)$ are equivalent, written $\mathbf{f}\sim \mathbf{g}$, provided there is a $C^{k,\alpha}$ diffeomorphism $\phi:\Gamma\to\Gamma$ that fixes infinity so that $\mathbf{f}\circ\phi=\mathbf{g}$. We observe that, by Items \eqref{ACInverseItem} and \eqref{ACHCCOMPItem} of Proposition \ref{AsympConicalProp},  $\mathbf{f}\sim\mathbf{g}$ if and only if
$$
\mathbf{f}(\Gamma)=\mathbf{g}(\Gamma) \mbox{ and }\mathcal{C}[\mathbf{f}]=\mathcal{C}[\mathbf{g}].
$$

\section{First and second variations of the $E$-functional} \label{VariationSec}
Given a $C^2$-hypersurface $\Sigma\subset\mathbb{R}^{n+1}$ an $M$-parameter differentiable family of $C^2$ embeddings, $\Psi_{\mathbf{s}}\colon\Sigma\to\mathbb{R}^{n+1}$ for $\mathbf{s}\in (-1,1)^M$, is a \emph{compactly supported variation of $\mathbf{x}|_\Sigma$} if $\Psi_{\mathbf{0}}=\mathbf{x}|_\Sigma$ and for some compact set $K\subset\Sigma$ each $\Psi_{\mathbf{s}}|_{\Sigma\setminus K}=\mathbf{x}|_{\Sigma\setminus K}$.

By a direct computation (see, for instance, \cite[Section 9]{Simon}) and integration by parts, we obtain the first variation formula of $E$.

\begin{prop} \label{1stVarProp}
Let $\Sigma$ be a $C^2$-hypersurface in $\mathbb{R}^{n+1}$, and let $\{\Psi_s\}_{s\in (-1,1)}$ be a compactly supported variation of $\mathbf{x}|_\Sigma$. If ${\frac{\partial\Psi_s}{\partial s}\vline}_{s=0}=\mathbf{V}$, then 
\begin{equation} \label{1stVarEqn}
{\frac{d}{ds}\vline}_{s=0} E[\Psi_s(\Sigma)]=-\int_\Sigma \mathbf{V}\cdot \left(\mathbf{H}_\Sigma-\frac{\mathbf{x}^\perp}{2}\right)e^{\frac{|\mathbf{x}|^2}{4}} \, d\mathcal{H}^n.
\end{equation}
\end{prop}

Next we compute the second variation of $E$ at its critical points, i.e., self-expanders.

\begin{prop} \label{2ndVarProp}
Let $\Sigma$ be a self-expander in $\mathbb{R}^{n+1}$, and let $\{\Psi_{s,t}\}_{-1<s,t<1}$ be a compactly supported variation of $\mathbf{x}|_\Sigma$. If ${\frac{\partial\Psi_{s,0}}{\partial s}\vline}_{s=0}=\mathbf{V}$ and ${\frac{\partial\Psi_{0,t}}{\partial t}\vline}_{t=0}=\mathbf{W}$, then
\begin{equation} \label{2ndVarEqn}
{\frac{\partial^2}{\partial t\partial s}\vline}_{s=t=0} E[\Psi_{s,t}(\Sigma)]=-\int_\Sigma \left(\mathbf{V}\cdot\mathbf{n}_\Sigma\right) L_\Sigma\left(\mathbf{W}\cdot\mathbf{n}_\Sigma\right)e^{\frac{|\mathbf{x}|^2}{4}} \, d\mathcal{H}^n,
\end{equation}
where 
$$
L_\Sigma=\Delta_\Sigma+\frac{1}{2}\mathbf{x}\cdot\nabla_\Sigma+|A_\Sigma|^2-\frac{1}{2}.
$$
\end{prop}

\begin{rem}
If $\mathbf{V}=f\mathbf{v}$ and $\mathbf{W}=g\mathbf{v}$, where $\mathbf{v}$ is a transverse section on $\Sigma$, then
$$
{\frac{\partial^2}{\partial t\partial s}\vline}_{s=t=0} E[\Psi_{s,t}(\Sigma)]=-\int_\Sigma f (L_{\mathbf{v}}g) e^{\frac{|\mathbf{x}|^2}{4}} \, d\mathcal{H}^n,
$$
where $L_{\mathbf{v}}g=(\mathbf{v}\cdot\mathbf{n}_\Sigma)L_\Sigma((\mathbf{v}\cdot\mathbf{n}_\Sigma)g)$ and $L_{\mathbf{v}}$ is called the \emph{$\mathbf{v}$-Jacobi operator}.
\end{rem}

\begin{proof}
Denote by $\Sigma_{s,t}=\Psi_{s,t}(\Sigma)$. A straightforward computation similar to \cite[Section 9]{Simon} gives that
\begin{align*}
& {\frac{\partial^2}{\partial t\partial s}\vline}_{s=t=0} E[\Sigma_{s,t}]={\frac{\partial}{\partial t}\vline}_{t=0}\int_{\Sigma_{0,t}} -\left({\frac{\partial\Psi_{s,t}}{\partial s}\vline}_{s=0}\right)\cdot \left(\mathbf{H}_{\Sigma_{0,t}}-\frac{\mathbf{x}^\perp}{2}\right)e^{\frac{|\mathbf{x}|^2}{4}} d\mathcal{H}^n \\
&=  \int_\Sigma \left(\mathrm{div}_\Sigma\mathbf{V}\cdot\mathrm{div}_\Sigma\mathbf{W}+\sum_{i=1}^n(\nabla_{\tau_i}\mathbf{V})^\perp\cdot(\nabla_{\tau_i}\mathbf{W})^\perp -\sum_{i,j=1}^n (\tau_i\cdot \nabla_{\tau_j}\mathbf{V})(\tau_j\cdot\nabla_{\tau_i}\mathbf{W}) \right. \\
 & \left. +\frac{1}{2}\left(\mathbf{x}\cdot\mathbf{V}\right)\mathrm{div}_\Sigma\mathbf{W}+\frac{1}{2}\left(\mathbf{x}\cdot\mathbf{W}\right)\mathrm{div}_\Sigma\mathbf{V}+\frac{1}{4}\left(\mathbf{x}\cdot\mathbf{V}\right)\left(\mathbf{x}\cdot\mathbf{W}\right)+\frac{1}{2}\mathbf{V}\cdot\mathbf{W}\right) e^{\frac{|\mathbf{x}|^2}{4}} \, d\mathcal{H}^n,
\end{align*}
where $\{\tau_1,\dots,\tau_n\}$ is an orthonormal frame on $\Sigma$ and $\nabla$ is the flat connection on $\mathbb{R}^{n+1}$. Thus, ${\frac{\partial^2}{\partial s\partial t}\vline}_{s=t=0}E[\Sigma_{s,t}]$ is a symmetric bilinear form, denoted by $Q_\Sigma(\mathbf{V},\mathbf{W})$. We write
$$
Q_\Sigma\left(\mathbf{V},\mathbf{W}\right)=Q_{\Sigma}\left(\mathbf{V}^\perp,\mathbf{W}^\perp\right)+Q_{\Sigma}\left(\mathbf{V}^\top,\mathbf{W}^\perp\right)+Q_{\Sigma}\left(\mathbf{V}^\perp,\mathbf{W}^\top\right)+Q_{\Sigma}\left(\mathbf{V}^\top,\mathbf{W}^\top\right).
$$
We claim that the last three terms on the right hand side are equal to $0$. Therefore, \eqref{2ndVarEqn} follows easily from integration by parts.

To prove the claim, let $\{\phi_t\}_{t\in\mathbb{R}}$ be the one-parameter group of diffeomorphisms of $\Sigma$ generated by $\mathbf{V}^\top$. Given $|s|$ small and $t\in\mathbb{R}$ we define: for $p\in\Sigma$,
$$
\tilde{\Psi}_{s,t}(p)=\gamma_{\mathbf{W}^\perp}(\phi_t(p),s)=\mathbf{x}(\phi_t(p))+s\mathbf{W}^\perp(\phi_t(p)).
$$
Clearly, $\{\tilde{\Psi}_{s,t}\}$ is a compactly supported variation of $\mathbf{x}|_\Sigma$ with ${\frac{\partial\tilde{\Psi}_{s,0}}{\partial s}\vline}_{s=0}=\mathbf{W}^\perp$ and ${\frac{\partial\tilde{\Psi}_{0,t}}{\partial t}\vline}_{t=0}=\mathbf{V}^\top$. Thus, $\tilde{\Psi}_{0,t}(\Sigma)=\mathbf{x}|_\Sigma\circ\phi_t(\Sigma)=\Sigma$. By the computations in the preceding paragraph, 
$$
Q_{\Sigma}\left(\mathbf{V}^\top,\mathbf{W}^\perp\right)={\frac{\partial^2}{\partial t\partial s}\vline}_{s=t=0} E[\tilde{\Psi}_{s,t}(\Sigma)]=0.
$$
By symmetries, $Q_\Sigma(\mathbf{V}^\perp,\mathbf{W}^\top)=0$. Finally, let $\{\psi_s\}_{s\in\mathbb{R}}$ be the one-parameter group of diffeomorphisms of $\Sigma$ generated by $\mathbf{W}^\top$. Given $s,t\in\mathbb{R}$ define $\hat{\Psi}_{s,t}=\mathbf{x}|_\Sigma\circ\psi_s\circ\phi_t$. Then $\{\hat{\Psi}_{s,t}\}$ is a compactly supported variation of $\mathbf{x}|_\Sigma$ with ${\frac{\partial\hat{\Psi}_{0,t}}{\partial t}\vline}_{t=0}=\mathbf{V}^\top$ and ${\frac{\partial\hat{\Psi}_{s,0}}{\partial s}\vline}_{s=0}=\mathbf{W}^\top$. As $E[\hat{\Psi}_{s,t}(\Sigma)]-E[\Sigma]=0$, $Q_\Sigma(\mathbf{V}^\top,\mathbf{W}^\top)=0$, proving the claim. 
\end{proof}

\section{Fredholm property of $\mathbf{v}$-Jacobi operator} \label{FredholmSec}
Throughout this section we fix a self-expander $\Sigma\in\mathcal{ACH}_n^{k,\alpha}$ and a transverse section on $\Sigma$, $\mathbf{v}\in C^{k,\alpha}_{0}\cap C^{k}_{0,\mathrm{H}}(\Sigma;\mathbb{R}^{n+1})$. In addition, we assume that $\mathcal{C}[\mathbf{v}]=\mathscr{E}^{\mathrm{H}}_1\circ\mathrm{tr}^1_\infty[\mathbf{v}]$ is a transverse section on $\mathcal{C}(\Sigma)$ and
\begin{equation} \label{RegVectorEqn}
\left(\mathscr{L}_\Sigma+\frac{1}{2}\right)\mathbf{v}=\Delta_\Sigma\mathbf{v}+\frac{1}{2}\mathbf{x}\cdot\nabla_\Sigma\mathbf{v}\in C^{k-2,\alpha}_{-2}(\Sigma;\mathbb{R}^{n+1}).
\end{equation} 
Observe that such $\mathbf{v}$ always exists: for instance, one may take $\mathbf{v}$ to be a transverse section on $\Sigma$ that, outside a compact set, equals $\mathscr{E}_{\mathbf{w}}[\mathbf{w}]\circ\mathbf{x}|_\Sigma$ for $\mathbf{w}$ a homogeneous transverse section on $\mathcal{C}(\Sigma)$, as one computes that, outside this compact set,
$$
\mathbf{x}\cdot\nabla_\Sigma\mathbf{v}=\nabla_{\mathbf{x}}\mathscr{E}_\mathbf{w}[\mathbf{w}]\circ\mathbf{x}|_\Sigma-\nabla_{\mathbf{x}^\perp}\mathscr{E}_\mathbf{w}[\mathbf{w}]\circ\mathbf{x}|_\Sigma=-2\nabla_{\mathbf{H}_\Sigma}\mathscr{E}_\mathbf{w}[\mathbf{w}]\circ\mathbf{x}|_\Sigma,
$$
where $\nabla$ is the flat connection on $\mathbb{R}^{n+1}$ and the second equality uses the homogeneity of $\mathscr{E}_\mathbf{w}[\mathbf{w}]$ and the self-expander equation. 

For an integer $l\ge 2$ and $\beta\in [0,1)$ with $l+\beta\leq k+\alpha$ we define the Banach space 
$$
\mathcal{D}^{l,\beta}(\Sigma)=\left\{f\in C^{l,\beta}_{1}\cap C^{l-1,\beta}_0\cap C^{l-2,\beta}_{-1}(\Sigma) \colon \mathbf{x}\cdot\nabla_\Sigma f\in C^{l-2,\beta}_{-1}(\Sigma) \right\}
$$
equipped with the norm
$$
\Vert f\Vert_{l,\beta}^*=\Vert f \Vert_{l-2,\beta}^{(-1)}+\sum_{i=l-1}^l \Vert \nabla_\Sigma^i f\Vert_{\beta}^{(1-l)}+\Vert\mathbf{x}\cdot\nabla_\Sigma f\Vert_{l-2,\beta}^{(-1)}.
$$

The goal of this section is to prove the Fredholm property of $L_{\mathbf{v}}$. First, we show the existence and uniqueness of solutions in $\mathcal{D}^{k,\alpha}(\Sigma)$ of $\mathscr{L}_\Sigma u=f$ for any given $f\in C^{k-2,\alpha}_{-1}(\Sigma)$. A similar problem has also been investigated by Deruelle \cite{Deruelle} in the study of asymptotically conical expanding Ricci solitons. However, the method we employ here is independent of \cite{Deruelle} (as well as \cite{Lunardi}) in part because the weighted function spaces considered in \cite[Section 2]{Deruelle} are different from ours and so the arguments there do not apply to our setting. Next we show that $L_{\mathbf{v}}-(\mathbf{v}\cdot\mathbf{n}_\Sigma)^2\mathscr{L}_\Sigma$ is a compact operator from $\mathcal{D}^{k,\alpha}(\Sigma)$ to $C^{k-2,\alpha}_{-1}(\Sigma)$. Thus it follows that $L_{\mathbf{v}}$ is Fredholm and of index $0$.

\subsection{$\mathcal{D}^{k,\alpha}$-solutions to $\mathscr{L}_\Sigma u=f$} \label{ModelSubsec}

\begin{lem} \label{C0EstLem}
Let $\Omega$ be a bounded domain in $\Sigma$ with smooth boundary. Given $f\in C^0(\Omega)$, if $u\in C^0(\bar{\Omega})\cap C^2(\Omega)$ satisfies 
$\mathscr{L}_\Sigma u=f$, then 
$$
\max_{p\in\bar{\Omega}} |u(p)| \le 2\sup_{p\in\Omega} 
|f(p)|+\max_{p\in\partial\Omega} |u(p)|.
$$
\end{lem}

\begin{proof}
Suppose that $u$ attains its maximum at an interior point $p_0\in\Omega$. Then $\nabla_\Sigma u(p_0)=0$ and $\Delta_\Sigma u(p_0)\le 0$. Thus,
$$
f(p_0)=\mathscr{L}_\Sigma u(p_0)\le -\frac{1}{2} u(p_0),
$$
implying $u(p_0)\le -2f(p_0)$. Hence we have
$$
\max_{p\in\bar{\Omega}} u(p) \le 2\sup_{p\in\Omega} -f(p)+\max_{p\in\partial\Omega} u(p).
$$
A similar argument gives
$$
\max_{p\in\bar{\Omega}} -u(p) \le 2\sup_{p\in\Omega} f(p)+\max_{p\in\partial \Omega} -u(p).
$$
The lemma follows from combining these two estimates.
\end{proof}

\begin{lem} \label{BarrierLem}
Given $d\in\mathbb{R}$ there is an $R_0=R_0(\Sigma,n,d)>1$ so that on $\Sigma\setminus \bar{B}_{R_0}$,
$$
\mathscr{L}_\Sigma r^d<\frac{d}{2} r^d
$$
where $r(p)=|\mathbf{x}(p)|$ for $p\in\Sigma$.
\end{lem}

\begin{proof} 
On $\Sigma\setminus B_1$ we compute
$$
\Delta_\Sigma r^d=\Delta r^d-\nabla^2 r^d(\mathbf{n}_\Sigma,\mathbf{n}_\Sigma)+\mathbf{H}_\Sigma\cdot\nabla r^d=O(r^{d-2}),
$$
where we used the hypothesis that $\Sigma\in\mathcal{ACH}^{k,\alpha}_n$ in the last equality. Similarly,
$$
\mathbf{x}\cdot\nabla_\Sigma r^d=d\left(1-|\mathbf{x}^\perp|^2r^{-2}\right)r^{d}=\left(d+o(1)\right)r^d.
$$
It follows from these two identities that 
$$
\mathscr{L}_\Sigma r^d=\frac{1}{2}\left(d-1+o(1)\right) r^d,
$$
proving the claim.
\end{proof}

\begin{prop} \label{C0DecayProp}
Given $f\in C^{k-2,\alpha}_{loc}\cap C^{0}_{-1}(\Sigma)$ there is a unique solution in $C^{k,\alpha}_{loc}\cap C^0_{-1}(\Sigma)$ of $\mathscr{L}_\Sigma u=f$. Moreover, there is a constant $C_0=C_0(\Sigma,n)$ so that
$$
\Vert u\Vert_{0}^{(-1)} \leq C_0 \Vert f\Vert_{0}^{(-1)}.
$$
\end{prop}

\begin{proof}
By \cite[Theorem 6.14]{GilbargTrudinger} the Dirichlet problem
$$
\left\{
\begin{array}{cc}
\mathscr{L}_\Sigma u_R=f & \mbox{in $\Sigma\cap B_R$} \\
u_R=0 & \mbox{on $\Sigma\cap\partial B_R$}
\end{array}
\right.
$$
has a unique solution in $C^{k,\alpha}(\Sigma\cap\bar{B}_R)$. Moreover, by Lemma \ref{C0EstLem}, 
\begin{equation} \label{C0ApproxEqn}
\Vert u_R\Vert_{0}\leq 2\Vert f\Vert_{0} \leq 2\Vert f\Vert_{0}^{(-1)}.
\end{equation}
Thus, setting $c=4R_0\Vert f\Vert_{0}^{(-1)}$, for all $R>R_0(\Sigma,n,-1)$, 
$$
2|f|<cr^{-1} \mbox{ in $\Sigma\cap (B_R\setminus\bar{B}_{R_0})$ and } |u_R|<cr^{-1} \mbox{ on $\Sigma\cap\partial B_{R_0}$}.
$$
Thus, by Lemma \ref{BarrierLem} and the elliptic maximum principle, 
\begin{equation} \label{C0DecayApproxEqn}
|u_R|<cr^{-1} \mbox{ in $\Sigma\cap (\bar{B}_R\setminus B_{R_0})$}.
\end{equation}
Hence, combining \eqref{C0ApproxEqn} and \eqref{C0DecayApproxEqn} gives
\begin{equation} \label{WeightC0ApproxEqn}
\sup_{\Sigma\cap \bar{B}_{R}} \left(1+r\right) |u_R| \leq C(R_0)\Vert f\Vert_{0}^{(-1)}.
\end{equation}

Therefore, invoking \eqref{WeightC0ApproxEqn} and the elliptic Schauder estimates \cite[Theorem 6.2]{GilbargTrudinger}, passing $R\to\infty$, it follows from the Arzel\`{a}-Ascoli theorem that there is a solution to $\mathscr{L}_\Sigma u=f$ in $C^{k,\alpha}_{loc}(\Sigma)$ that satisfies
$$
\Vert u\Vert_{0}^{(-1)} \leq C\Vert f\Vert_{0}^{(-1)}.
$$
The uniqueness of such solutions is implied by the elliptic maximum principle.
\end{proof}

\begin{lem} \label{DerivativeDecayLem}
Given $f\in C^{k,\alpha}_{loc}\cap C^{\alpha}_{-1}(\Sigma)$ let $u$ be the solution given in Proposition \ref{C0DecayProp}. There exist constants $C_i$, $i=1,2$, depending only on $\Sigma$, $n$, and $\alpha$, so that 
$$
\Vert\nabla^i_\Sigma u\Vert_{0}^{(-1)} \leq C_i \Vert f\Vert_{\alpha}^{(-1)}.
$$
\end{lem}

\begin{proof}
Let $\Sigma_t=\sqrt{t}\, \Sigma$ for $t>0$. As $\Sigma$ is a self-expander, $\{\Sigma_t\}_{t>0}$ is a MCF. As $\Sigma\in\mathcal{ACH}^{k,\alpha}_n$, it follows from the interior estimates for graphical MCF (cf. \cite[Theorem 3.4]{EHInterior}) that there is an $R=R(\Sigma)$ sufficiently large so that for any $p\in\Sigma\setminus B_R$ and $t\in \left[\frac{1}{2},1\right]$, each $\Sigma_t\cap B_2^{n+1}(p)$ can be parametrized by the map $\Psi_{p}(\cdot, t)\colon \Omega_{p}(t)\subset T_{p}\Sigma\to\mathbb{R}^{n+1}$ given by 
$$
\Psi_{p}(x,t)=\mathbf{x}(p)+\mathbf{x}(x)+\psi_{p}(x,t)\mathbf{n}_\Sigma(p),
$$ 
satisfying that $\psi_{p}(\mathbf{0},1)=0$ and 
\begin{equation} \label{GraphEstEqn}
\sum_{i=0}^3 |\nabla^i\psi_{p}|+\sum_{i=0}^2 |\partial_t\nabla^i\psi_{p}|\leq 10^{-1},
\end{equation}
where $\nabla$ is the flat connection on $T_p\Sigma$. Clearly, $B_1^n\subset\Omega_{p}(t)$.

Fix any $p\in\Sigma\setminus B_R$. We next define a function $v_{p}$ on $Q=B_1^n\times \left[\frac{1}{2},1\right]$ by
$$
v_{p}(x,t)=t^{\frac{1}{2}} u(t^{-\frac{1}{2}} \Psi_{p}(x,t)).
$$
As $\mathscr{L}_\Sigma u=f$, it follows from the chain rule that 
$$
\partial_t v_{p}-\Delta_t v_{p}-\partial_t \psi_{p}\mathbf{n}_\Sigma^\top(p) \cdot\nabla_t v_{p}=g_{p}
$$
where $\nabla_t$ and $\Delta_t$ are the gradient and Laplacian, respectively,  for the pull-back metric via $\Psi_{p}(\cdot, t)$,  $\mathbf{n}_\Sigma^\top(p)$ is the pull-back of the projection of $\mathbf{n}_\Sigma(p)$ to the appropriate tangent space of $\sqrt{t}\, \Sigma$ and 
$$
g_{p}(x,t)=-t^{-\frac{1}{2}} f(t^{-\frac{1}{2}}\Psi_{p}(x,t)).
$$

Next we will show that $g_{p}$ is H\"{o}lder continuous in space-time. More precisely, 
$$
\Vert g_{p}\Vert_{0;Q}+\sup_{(x,t),(y,s)\in Q}\frac{|g_{p}(x,t)-g_{p}(y,s)|}{|x-y|^\alpha+|t-s|^{\frac{\alpha}{2}}} \leq C\Vert f\Vert_{\alpha}^{(-1)}|\mathbf{x}(p)|^{-1}
$$
for some $C=C(\alpha)$. First, by \eqref{GraphEstEqn}, 
$$
\frac{1}{2}|\mathbf{x}(p)|<|\Psi_{p}(x,t)|<2|\mathbf{x}(p)|,
$$
and thus 
$$
\Vert g_{p}\Vert_{0;Q} \leq 2\Vert f\Vert_{0}^{(-1)}|\mathbf{x}(p)|^{-1}.
$$
Next, for $(x,t), (y,s)\in Q$, 
\begin{align*}
 |t^{-\frac{1}{2}}\Psi_{p}(x,t)-s^{-\frac{1}{2}}\Psi_{p}(y,s)|
& \leq |t^{-\frac{1}{2}}-s^{-\frac{1}{2}}| |\Psi_{p}(x,t)|+s^{-\frac{1}{2}} |\Psi_{p}(x,t)-\Psi_{p}(y,s)| \\
&\leq \left(4|\mathbf{x}(p)|+1\right) |s-t|+|x-y|,
\end{align*}
where we used \eqref{GraphEstEqn} in the last inequality. It follows that
\begin{align*}
|g_{p}(x,t)-g_{p}(y,s)| & \leq s^{-\frac{1}{2}} |f(t^{-\frac{1}{2}}\Psi_{p}(x,t))-f(s^{-\frac{1}{2}}\Psi_{p}(y,s))| \\
& \quad +|t^{-\frac{1}{2}}-s^{-\frac{1}{2}}| |f(t^{-\frac{1}{2}}\Psi_{p}(x,t))| \\
& \leq C\Vert f\Vert_{\alpha}^{(-1)} |\mathbf{x}(p)|^{-1}\left(|x-y|^\alpha+|s-t|^{\frac{\alpha}{2}}\right),
\end{align*}
proving the claim.

Finally, combining the hypothesis on $u$ and the preceding discussions, it follows from the parabolic Schauder estimates (cf. \cite[Chapter 4, Theorem 10.1]{LSU}) that
$$
\sum_{i=1}^2 |\nabla^i v_{p}(\mathbf{0},1)|\leq C^\prime(n,\alpha)\Vert f\Vert_{\alpha}^{(-1)} |\mathbf{x}(p)|^{-1}.
$$
This together with \eqref{GraphEstEqn} implies 
$$
\sum_{i=1}^2 |\nabla_\Sigma^i u(p)| \leq C^\prime \Vert f\Vert_{\alpha}^{(-1)} |\mathbf{x}(p)|^{-1}.
$$
As $p\in\Sigma\setminus B_R$ was arbitrary, it remains only to obtain the appropriate bounds in $\bar{B}_R\cap \Sigma$. Here one appeals to standard elliptic Schauder estimates for compact domains.
\end{proof}

\begin{lem} \label{CauchyProbLem}
Given $\beta\in (0,1)$ there is an $\epsilon_0=\epsilon_0(n,\beta)$ sufficiently small so that if $a^{ij}\in C^0((0,1); C^{\beta}(\mathbb{R}^n))$ for $1\leq i,j\leq n$ satisfy  
$$
\sup_{0<t<1}\sum_{i,j=1}^n\Vert a^{ij}(\cdot,t)-\delta_{ij}\Vert_{\beta} \leq \epsilon_0,
$$
then given $h\in C^0((0,1); C^{\beta}(\mathbb{R}^n))$ the Cauchy problem
\begin{equation} \label{CauchyProbEqn}
\left\{
\begin{array}{cc}
\partial_t w-\sum\limits_{i,j=1}^n a^{ij}\partial^2_{x_ix_j}w=h & \mbox{in $\mathbb{R}^n\times (0,1)$} \\
\displaystyle\lim_{(x^\prime,t)\to (x,0)} w(x^\prime,t)=0 & \mbox{for $x\in\mathbb{R}^n$}
\end{array}
\right.
\end{equation}
has a unique solution $w\in C^0((0,1); C^{2,\beta}(\mathbb{R}^n))$. Moreover, for some $C_3=C_3(n,\beta)$,
\begin{equation} \label{TimeDecayEqn}
\sup_{0<t<1}\sum_{i=0}^2 t^{\frac{i-2}{2}} \Vert\nabla^iw(\cdot, t)\Vert_{\beta} \leq C_3 \sup_{0<t<1} \Vert h(\cdot,t)\Vert_{\beta}.
\end{equation}
\end{lem}

\begin{rem} \label{CauchyProbRem}
By a change of variables, it is easy to see that Lemma \ref{CauchyProbLem} still holds true if we replace the identity matrix $(\delta_{ij})$ by any positive definite constant matrix. 
\end{rem}

\begin{proof}
When $a^{ij}=\delta_{ij}$, the lemma was known as a classical result for the heat equation on $\mathbb{R}^n$. For reader's convenience, we include a proof of that in Appendix \ref{EstHeatEqnApp}. For general $(a^{ij})$ that are small perturbations of the identity matrix, we appeal to the method of continuity to establish the existence of solutions; see, for instance, \cite[Theorem 5.2]{GilbargTrudinger}. The uniqueness follows from a parabolic maximum principle, Proposition \ref{MaxPrincipleProp}. Thus, it remains to prove the a priori estimate \eqref{TimeDecayEqn} for solutions $w\in C^0((0,1); C^{2,\beta}(\mathbb{R}^n))$.

The problem \eqref{CauchyProbEqn} can be rewritten as follows:
$$
\left\{
\begin{array}{cc}
\partial_t w-\Delta w=\tilde{h} & \mbox{in $\mathbb{R}^n\times (0,1)$} \\
\displaystyle \lim_{(x^\prime,t)\to (x,0)}w(x^\prime,t)=0 & \mbox{for $x\in\mathbb{R}^n$}
\end{array}
\right.
$$
where 
$$
\tilde{h}=\sum_{i,j=1}^n(a^{ij}-\delta_{ij}) \partial^2_{x_ix_j} w+h.
$$
Using the estimate $|| f g||_{\beta}\leq 2 ||f||_\beta ||g||_\beta$ we obtain
$$
\Vert\tilde{h}(\cdot,t)\Vert_{\beta} \leq 2\sum_{i,j=1}^n \Vert a^{ij}(\cdot,t)-\delta_{ij}\Vert_{\beta}\Vert \partial^2_{x_ix_j} w(\cdot,t)\Vert_{\beta}+\Vert h(\cdot,t)\Vert_{\beta}.
$$
Thus, it follows from Proposition \ref{SpecialCauchyProbProp} that 
$$
\sup_{0<t<1}\sum_{i=0}^2 t^{\frac{i-2}{2}} \Vert\nabla^iw(\cdot,t)\Vert_{\beta} \leq C\epsilon_0 \sup_{0<t<1}\Vert\nabla^2w(\cdot,t)\Vert_{\beta}+C\sup_{0<t<1}\Vert h(\cdot,t)\Vert_{\beta}
$$
for some $C=C(n,\beta)$. Now, choosing $\epsilon_0=1/(2C)$, the first term on the right hand side can be absorbed into the left, so
$$
\sup_{0<t<1} \sum_{i=0}^2 t^{\frac{i-2}{2}} \Vert\nabla^i w(\cdot,t)\Vert_{\beta} \leq 2C\sup_{0<t<1}\Vert h(\cdot,t)\Vert_{\beta},
$$
completing the proof. 
\end{proof}

\begin{thm} \label{AsymptoticDirichletThm}
Given $f\in C^{k-2,\alpha}_{-1}(\Sigma)$ there is a unique solution in $\mathcal{D}^{k,\alpha}(\Sigma)$ of $\mathscr{L}_\Sigma u=f$ satisfying
$$
\Vert u\Vert_{k,\alpha}^* \leq C_4 \Vert f\Vert_{k-2,\alpha}^{(-1)}
$$
for some $C_4=C_4(\Sigma,n,k,\alpha)$.
\end{thm}

\begin{proof}
By Proposition \ref{C0DecayProp} and Lemma \ref{DerivativeDecayLem}, and the fact that $f\in C^{0}_{-1}(\Sigma)$, there is a unique solution, $u$, in $C^{k,\alpha}_{loc}\cap\mathcal{D}^{2,0}(\Sigma)$ of $\mathscr{L}_\Sigma u=f$. Moreover, for $i=0,1,2$, this $u$ satisfies
\begin{equation} \label{InitialC2DecayEqn}
\Vert \nabla_\Sigma^i u\Vert_{0}^{(-1)} \leq C_i \Vert f\Vert_{\alpha}^{(-1)}.
\end{equation}
Letting $\Sigma_t=\sqrt{t}\, \Sigma$, we define a function $v(p,t)=t^{\frac{1}{2}} u(t^{-\frac{1}{2}} p)$ for $p\in\Sigma_t$ and $t>0$. In order to show the solution $u\in \mathcal{D}^{k,\alpha}(\Sigma)$, it suffices to establish that for each $0\leq l\leq k-2$, there is $c_l=c_l(\Sigma,n,k,\alpha)$ so that
\begin{equation} \label{InductiveDecayEqn}
\sup_{0<t<1}\sum_{j=0}^2 t^{\frac{j-2}{2}} \Vert\nabla_{\Sigma_t}^{l+j} v(\cdot,t)\Vert_{\alpha; \Sigma_t\cap (B_2\setminus \bar{B}_1)} \leq c_l \Vert f\Vert^{(-1)}_{k-2,\alpha}.
\end{equation}
In what follows we will prove \eqref{InductiveDecayEqn} by induction on $l$.

Denote by $\mathcal{C}$ the asymptotic cone of $\Sigma$. For $\Sigma\in\mathcal{ACH}^{k,\alpha}_n$, there exists $\delta>0$ sufficiently small so that for any $q\in\mathcal{C}\cap (B_2\setminus\bar{B}_1)$ and $t\in (0,\delta^2)$, $\Sigma_t\cap B_{2\delta}^{n+1}(q)$ can be parametrized by the map $\Psi_q(\cdot,t)\colon\Omega_q(t)\subset T_q\mathcal{C}\to\mathbb{R}^{n+1}$ defined by
$$
\Psi_q(x,t)=\mathbf{x}(q)+\mathbf{x}(x)+\psi_q(x,t)\mathbf{n}_{\mathcal{C}}(q),
$$
where $B^n_\delta\subset\Omega_q(t)$ and for some $\mu=\mu(\mathcal{C},\delta)$,
\begin{equation} \label{PerturbConeEqn}
\sup_{0<t<\delta^2} \Vert\psi_q(\cdot,t)\Vert_{k,\alpha; B^n_\delta} \leq \mu.
\end{equation}
As $\lim_{t\to 0^+}\Sigma_t=\mathcal{C}$ in $C^{k}_{loc}(\mathbb{R}^{n+1}\setminus \{\mathbf{0}\})$, there is an element $\psi^{\mathcal{C}}_q\in C^{k,\alpha}(\bar{B}_\delta^n)$ so that
$$
\lim_{t\to 0^+} \psi_q(\cdot, t)= \psi_q^{\mathcal{C}} \mbox{ in $C^k(\bar{B}_\delta^n)$}.
$$

Fix any $q\in\mathcal{C}\cap (B_2\setminus\bar{B}_1)$. We define a function $v_q$ on $B^n_\delta\times (0,\delta^2)$ by
$$
v_q(x,t)=v(\Psi_q(x,t),t).
$$
Thus, \eqref{InitialC2DecayEqn} and \eqref{PerturbConeEqn} imply that 
\begin{equation} \label{InductionHypothesisEqn}
\sup_{0<t<\delta^2} \sum_{i=0}^1 t^{\frac{i-1}{2}} \Vert\nabla^i v_q(\cdot,t)\Vert_{\alpha; B^n_\delta} \leq C \Vert f\Vert^{(-1)}_\alpha
\end{equation}
for some $C=C(\mu,C_0,C_1,C_2)$. As $\mathscr{L}_\Sigma u=f$ and $\Sigma$ is a self-expander, the chain rule together with \eqref{InductionHypothesisEqn} gives that
\begin{equation} \label{LocCauchyProbEqn}
\left\{
\begin{array}{cc}
\partial_t v_q-\sum\limits_{i,j=1}^n a^{ij}_q\partial_{x_ix_j}^2 v_q-\sum\limits_{i=1}^n b^i_q \partial_{x_i} v_q=g_q & \mbox{in $B_\delta^n\times (0,\delta^2)$} \\
\displaystyle \lim_{(x^\prime,t)\to (x,0)} v_q(x^\prime,t)=0 & \mbox{for $x\in B_\delta^n$},
\end{array}
\right.
\end{equation}
where $a^{ij}_q\in C^0((0,\delta^2); C^{k-1,\alpha}(B^n_\delta))$ are given by rational functions of $\nabla\psi_q$, $b^i_q\in C^0((0,\delta^2); C^{k-2,\alpha}(B^n_\delta))$ are given by smooth functions of $\nabla\psi_q$ and $\nabla^2\psi_q$, and 
$$
g_{q}(x,t)=-t^{-\frac{1}{2}} f(t^{-\frac{1}{2}}\Psi_q(x,t)).
$$
By taking $\delta$ small, the matrix $(a^{ij}_q)$ is a small $C^0$ perturbation of the identity matrix. Moreover, as $f\in C^{k-2,\alpha}_{-1}(\Sigma)$,  \eqref{PerturbConeEqn} ensures that $g_q\in C^0((0,\delta^2); C^{k-2,\alpha}(B_\delta^n))$ with
\begin{equation} \label{InhomoEstEqn}
\sup_{0<t<\delta^2} \Vert g_q(\cdot,t)\Vert_{k-2,\alpha; B_\delta^n} \leq  C^\prime \Vert f\Vert_{k-2,\alpha}^{(-1)}
\end{equation}
for some $C^\prime=C^\prime(n,k,\alpha,\mu)$.

Let $\chi\colon\mathbb{R}^n\to [0,1]$ be a smooth function so that $\chi\equiv 1$ in $B^n_{1}$ and $\chi\equiv 0$ outside $B_2^n$. And let $\mathscr{D}_\rho (x,t)=(\rho x,\rho^2 t)$ for $(x,t)\in\mathbb{R}^n\times\mathbb{R}$. For $\rho<\frac{1}{8}\delta$, let $w_{q,\rho}=\chi (v_q\circ\mathscr{D}_\rho)$ which we take to be defined on $\mathbb{R}^n\times (0,1)$. Thus, $w_{q,\rho}$ is supported in $B_2^n\times [0,1]$ and, by \eqref{LocCauchyProbEqn}, satisfies
$$
\left\{
\begin{array}{cc}
\partial_t w_{q,\rho}-\sum\limits_{i,j=1}^n a^{ij}_{q,\rho} \partial^2_{x_ix_j} w_{q,\rho}=h_{q,\rho} & \mbox{in $\mathbb{R}^n\times (0,1)$} \\
\displaystyle \lim_{(x^\prime,t)\to (x,0)} w_{q,\rho}(x^\prime,t)=0 & \mbox{for $x\in\mathbb{R}^n$}
\end{array}
\right.
$$
where 
\begin{align*}
a^{ij}_{q,\rho} & =\chi \circ \mathscr{D}_{1/2} (a^{ij}_q\circ\mathscr{D}_\rho)+(1-\chi \circ \mathscr{D}_{1/2}) \delta_{ij}, \mbox{ and}, \\
h_{q,\rho} & =\rho\chi\sum_{i=1}^n (b^i_q\circ\mathscr{D}_\rho)\partial_{x_i} (v_q\circ\mathscr{D}_\rho)-2\sum_{i,j=1}^n (a^{ij}_q\circ\mathscr{D}_\rho)\partial_{x_i}\chi \partial_{x_j}(v_q\circ\mathscr{D}_\rho) \\
& \quad +\rho^2 \chi (g_q\circ\mathscr{D}_\rho)-(v_q\circ\mathscr{D}_\rho)\sum_{i,j=1}^n (a^{ij}_q\circ\mathscr{D}_\rho)\partial^2_{x_ix_j}\chi.
\end{align*}
Here we used that $w_{q,p}=0$ outside $B^n_{2}$ while $(1-\chi \circ \mathscr{D}_{1/2}) =0$ in $B^n_2$.

By \eqref{PerturbConeEqn}, we may choose $\rho=\rho(\Sigma,n,\alpha,\mu)$ sufficiently small so that 
$$
\sup_{0<t<1}\sum_{i,j=1}^n \Vert a^{ij}_{q,\rho}(\cdot,t)-\delta_{ij}\Vert_{\alpha} \leq \epsilon_0(n,\alpha).
$$
Furthermore, by \eqref{InductionHypothesisEqn}, \eqref{PerturbConeEqn} and \eqref{InhomoEstEqn},
$$
\sup_{0<t<1}\Vert h_{q,\rho}(\cdot,t)\Vert_{\alpha}\leq C^{\prime\prime} \sup_{0<t<\delta^2} \Vert g_q(\cdot,t)\Vert_{\alpha; B_{\delta}^n}+\sum_{i=0}^1 \Vert\nabla^i v_q(\cdot,t)\Vert_{\alpha; B^n_\delta} \leq C^{\prime\prime} \Vert f\Vert_{k-2,\alpha}^{(-1)}
$$
for some $C^{\prime\prime}$ depending only on $n,\alpha,\mu, C, C^\prime$. This together with Lemma \ref{CauchyProbLem} and the parabolic maximum principle (i.e., Proposition \ref{MaxPrincipleProp}) implies 
$$
\sup_{0<t<1} \sum_{i=0}^2 t^{\frac{i-2}{2}} \Vert\nabla^i w_{q,\rho}(\cdot,t)\Vert_{\alpha} \leq C_3 C^{\prime\prime} \Vert f\Vert_{k-2,\alpha}^{(-1)}.
$$
Hence, \eqref{InductiveDecayEqn} for $l=0$ follows from transforming the above estimate to that for $v_q$ by the chain rule and invoking \eqref{PerturbConeEqn} and the arbitrariness of $q\in\mathcal{C}\cap (B_2\setminus\bar{B}_1)$.

Suppose \eqref{InductiveDecayEqn} holds for $l=0,\ldots,m$ and $m<k-2$. Let $I=(i_1,\ldots,i_{n})$ be a multi-index of elements of $\{1,\ldots,n\}$ with $|I|=m+1$, and let $x^I=x_1^{i_1}\cdots x_n^{i_n}$. As before, fix any $q\in\mathcal{C}\cap (B_2\setminus\bar{B}_1)$. And we follow the previous notation. By the inductive hypothesis and \eqref{PerturbConeEqn}, 
$$
\sup_{0<t<\delta^2}\sum_{i=0}^1 t^{\frac{i-1}{2}} \Vert\nabla^i\partial_{x^I}^{m+1} v_q(\cdot,t)\Vert_{\alpha; B_\delta^n} \leq \tilde{C} \Vert f\Vert_{k-2,\alpha}^{(-1)},
$$
where $\tilde{C}$ depends only on $n, m, \alpha, \mu, c_0, \ldots, c_m$. This together with \eqref{LocCauchyProbEqn} further gives
$$
\left\{
\begin{array}{cc}
\left(\partial_t-\sum\limits_{i,j=1}^n a_{q}^{ij} \partial^2_{x_ix_j}\right) \partial^{m+1}_{x^I} v_q=g_{q,I} & \mbox{in $B_{\delta}^n\times (0,\delta^2)$} \\
\displaystyle \lim_{(x^\prime,t)\to (x,0)} \partial^{m+1}_{x^I} v_q (x^\prime,t)=0 & \mbox{for $x\in B_{\delta}^n$}
\end{array}
\right.
$$
where 
$$
g_{q,I}=\partial^{m+1}_{x^I} g_q+P_{q,I}(\nabla v_q,\ldots,\nabla^{m+2} v_q)
$$
and $P_{q,I}$ is a multi-linear function with coefficients given by {smooth} functions of $a_{q}^{ij}$ and $b_q^i$ and their spatial derivatives up to the $m$-th order. It follows from our inductive hypothesis, \eqref{PerturbConeEqn} and \eqref{InhomoEstEqn} that 
$$
\sup_{0<t<\delta^2} \Vert g_{q,I}(\cdot,t)\Vert_{\alpha; B_\delta^n} \leq \tilde{C}^\prime \Vert f\Vert^{(-1)}_{k-2,\alpha},
$$
where $\tilde{C}^\prime$ depends only on $n, m, \alpha, \mu, C^\prime, c_0, \ldots, c_m$. Now we may apply the same reasoning to $\partial_{x^I}^{m+1} v_q$ as in the preceding two paragraphs to establish \eqref{InductiveDecayEqn} for $l=m+1$, completing the induction procedure.
\end{proof}

\begin{cor} \label{ExtensionCor}
Given $\varphi\in C^{k,\alpha}(\mathcal{L}(\Sigma))$ the asymptotic Dirichlet problem 
\begin{equation} \label{AsympDirichletEqn}
\left\{
\begin{array}{cc}
\mathscr{L}_\Sigma u=0 & \mbox{in $\Sigma$} \\
\mathrm{tr}^1_\infty[u]=\varphi & \mbox{in $\mathcal{L}(\Sigma)$}
\end{array}
\right.
\end{equation}
has a unique solution in $C^{k,\alpha}_1\cap C^k_{1,\mathrm{H}}(\Sigma)$. Moreover, there is a $C_5=C_5(\Sigma,n,k,\alpha)$ so that
$$
\Vert u\Vert_{k,\alpha}^{(1)} \leq C_5 \Vert\varphi\Vert_{k,\alpha}.
$$
\end{cor}

\begin{proof}
Let $\mathbf{w}$ be a homogeneous transverse section on $\mathcal{C}(\Sigma)$ and $\hat{\mathcal{C}}$ a  $\mathbf{w}$-regular open cone in $\mathbb{R}^{n+1}$ containing $\mathcal{C}(\Sigma)$. Pick an $R>0$ sufficiently large, so that $\Sigma\setminus B_R\subset\hat{\mathcal{C}}$. Let $\chi\colon\mathbb{R}^{n+1}\to [0,1]$ satisfy $\chi\equiv 1$ in $\mathbb{R}^{n+1}\setminus B_{2R}$ and $\chi\equiv 0$ in $B_R$. Then we define a function $g$ on $\Sigma$ to be
$$
g=\left(\chi\mathscr{E}^{\mathrm{H}}_{\mathbf{w},1}[\varphi]\right)\circ\mathbf{x}|_\Sigma.
$$
As $\Sigma\in\mathcal{ACH}^{k,\alpha}_n$, it follows that $g\in C^{k,\alpha}_1\cap C^{k}_{1,\mathrm{H}}(\Sigma)$ with the properties that
\begin{equation} \label{HomoExtEstEqn}
\Vert g \Vert_{k,\alpha}^{(1)} \leq C \Vert\varphi\Vert_{k,\alpha} \mbox{ and } \mathrm{tr}_\infty^1[g]=\varphi
\end{equation}
where $C=C(\Sigma,n,k,\alpha)>0$.

Observe, that on $\Sigma\setminus B_{2R}$, $g=\mathscr{E}^{\mathrm{H}}_{\mathbf{w},1}[\varphi]\circ\mathbf{x}|_\Sigma$, and so
$$
\mathbf{x}\cdot\nabla_\Sigma g-g=-\nabla_{\mathbf{x}^\perp} \mathscr{E}^{\mathrm{H}}_{\mathbf{w},1}[\varphi]\circ\mathbf{x}|_\Sigma=-2\nabla_{\mathbf{H}_\Sigma} \mathscr{E}^{\mathrm{H}}_{\mathbf{w},1}[\varphi]\circ\mathbf{x}|_\Sigma
$$
where we used that $\mathscr{E}^{\mathrm{H}}_{\mathbf{w},1}[\varphi]$ is homogeneous of degree $1$ in the first equality and that $2\mathbf{H}_\Sigma=\mathbf{x}^\perp$ in the last equality. This identity together with $\mathbf{H}_\Sigma\in C^{k-2,\alpha}_{-1}(\Sigma; \mathbb{R}^{n+1})$ implies that for some $C^\prime=C^\prime(\Sigma,n,k,\alpha)$,
\begin{equation} \label{AlmostHomoEstEqn}
\Vert\mathbf{x}\cdot\nabla_\Sigma g-g\Vert_{k-2,\alpha}^{(-1)} \leq C^\prime \Vert\varphi\Vert_{k-1,\alpha}.
\end{equation}

Thus, it follows from \eqref{HomoExtEstEqn} and \eqref{AlmostHomoEstEqn} that $\mathscr{L}_\Sigma g\in C^{k-2,\alpha}_{-1}(\Sigma)$ with 
\begin{equation} \label{InhomoEqn}
\Vert\mathscr{L}_\Sigma g\Vert_{k-2,\alpha}^{(-1)} \leq C^{\prime\prime}\Vert\varphi\Vert_{k,\alpha}
\end{equation}
where $C^{\prime\prime}$ depends only on $\Sigma,n,k,\alpha,C,C^\prime$. Hence, by Theorem \ref{AsymptoticDirichletThm} and \eqref{InhomoEqn}, there is a unique solution, $v$, in $\mathcal{D}^{k,\alpha}(\Sigma)$ of $\mathscr{L}_\Sigma v=-\mathscr{L}_\Sigma g$ with 
\begin{equation} \label{PerturbEstEqn}
\Vert v\Vert_{k,\alpha}^* \leq C_4 C^{\prime\prime}\Vert\varphi\Vert_{k,\alpha}.
\end{equation}

Observe that the inclusion $\mathcal{D}^{k,\alpha}(\Sigma)\subset C^{k,\alpha}_{1}\cap C^{k}_{1,0}(\Sigma)$ is continuous. Therefore, combining \eqref{HomoExtEstEqn} and \eqref{PerturbEstEqn}, we conclude that $v+g\in C^{k,\alpha}_1\cap C^k_{1,\mathrm{H}}(\Sigma)$ is a solution to the problem \eqref{AsympDirichletEqn} with the claimed estimate.

Finally, we show the uniqueness of the solutions $u\in C^{k,\alpha}_1\cap C^k_{1,\mathrm{H}}(\Sigma)$ of the problem \eqref{AsympDirichletEqn}. It suffices to prove that $\varphi=0$ implies $u=0$. For $\Sigma\in\mathcal{ACH}_n^{k,\alpha}$, the vector field
$$
\mathbf{V}=\frac{|\mathbf{x}|}{|\mathbf{x}^\top|^{2}}\, \mathbf{x}^\top
$$
is well defined on $\Sigma\setminus B_{R^\prime}$ for some sufficiently large $R^\prime$. We extend $\mathbf{V}$ to a $C^{k-1,\alpha}$ vector field $\tilde{\mathbf{V}}$ on $\Sigma$. Observe, that $\tilde{\mathbf{V}}$ is complete as $\Sigma$ is asymptotically conical. Then $\tilde{\mathbf{V}}$ generates a one-parameter group, $\{\phi(\cdot,s)\}_{s\in\mathbb{R}}$, of diffeomorphisms of $\Sigma$. Observe, that $|\mathbf{x}(\phi(p,s))|=s+R^\prime$ for $p\in\Sigma\cap\partial B_{R^\prime}$ and $s\geq 0$. Thus, we obtain a diffeomorphism $\phi\colon(\Sigma\cap\partial B_{R^\prime}) \times [0,\infty)\to\Sigma\setminus B_{R^\prime}$. As
$$
(\mathbf{x}\cdot\nabla_\Sigma u-u)\circ\phi=(s+R')\partial_s( u\circ\phi)-u\circ\phi-\frac{|\mathbf{x}^\perp\circ\phi|^2}{s+R'} \, \partial_s (u\circ\phi),
$$
it follows from $2\mathbf{H}_\Sigma=\mathbf{x}^\perp$, $u\in C^{k,\alpha}_{1}(\Sigma)$ and $\mathscr{L}_\Sigma u=0$ that 
$$
(R'+s)\partial_s (u\circ\phi)-u\circ\phi=(R'+s)^2\partial_s ((R'+s)^{-1} u\circ\phi)=O(s^{-1}).
$$
Integrating this gives $u\circ\phi=O(s^{-1})$, that is, $u\in C^{0}_{-1}(\Sigma)$. Therefore, the claim follows from the elliptic maximum principle.
\end{proof}

\subsection{Fredholm property of $L_{\mathbf{v}}$} \label{FredholmSubsec}
Recall, that
$$
L_\Sigma=\Delta_\Sigma+\frac{1}{2}\mathbf{x}\cdot\nabla_\Sigma+|A_\Sigma|^2-\frac{1}{2} \mbox{ and }
L_{\mathbf{v}}=(\mathbf{v}\cdot\mathbf{n}_\Sigma)L_{\Sigma}((\mathbf{v}\cdot\mathbf{n}_\Sigma)\times).
$$
Let
$$
K_{\mathbf{v}}=L_{\mathbf{v}}-(\mathbf{v}\cdot\mathbf{n}_\Sigma)^2\mathscr{L}_\Sigma.
$$

\begin{lem} \label{NormalEqnLem}
The unit normal $\mathbf{n}_\Sigma$ satisfies
$$
L_\Sigma\mathbf{n}_\Sigma+\frac{1}{2}\mathbf{n}_\Sigma=\mathbf{0}.
$$
\end{lem}

\begin{proof}
By a direct computation
$$
\Delta_\Sigma \mathbf{n}_\Sigma=\nabla_\Sigma H_\Sigma-|A_\Sigma|^2\mathbf{n}_\Sigma.
$$
As $2H_\Sigma=-\mathbf{x}\cdot\mathbf{n}_\Sigma$, we have
$$
\nabla_\Sigma H_\Sigma=-\frac{1}{2}\mathbf{x}\cdot\nabla_\Sigma \mathbf{n}_\Sigma.
$$
The lemma follows from combining these two identities.
\end{proof}

\begin{lem} \label{SimplifyKvLem}
There exist $a_{\mathbf{v}}\in C^{k-2,\alpha}_{-2}(\Sigma)$ and $\mathbf{b}_{\mathbf{v}}\in C^{k-2,\alpha}_{-1}(\Sigma; \mathbb{R}^{n+1})$ such that 
$$
K_{\mathbf{v}}=a_{\mathbf{v}}+\mathbf{b}_{\mathbf{v}}\cdot\nabla_\Sigma.
$$
\end{lem}

\begin{proof}
Using the product rule, we compute
\begin{equation} \label{Product1Eqn}
L_\Sigma \left((\mathbf{v}\cdot\mathbf{n}_\Sigma)g\right) =(\mathbf{v}\cdot\mathbf{n}_\Sigma) \mathscr{L}_\Sigma g+2\nabla_\Sigma g \cdot \nabla_\Sigma (\mathbf{v}\cdot\mathbf{n}_\Sigma) + g \left(L_\Sigma+\frac{1}{2}\right)(\mathbf{v}\cdot\mathbf{n}_\Sigma).
\end{equation}
Invoking Lemma \ref{NormalEqnLem}, we simplify
\begin{equation} \label{Product2Eqn}
\left(L_\Sigma+\frac{1}{2}\right)(\mathbf{v}\cdot\mathbf{n}_\Sigma)=\left(\mathscr{L}_\Sigma+\frac{1}{2}\right) \mathbf{v}\cdot\mathbf{n}_\Sigma+2\nabla_\Sigma\mathbf{v}\cdot\nabla_\Sigma\mathbf{n}_\Sigma.
\end{equation}
Thus, substituting \eqref{Product2Eqn} into \eqref{Product1Eqn} and invoking the definition of $K_{\mathbf{v}}$, we get
\begin{equation} \label{KvEqn}
K_{\mathbf{v}} g=a_{\mathbf{v}} g+\mathbf{b}_{\mathbf{v}}\cdot\nabla_\Sigma g,
\end{equation}
where 
$$
a_{\mathbf{v}}=(\mathbf{v}\cdot\mathbf{n}_\Sigma) \left\{\left(\mathscr{L}_\Sigma+\frac{1}{2}\right)\mathbf{v}\cdot\mathbf{n}_\Sigma+2\nabla_\Sigma\mathbf{v}\cdot\nabla_\Sigma\mathbf{n}_\Sigma\right\} \mbox{ and } 
\mathbf{b}_{\mathbf{v}}=\nabla_{\Sigma} (\mathbf{v}\cdot\mathbf{n}_\Sigma)^2.
$$

We show that $a_{\mathbf{v}}\in C^{k-2,\alpha}_{-2}(\Sigma)$ and $\mathbf{b}_{\mathbf{v}}\in C^{k-2,\alpha}_{-1}(\Sigma; \mathbb{R}^{n+1})$. For $\Sigma\in\mathcal{ACH}^{k,\alpha}_n$ one has $\mathbf{n}_\Sigma\in C^{k-1,\alpha}_0(\Sigma;\mathbb{R}^{n+1})$. Our assumptions on $\mathbf{v}$ ensure that $\mathbf{v}\in C^{k,\alpha}_0(\Sigma;\mathbb{R}^{n+1})$ and 
$$
\left(\mathscr{L}_\Sigma+\frac{1}{2}\right)\mathbf{v}\in C^{k-2,\alpha}_{-2}(\Sigma;\mathbb{R}^{n+1}).
$$
Now the claim is a direct consequence of Item \eqref{ListProp2} of Proposition \ref{ListProp}.
\end{proof}

\begin{prop} \label{CompactOperatorProp}
The operator $K_{\mathbf{v}}\colon\mathcal{D}^{k,\alpha}(\Sigma)\to C^{k-2,\alpha}_{-1}(\Sigma)$ is compact.
\end{prop}

\begin{proof}
By our definitions, the inclusion $\mathcal{D}^{k,\alpha}(\Sigma)\subset C^{k-2,\alpha}_{-1}(\Sigma)$ and the linear map
$$
\nabla_\Sigma\colon\mathcal{D}^{k,\alpha}(\Sigma)\to C^{k-2,\alpha}_{-1}(\Sigma; \mathbb{R}^{n+1})
$$
are both continuous. This together with Lemma \ref{SimplifyKvLem} implies that $K_{\mathbf{v}}$ is a bounded linear operator between $\mathcal{D}^{k,\alpha}(\Sigma)$ and $C^{k-2,\alpha}_{-1}(\Sigma)$. 

Next, let $(g_i)_{i\in\mathbb{N}}$ be any sequence of functions in the unit ball of $\mathcal{D}^{k,\alpha}$. We show there is a subsequence $(i_j)_{j\in\mathbb{N}}$ so that $(K_{\mathbf{v}} g_{i_j})_{j\in\mathbb{N}}$ is a Cauchy sequence in $C^{k-2,\alpha}_{-1}(\Sigma)$, implying $K_{\mathbf{v}}$ is compact. Indeed, by the Arzel\`{a}-Ascoli theorem and a diagonal argument, there is a subsequence $(i_j)_{j\in\mathbb{N}}$ so that $(g_{i_j})_{j\in\mathbb{N}}$ is a Cauchy sequence in $C^{k}(\Sigma\cap\bar{B}_R)$ for all $R>0$. Moreover, by our definitions and Item \eqref{ListProp2} of Proposition \ref{ListProp},
\begin{align*}
\Vert K_{\mathbf{v}} g_{i_j}\Vert_{k-2,\alpha; \Sigma\setminus B_R}^{(-1)} & \leq \Vert a_{\mathbf{v}}\Vert_{k-2,\alpha; \Sigma\setminus B_R}^{(0)} \Vert g_{i_j}\Vert^{(-1)}_{k-2,\alpha}+\Vert \mathbf{b}_{\mathbf{v}} \Vert_{k-2,\alpha; \Sigma\setminus B_R}^{(0)} \Vert\nabla_\Sigma g_{i_j}\Vert_{k-2,\alpha}^{(-1)} \\
& \leq R^{-2}\left(\Vert a_{\mathbf{v}} \Vert_{k-2,\alpha}^{(-2)}+R\Vert \mathbf{b}_{\mathbf{v}} \Vert_{k-2,\alpha}^{(-1)}\right) \Vert g_{i_j} \Vert_{k,\alpha}^*.
\end{align*}

Thus, given $\epsilon>0$ there is an $R^\prime>0$ depending only on $\Vert a_{\mathbf{v}}\Vert_{k-2,\alpha}^{(-2)}, \Vert \mathbf{b}_{\mathbf{v}}\Vert_{k-2,\alpha}^{(-1)}$ and $\epsilon$ so that for all $j$,
\begin{equation} \label{EstEndEqn}
\Vert K_{\mathbf{v}} g_{i_j}\Vert_{k-2,\alpha; \Sigma\setminus B_{R^\prime}}^{(-1)} < \frac{\epsilon}{4}.
\end{equation}
Furthermore, there is an integer $j_0$ depending only on $\Vert a_{\mathbf{v}}\Vert_{k-2,\alpha}^{(-2)}, \Vert \mathbf{b}_{\mathbf{v}}\Vert_{k-2,\alpha}^{(-1)}, g_i, \epsilon$ and $R^\prime$ so that for all $l,m>j_0$,
\begin{equation} \label{EstCompactEqn}
\Vert K_{\mathbf{v}} g_{i_l}-K_{\mathbf{v}} g_{i_m}\Vert_{k-2,\alpha; \Sigma\cap B_{2R^\prime}}^{(-1)}  <\frac{\epsilon}{2}.
\end{equation}
Hence, it follows from the triangle inequality, \eqref{EstEndEqn} and \eqref{EstCompactEqn} that for all $l,m>j_0$,
$$
\Vert K_{\mathbf{v}} g_{i_l}-K_{\mathbf{v}} g_{i_m}\Vert_{k-2,\alpha}^{(-1)}<\epsilon.
$$
Hence, it follows that $(K_{\mathbf{v}} g_{i_j})_{j\in\mathbb{N}}$ is a Cauchy sequence in $C^{k-2,\alpha}_{-1}(\Sigma)$. 
\end{proof}

\begin{thm} \label{FredholmThm}
The operator $L_{\mathbf{v}}\colon\mathcal{D}^{k,\alpha}(\Sigma)\to C^{k-2,\alpha}_{-1}(\Sigma)$ is Fredholm of index $0$.
\end{thm}

\begin{proof}
Since 
$$
L_{\mathbf{v}}=(\mathbf{v}\cdot\mathbf{n}_\Sigma)^2\mathscr{L}_\Sigma+K_{\mathbf{v}}
$$
where $K_{\mathbf{v}}$ is a compact operator by Proposition \ref{CompactOperatorProp}, to prove the theorem, it suffices to show that $(\mathbf{v}\cdot\mathbf{n}_\Sigma)^2\mathscr{L}_\Sigma$ is an isomorphism between $\mathcal{D}^{k,\alpha}(\Sigma)$ and $C^{k-2,\alpha}_{-1}(\Sigma)$.

First, by our definition of $\mathcal{D}^{k,\alpha}(\Sigma)$ and Theorem \ref{AsymptoticDirichletThm}, $\mathscr{L}_\Sigma$ is an isomorphism between $\mathcal{D}^{k,\alpha}(\Sigma)$ and $C^{k-2,\alpha}_{-1}(\Sigma)$. If the map defined by
$$
(\mathbf{v}\cdot\mathbf{n}_\Sigma)^2\times\colon g\mapsto (\mathbf{v}\cdot\mathbf{n}_\Sigma)^2 g
$$
were an isomorphism of $C^{k-2,\alpha}_{-1}(\Sigma)$, then the claim would follow immediately.

For $\Sigma\in\mathcal{ACH}^{k,\alpha}_n$ and our assumptions on $\mathbf{v}$, 
$$
\mathbf{n}_\Sigma\in C^{k-1,\alpha}_0\cap C^{k-1}_{0,\mathrm{H}} (\Sigma; \mathbb{R}^{n+1}) \mbox{ and } \mathbf{v}\in C^{k,\alpha}_0\cap C^{k}_{0,\mathrm{H}} (\Sigma; \mathbb{R}^{n+1}).
$$
It follows from Item \eqref{ListProp2} of Proposition \ref{ListProp} that $(\mathbf{v}\cdot\mathbf{n}_\Sigma)^2\in C^{k-1,\alpha}_0\cap C^{k-1}_{0,\mathrm{H}}(\Sigma)$. Moreover, as $\mathbf{v}$ and $\mathcal{C}[\mathbf{v}]$ are transverse to, respectively, $\Sigma$ and $\mathcal{C}(\Sigma)$ we have
$$
\inf_{p\in\Sigma} |\mathbf{v}\cdot\mathbf{n}_\Sigma|(p)=\delta>0.
$$ 
It follows from Item \eqref{ListProp3} of Proposition \ref{ListProp} that $(\mathbf{v}\cdot\mathbf{n}_\Sigma)^{-2}\in C^{k-1,\alpha}_0 (\Sigma)$. Hence,
$$
g\in C^{k-2,\alpha}_{-1}(\Sigma) \Longleftrightarrow (\mathbf{v}\cdot\mathbf{n}_\Sigma)^2 g \in C^{k-2,\alpha}_{-1}(\Sigma).
$$
Moreover, there is a $C>0$ depending only on $\Sigma,\mathbf{v},n,k,\alpha$ so that
\begin{equation} \label{NormEquivEqn}
C^{-1} \Vert g \Vert_{k-2,\alpha}^{(-1)}\leq \Vert (\mathbf{v}\cdot\mathbf{n}_\Sigma)^{2} g \Vert_{k-2,\alpha}^{(-1)} \leq C \Vert g \Vert_{k-2,\alpha}^{(-1)},
\end{equation}
implying the map $(\mathbf{v}\cdot\mathbf{n}_\Sigma)^2\times$ is an isomorphism of $C^{k-2,\alpha}_{-1}(\Sigma)$.
\end{proof}

\section{Asymptotic behavior of Jacobi functions}\label{JacobiSec}
Continue to consider a self-expander $\Sigma\in\mathcal{ACH}_n^{k,\alpha}$ and a transverse section on $\Sigma$, $\mathbf{v}\in C^{k,\alpha}_0\cap C^k_{0,\mathrm{H}}(\Sigma;\mathbb{R}^{n+1})$ so that $\mathbf{w}=\mathcal{C}[\mathbf{v}]$ is a transverse section on $\mathcal{C}(\Sigma)$. We say a function $u$ is a \emph{Jacobi function of $\Sigma$} if $L_\Sigma u=0$. In this case, $u\mathbf{n}_\Sigma$ is a \emph{normal Jacobi field of $\Sigma$}. Similarly, a function $u$ is a \emph{$\mathbf{v}$-Jacobi function} if $L_\mathbf{v} u=0$ and in this case $u\mathbf{v}$ is a \emph{$\mathbf{v}$-Jacobi field}. Let us denote the set of Jacobi functions of $\Sigma$ so that the corresponding normal Jacobi fields fix the asymptotic cone of $\Sigma$ by
\begin{equation}\label{KspaceEqn}
\mathcal{K}=\left\{u\in C^2_{loc}(\Sigma)\cap C^{0}_{1, 0}(\Sigma)\colon L_\Sigma u=0\right\}.
\end{equation}
Similarly, denote the set of $\mathbf{v}$-Jacobi functions on $\Sigma$ whose corresponding $\mathbf{v}$-Jacobi fields fix the asymptotic cone of $\Sigma$ by
\begin{equation}\label{KvspaceEqn}
\mathcal{K}_{\mathbf{v}}=\left\{u\in C^2_{loc}(\Sigma)\cap C^{0}_{1, 0}(\Sigma)\colon L_{\mathbf{v}} u=0\right\}.
\end{equation}
The elements of these spaces are significantly more regular locally and have better asymptotic decay properties.  To see this, we first use the analysis of Section \ref{FredholmSec} and a result from \cite{Bernstein} to show improved regularity and decay for elements of the two spaces.  

Indeed, for $r(p)=|\mathbf{x}(p)|$, $p\in \Sigma$, and $\partial_r =\nabla_\Sigma r$, one has

\begin{lem} \label{RegJacobiLem}
The spaces $\mathcal{K}$ and  $\mathcal{K}_{\mathbf{v}}$ are both finite dimensional subspaces of  $\mathcal{D}^{k, \alpha}(\Sigma)$, and $\dim \mathcal{K}=\dim \mathcal{K}_{\mathbf{v}}$. Moreover, for all $m\in \mathbb{R}$ any element $u\in \mathcal{K}\cup \mathcal{K}_{\mathbf{v}}$ satisfies
\begin{equation} \label{IntDecayEst}
\int_{\Sigma} \left(|\nabla_\Sigma u|^2 + u^2\right) r^m {e}^{\frac{r^2}{4}} \, d\mathcal{H}^n<\infty.
\end{equation}
\end{lem}

\begin{proof}
Suppose that $u\in\mathcal{K}$. By standard elliptic regularity results $\Sigma$ is a smooth hypersurface and $u\in C^\infty_{loc}(\Sigma)$. As $|A_\Sigma|\in C^{\infty}_{loc}\cap C^{k-2,\alpha}_{-1}(\Sigma)$, the function $f=\mathscr{L}_\Sigma u=-|A_\Sigma|^2 u$ satisfies $f\in C^\infty_{loc}\cap C^0_{-1}(\Sigma)$. By Proposition \ref{C0DecayProp} and the local gradient estimate \cite[Theorem 8.32]{GilbargTrudinger} there is a unique solution, $\tilde{u}$, in $C^0_{-1}\cap C^1_{1}(\Sigma)$ of $\mathscr{L}_\Sigma\tilde{u}=f$. We first show $\tilde{u}=u$.
	 
To that end, observe that as $\Sigma\in\mathcal{ACH}_{n}^{k,\alpha}$ is a self-expander, on $\Sigma\setminus B_1$,
$$
r|A_\Sigma|+r^4||\nabla_\Sigma r|-1|+r^2|\nabla^2_\Sigma r^2-2g_\Sigma|\leq C.
$$
Hence, there is an $R_\Sigma>1$ so that $(\Sigma\backslash \bar{B}_{R_\Sigma}, g_\Sigma, r)$ is a weakly conical end in the sense of \cite{Bernstein}. Observe that $\mathscr{L}_\Sigma=\Delta_{g_\Sigma} +\frac{r}{2}\partial_r -\frac{1}{2}$. Let $v=u-\tilde{u}$ and observe that as $u, v\in C^{0}_{1,0}(\Sigma)$, 
$$
\lim_{\rho\to \infty} \rho^{-n-1} \int_{\Sigma\cap\partial B_{\rho}} u^2 \, d\mathcal{H}^{n-1}=\lim_{\rho\to \infty} \rho^{-n-1} \int_{\Sigma\cap\partial B_{\rho}} v^2 \, d\mathcal{H}^{n-1}=0.
$$
Hence, as $L_\Sigma u=0$ and $\mathscr{L}_\Sigma v=\mathscr{L}_\Sigma u -\mathscr{L}_\Sigma \tilde{u}=0$, \cite[Theorem 9.1]{Bernstein} implies that \eqref{IntDecayEst} holds, with any $m$, for both $u$ and $v$. In particular, for $R>1$ sufficiently large, integrating by parts and using $\mathscr{L}_\Sigma v=0$, gives
\begin{align*}
\int_{\Sigma\cap B_R} & \left( |\nabla_\Sigma v|^2 +\frac{1}{2}v^2 \right) {e}^{\frac{r^2}{4}} \, d\mathcal{H}^n=\int_{\Sigma \cap \partial B_R} v\partial_r v |\nabla_\Sigma r|^{-1} {e}^{\frac{r^2}{4}} \, d\mathcal{H}^n \\
& \leq 2\left( \int_{\Sigma \cap \partial B_R} v^2 {e}^{\frac{r^2}{4}} \, d\mathcal{H}^n\right)^{1/2} \left( \int_{\Sigma\cap  \partial B_R} |\nabla_\Sigma v|^2 {e}^{\frac{r^2}{4}} \, d\mathcal{H}^n\right)^{1/2} .
\end{align*}
As \eqref{IntDecayEst} holds for $v$ with $m=0$, there is a sequence of $R\to \infty$, for which the right hand side tends to zero. Hence, $v$ identically vanishes and so $u=\tilde{u}$ and $u\in C^{0}_{-1}\cap C^1_1(\Sigma)$.
	 
We will now argue by induction to show $u\in \mathcal{D}^{k,\alpha}(\Sigma)$. As $C^1_1(\Sigma)\subset C^\alpha_1(\Sigma)$, that $u\in C^1_1(\Sigma)$ and the decay of $|A_\Sigma|$ ensure that $f=-|A_\Sigma|^2 u\in C^{\alpha}_{-1}(\Sigma)$. Let $\bar{u}$ be the unique element of $\mathcal{D}^{2,\alpha}(\Sigma)$ that satisfies $\mathscr{L}_\Sigma\bar{u}=f$ given by Theorem \ref{AsymptoticDirichletThm}. As $\mathscr{L}_\Sigma (\bar{u}-u)=0$ and $\bar{u}-u\in C^2_{loc}\cap C^0_{-1}(\Sigma)$, the elliptic maximum principle implies $u=\bar{u}$ and so $u\in\mathcal{D}^{2,\alpha}(\Sigma)$, proving the claim for $k=2$. Suppose that we have shown $u\in\mathcal{D}^{l,\alpha}(\Sigma)$ for $2\leq l \leq k-1$. As $\mathcal{D}^{l,\alpha}(\Sigma)\subset C^{l,\alpha}_{1}(\Sigma)$,  the decay of $|A_\Sigma|^2$ ensures that $f=-|A_{\Sigma}|^2 u\in C^{l^\prime,\alpha}_{-1}(\Sigma)$ where $l^\prime=\min\{k-2, l\}$. Theorem \ref{AsymptoticDirichletThm} gives a unique $\hat{u}\in \mathcal{D}^{l^\prime+2, \alpha}(\Sigma)$ with $\mathscr{L}_\Sigma \hat{u}=f$. As $\hat{u}-u\in C^2_{loc}\cap C^0_{-1}(\Sigma)$ and $\mathscr{L}_\Sigma(\hat{u}-u)=0$ the elliptic maximum principle implies $u=\hat{u}$ and so $u\in \mathcal{D}^{k,\alpha}(\Sigma)$ as claimed.
	
Suppose now that $u\in\mathcal{K}_{\mathbf{v}}$. For $\Sigma\in\mathcal{ACH}^{k,\alpha}_n$ and our assumptions on $\mathbf{v}$ we have that $|\mathbf{v}\cdot\mathbf{n}_\Sigma|\in C^{k-1,\alpha}_0(\Sigma)$ and is uniformly bounded from below by a positive constant. This implies $(\mathbf{v}\cdot\mathbf{n}_\Sigma)^{-1}\in C^{k-1,\alpha}_0(\Sigma)$. Thus, $u\in \mathcal{K}_{\mathbf{v}}$ if and only if $(\mathbf{v}\cdot \mathbf{n}_\Sigma) u \in \mathcal{K}$. By what have shown \eqref{IntDecayEst} holds for $u$ and $(\mathbf{v}\cdot\mathbf{n}_\Sigma) u\in\mathcal{D}^{k,\alpha}(\Sigma)$ implying $u\in C^{k-2,\alpha}_{-1}\cap C^{k-1,\alpha}_0(\Sigma)$. By writing $L_{\mathbf{v}}=(\mathbf{v}\cdot \mathbf{n}_\Sigma)^2 \mathscr{L}_\Sigma +K_{\mathbf{v}}$,  the definition of $K_\mathbf{v}$ (see \eqref{KvEqn}) ensures that $\mathscr{L}_\Sigma u \in C^{k-2,\alpha}_{-1}(\Sigma)$.  Hence, arguing as above gives $u\in \mathcal{D}^{k, \alpha}(\Sigma)$. 

Finally, by Theorem \ref{FredholmThm}, $L_{\mathbf{v}}$ is Fredholm and so $\mathcal{K}_{\mathbf{v}}$ is finite dimensional and, as $\mathcal{K}_{\mathbf{v}}$ and $\mathcal{K}$ are isomorphic, the same is true of $\mathcal{K}$.
\end{proof}

We will actually need a sharper decay estimate for elements of $\mathcal{K}$ proved in \cite{Bernstein}. Roughly speaking, this result says that for any non-trivial $u\in \mathcal{K}$, there will be a non-zero element $\mathrm{tr}_\infty^*[u]\in L^2(\mathcal{L}(\Sigma))$ so that one has the asymptotic expansion
$$
u= r^{-n-1} e^{-\frac{r^2}{4}} \mathscr{E}_{0,\mathbf{w}}^{\mathrm{H}}[\mathrm{tr}_\infty^*[u]]\circ\mathbf{x}|_\Sigma+o\left(r^{-n-1} e^{-\frac{r^2}{4}} \right).
$$
This expansion will be crucial in justifying several integration by parts arguments.

More precisely, with  $A_{R^\prime,R}=\Sigma\cap (\bar{B}_{R^\prime}\setminus B_R)$, we have

\begin{prop} \label{AsymptoticKerProp}
If $u\in\mathcal{K}$, then 
\begin{equation} \label{L2AsymptoticEqn}
\lim_{\rho\to\infty} \rho^{n+3} e^{\frac{\rho^2}{2}} \int_{\Sigma\cap\partial B_\rho} u^2 \, d\mathcal{H}^{n-1}=a[u]^2<\infty,
\end{equation}
and there are constants $R_1$ and $C_6$, depending on $\Sigma$ and $u$, so that for all $R\geq R_1$, 
\begin{equation} \label{EnergyDecayEqn}
\int_{A_{2R,R}} \left(|\nabla_\Sigma u|^2+r^4(2\partial_r u+ru)^2\right) e^{\frac{r^2}{2}}\, d\mathcal{H}^n \leq C_6 a[u]^2 R^{-n}.
\end{equation}
Moreover, there is an injective linear map
$$
\mathrm{tr}^*_\infty\colon \mathcal{K} \to L^2(\mathcal{L}(\Sigma))
$$
satisfying 
\begin{equation} \label{L2TraceEqn}
\int_{\mathcal{L}(\Sigma)} \mathrm{tr}^*_\infty[u]^2\, d\mathcal{H}^{n-1}=a[u]^2,
\end{equation}
and for all $R\geq R_1$,
\begin{equation} \label{TraceApproxDecayEqn}
\int_{A_{2R,R}} (u-\mathscr{F}_{\mathbf{w}}[u])^2 e^{\frac{r^2}{2}} \, d\mathcal{H}^n\leq C_6 a[u]^2 R^{-n-4}
\end{equation}
where
$$
\mathscr{F}_{\mathbf{w}}[u]=r^{-n-1} e^{-\frac{r^2}{4}}\mathscr{E}_{0,\mathbf{w}}^{\mathrm{H}}[\mathrm{tr}^*_\infty[u]].
$$
\end{prop}

\begin{proof}
As in the proof of Lemma \ref{RegJacobiLem},  $(\Sigma\backslash \bar{B}_{R_\Sigma}, g_\Sigma, r)$ is a weakly conical end in the sense of \cite{Bernstein}.
Define $\hat{u}=r^{n+1} e^{r^2/4} u$. As $u\in \mathcal{K}$ satisfies $L_\Sigma u=0$ and \eqref{IntDecayEst}, it follows from \cite[Theorem 7.2]{Bernstein} that 
\begin{equation} \label{HatL2TraceEqn}
\lim_{\rho\to\infty} \rho^{1-n} \int_{\Sigma\cap\partial B_{\rho}} \hat{u}^2 \, d\mathcal{H}^{n-1}=a[u]^2<\infty,
\end{equation}
and there are constants $\hat{R}$ and $\hat{C}$, depending on $\Sigma$ and $u$, so that for $R\geq \hat{R}$,
\begin{equation} \label{HatEnergyDecayEqn}
\int_{A_{2R,R}} \left(|\nabla_\Sigma \hat{u}|^2+r^2(\partial_r \hat{u})^2\right)r^{-n}\, d\mathcal{H}^n\leq \hat{C} a[u]^2 R^{-2}.
\end{equation}
Thus, \eqref{L2AsymptoticEqn} and \eqref{EnergyDecayEqn} follow from \eqref{HatL2TraceEqn} and \eqref{HatEnergyDecayEqn} together with the observations that 
\begin{align*}
e^{\frac{r^2}{4}}\nabla_\Sigma u & =-\left(\frac{n+1}{r}+\frac{r}{2}\right)r^{-n-1}\hat{u}\partial_r+r^{-n-1} \nabla_\Sigma \hat{u},  \\
e^{\frac{r^2}{4}}(2\partial_r u+ru) & = -2(n+1)(1+O(r^{-2}))r^{-n-2}\hat{u}+2r^{-n-1} \partial_r \hat{u}.
\end{align*}

It remains to construct the map $\mathrm{tr}_\infty^*$ with claimed properties. As $\Sigma\in\mathcal{ACH}^{k,\alpha}_n$, by our definition, there is a $C^{k,\alpha}$-regular cone $\mathcal{C}=\mathcal{C}(\Sigma)$ and a function $\psi\in C^{k,\alpha}_1\cap C^k_{1,0}(\mathcal{C}_{R^\prime})$ where $\mathcal{C}_{R^\prime}=\mathcal{C}\setminus\bar{B}_{R^\prime}$ so that for some compact set $K\subset\Sigma$, $\Sigma^\prime=\Sigma\setminus K$ can be parametrized by the map $\Psi\colon\mathcal{C}_{R^\prime}\to\Sigma^\prime\subset\mathbb{R}^{n+1}$ given by
$$
\Psi(p)=\mathbf{x}(p)+\psi(p)\mathbf{w}(p).
$$
As $\Sigma$ is a self-expander, we have
$$
(\mathbf{x}\cdot\nabla_{\mathcal{C}}\Psi-\Psi)\cdot(\mathbf{n}_\Sigma\circ\Psi)=2H_\Sigma\circ\Psi \in C^{0}_{-1}(\mathcal{C}_{R^\prime}).
$$
Invoking the homogeneity of cone and $\mathbf{w}$, this gives
\begin{equation} \label{ApproxHomoEqn}
\mathbf{x}\cdot\nabla_{\mathcal{C}}\psi-\psi \in C^0_{-1}(\mathcal{C}_{R^\prime}).
\end{equation}

Given $\tau>0$ we define $\hat{u}_\tau=\hat{u}(\Psi(\tau\cdot))$ on $\mathcal{C}_{R^\prime/\tau}$. Then
$$
\partial_\tau\hat{u}_\tau(p)=\tau^{-1}\left\{|\Psi(\tau p)|\partial_r\hat{u}(\Psi(\tau p))+(\tau p\cdot\nabla_{\mathcal{C}}\psi(\tau p)-\psi(\tau p))\nabla_\Sigma\hat{u}(\Psi(\tau p))\cdot\mathbf{w}(p)\right\}.
$$
Thus, invoking \eqref{ApproxHomoEqn}, for some constant $C^\prime=C^\prime(\mathcal{C},\mathbf{w},\psi)$,
\begin{equation} \label{TimeDerEqn}
|\partial_\tau\hat{u}_\tau(p)|\leq \tau^{-1}|r\partial_r\hat{u}|(\Psi(\tau p))+C^\prime\tau^{-2}|\mathbf{x}(p)|^{-1}|\nabla_\Sigma\hat{u}|(\Psi(\tau p)).
\end{equation}
Fix any $\rho_2\geq \rho_1>0$. Denote by $A^{c}_{\rho_2,\rho_1}=\mathcal{C}\cap (\bar{B}_{\rho_2}\setminus B_{\rho_1})$. For $\tau_2\geq\tau_1\gg1$, we use the H\"{o}lder inequality and \eqref{TimeDerEqn} to estimate:
\begin{equation} \label{L2DiffEqn}
\begin{split}
& \int_{A^{c}_{\rho_2,\rho_1}} \left(\hat{u}_{\tau_2}-\hat{u}_{\tau_1}\right)^2 d\mathcal{H}^n \leq \int_{A^{c}_{\rho_2,\rho_1}} \left( \int_{\tau_1}^{\tau_2} |\partial_\tau\hat{u}_\tau| \, d\tau\right)^2 d\mathcal{H}^n \\
 &\leq 2\int_{\tau_1}^{\tau_2} \frac{1}{\tau^{2}}\, d\tau \int_{\tau_1}^{\tau_2}\int_{A_{\rho_2,\rho_1}^c} |r\partial_r\hat{u}|^2(\Psi(\tau p))\, d\mathcal{H}^n d\tau \\
 & \quad +\frac{2(C^\prime)^2}{\rho_1^2}\int_{\tau_1}^{\tau_2} \frac{1}{\tau^4}\, d\tau \int_{\tau_1}^{\tau_2}\int_{A_{\rho_2,\rho_1}^c} |\nabla_\Sigma\hat{u}|^2(\Psi(\tau p))\, d\mathcal{H}^n d\tau.
\end{split}
\end{equation}

Observe that for sufficiently large $\tau$, $\Psi(A^c_{\tau\rho_2,\tau\rho_1})\subset A_{2\tau\rho_2,\frac{\tau\rho_1}{2}}$ and the Jacobian of $\Psi$ on $A^c_{\tau\rho_2,\tau\rho_1}$ is bounded from below by $\frac{1}{2}$. Thus, an application of the change of variables formula and \eqref{HatEnergyDecayEqn} gives that
\begin{align*}
& \int_{A^c_{\rho_2,\rho_1}} \left(|\nabla_\Sigma\hat{u}|^2+(r\partial_r\hat{u})^2\right)(\Psi(\tau p)) \, d\mathcal{H}^n \\
& \leq \frac{2}{\tau^n}\int_{A_{2\tau\rho_2,\frac{\tau\rho_1}{2}}} \left(|\nabla_\Sigma\hat{u}|^2+(r\partial_r\hat{u})^2\right) \, d\mathcal{H}^n \leq C^{\prime\prime} a[u]^2\tau^{-2},
\end{align*}
where $C^{\prime\prime}$ depends on $\hat{C},n,\rho_1$ and $\rho_2$. Hence, substituting this estimate into \eqref{L2DiffEqn} gives
$$
\int_{A^c_{\rho_2,\rho_1}} \left(\hat{u}_{\tau_2}-\hat{u}_{\tau_1}\right)^2 d\mathcal{H}^n\leq \tilde{C}a[u]^2\tau_1^{-2}
$$
where $\tilde{C}$ depends on $\rho_1,C^\prime$ and $C^{\prime\prime}$. 

Therefore, it follows that 
$$
\lim_{\tau\to\infty} \hat{u}_\tau=\hat{u}_\infty \mbox{ in $L^2_{loc}(\mathcal{C})$}
$$
for some $\hat{u}_\infty\in L^2_{loc}(\mathcal{C})$. Moreover, $\hat{u}_\infty(\rho p)=\hat{u}_\infty(p)$ for all $p\in\mathcal{C}$ and $\rho>0$, and for $\rho_2\geq \rho_1>0$ and for sufficiently large $\tau$,
\begin{equation} \label{L2ApproxEqn}
\int_{A^c_{\rho_2,\rho_1}} \left(\hat{u}_\tau-\hat{u}_\infty\right)^2 d\mathcal{H}^n \leq \tilde{C}a[u]^2\tau^{-2}.
\end{equation}
As $\hat{u}_\infty$ is homogeneous of degree $0$, we can define
$$
\mathrm{tr}_\infty^*[u]=\hat{u}_\infty|_{\mathcal{L}(\Sigma)}\in L^2(\mathcal{L}(\Sigma)).
$$
The linear dependence of $\mathrm{tr}_\infty^*$ on $u$ is justified by our construction of $\hat{u}_\infty$. And \eqref{L2TraceEqn} is given by combining \eqref{HatL2TraceEqn} and \eqref{L2ApproxEqn}. Furthermore, \eqref{TraceApproxDecayEqn} follows from substituting $u=r^{-n-1}e^{-r^2/4}\hat{u}$ into \eqref{L2ApproxEqn} and a change of variables. Finally, in view of \eqref{L2AsymptoticEqn}, \eqref{EnergyDecayEqn} and \eqref{L2TraceEqn}, the map $\mathrm{tr}_\infty^*$ is injective.
\end{proof}

We obtain several useful corollaries of Proposition \ref{AsymptoticKerProp}.  
 
\begin{cor} \label{CokerCor}
For all $\kappa\in\mathcal{K}_{\mathbf{v}}\setminus\{0\}$ there are no solutions in $\mathcal{D}^{k,\alpha}(\Sigma)$ of $L_{\mathbf{v}} u=\kappa$. 
\end{cor}

\begin{proof}
We argue by contradiction. Suppose there were a $\kappa\in\mathcal{K}_{\mathbf{v}}\setminus\{0\}$ and a $u\in\mathcal{D}^{k,\alpha}(\Sigma)$ satisfying $L_{\mathbf{v}}u=\kappa$. Let $\tilde{\kappa}=(\mathbf{v}\cdot\mathbf{n}_\Sigma)\kappa$ and $\tilde{u}=(\mathbf{v}\cdot\mathbf{n}_\Sigma) u$. By standard elliptic regularity theory, $\Sigma$ is a smooth hypersurface. Then $\tilde{\kappa}\in\mathcal{K}$, $\tilde{u}\in C^{2}_{loc}(\Sigma)$ and 
\begin{equation} \label{FuncDecayEqn}
\tilde{u}, |\nabla_\Sigma\tilde{u}| \in C^{0}_{-1}(\Sigma).
\end{equation}
Denote by $S_\rho=\Sigma\cap\partial B_\rho$. Since $L_\Sigma\tilde{\kappa}=0$ and $(\mathbf{v}\cdot\mathbf{n}_\Sigma)L_\Sigma\tilde{u}=\kappa$, it follows from the divergence theorem and the Cauchy-Schwarz inequality that for sufficiently large $\rho$,
\begin{align*}
& \int_{\Sigma\cap B_\rho} \kappa^2 e^{\frac{r^2}{4}}\, d\mathcal{H}^n =e^{\frac{\rho^2}{4}}\int_{S_\rho} (\tilde{\kappa}\partial_r\tilde{u}-\tilde{u}\partial_r\tilde{\kappa}) |\nabla_\Sigma r|^{-1}\, d\mathcal{H}^{n-1} \\
& \leq 2\left(\rho^{3-n}\int_{S_\rho} (|\tilde{u}|^2+|\partial_r\tilde{u}|^2)\, d\mathcal{H}^{n-1}\right)^{\frac{1}{2}} \left(\rho^{n-3}e^{\frac{\rho^2}{2}}\int_{S_\rho} (|\tilde{\kappa}|^2+|\partial_r \tilde{\kappa}|^2) \, d\mathcal{H}^{n-1} \right)^{\frac{1}{2}}.
\end{align*}
Invoking \eqref{FuncDecayEqn}, \eqref{L2AsymptoticEqn} and \eqref{EnergyDecayEqn}, it follows that for some sequence of $\rho\to\infty$ the last term of the above identity converges to $0$. Thus, by the monotone convergence theorem,
$$
\int_\Sigma \kappa^2 e^{\frac{|\mathbf{x}|^2}{4}} \, d\mathcal{H}^n=0,
$$
giving a contradiction.
\end{proof}

\begin{cor} \label{MonomorphismCor}
Given $\epsilon\in (0,1)$ there is an injective linear map $T^\epsilon_{\mathbf{v}}\colon\mathcal{K}_{\mathbf{v}}\to C^{k,\alpha}(\mathcal{L}(\Sigma))$ satisfying that for any $\kappa\in\mathcal{K}_{\mathbf{v}}$,
$$
\Vert T^\epsilon_{\mathbf{v}}[\kappa]-\mathrm{tr}_\infty^*[\kappa\mathbf{v}\cdot\mathbf{n}_\Sigma]\Vert_{L^2} < \epsilon \Vert \mathrm{tr}_\infty^*[\kappa\mathbf{v}\cdot\mathbf{n}_\Sigma]\Vert_{L^2}.
$$
\end{cor}

\begin{proof}
We first define a map $T_{\mathbf{v}}\colon\mathcal{K}_{\mathbf{v}}\to L^2(\mathcal{L}(\Sigma))$ by 
$$
T_{\mathbf{v}}[\kappa]=\mathrm{tr}_\infty^*[\kappa\mathbf{v}\cdot\mathbf{n}_\Sigma].
$$
The map $T_{\mathbf{v}}$ is linear by the linearity of $\mathrm{tr}_\infty^*$. It is injective because $\mathbf{v}$ is a transverse section on $\Sigma$ and Proposition \ref{AsymptoticKerProp}.

Next, as $\dim \mathcal{K}_{\mathbf{v}}<\infty$, the subspace $T_{\mathbf{v}}(\mathcal{K}_{\mathbf{v}})$ is also finite dimensional. We choose an orthonormal basis, $\varphi_1,\ldots,\varphi_m$, of $T_{\mathbf{v}}(\mathcal{K}_{\mathbf{v}})$. Using the partition of unity and the mollification, given $\epsilon\in (0,1)$ we may find $\varphi_i^\epsilon\in C^{k,\alpha}(\mathcal{L}(\Sigma))$ for $1\leq i \leq m$ so that
$$
\sum_{i=1}^m\Vert\varphi_i^\epsilon-\varphi_i\Vert^2_{L^2} <\epsilon^2.
$$
We then define a linear map $M^\epsilon\colon T_{\mathbf{v}}(\mathcal{K}_{\mathbf{v}})\to C^{k,\alpha}(\mathcal{L}(\Sigma))$ as follows. If $\varphi\in T_{\mathbf{v}}(\mathcal{K}_{\mathbf{v}})$ is given by the unique linear combination of $\varphi_i$, namely,
$$
\varphi=\sum_{i=1}^m a_i\varphi_i \mbox{ where $a_i\in\mathbb{R}$},
$$
then we let 
$$
M^{\epsilon}[\varphi]=\sum_{i=1}^m a_i\varphi_i^\epsilon=\sum_{i=1}^m a_i M^\epsilon[\varphi_i].
$$
Observe that $\Vert\varphi\Vert^2_{L^2}=\sum_{i=1}^m a_i^2$. Thus, by the triangle inequality,
$$
\Vert M^\epsilon[\varphi]-\varphi\Vert_{L^2}\leq \sum_{i=1}^m |a_i|\Vert\varphi_i^\epsilon-\varphi_i\Vert_{L^2}<\epsilon\Vert\varphi\Vert_{L^2}.
$$
Therefore, the corollary follows by setting $T^\epsilon_{\mathbf{v}}=M^\epsilon\circ T_{\mathbf{v}}$.
\end{proof}

\begin{cor} \label{NonExistAsympDirichletCor}
There is an $\epsilon_1=\epsilon_1(\Sigma,\mathbf{w})$ so that given $\zeta\in T^{\epsilon_1}_{\mathbf{v}}(\mathcal{K}_{\mathbf{v}}\setminus\{0\})$ the asymptotic Dirichlet problem
\begin{equation} \label{AsympBoundaryEqn}
\left\{
\begin{array}{cc}
L_{\Sigma} u=0 & \mbox{in $\Sigma$} \\
\mathrm{tr}_\infty^1[u]=\zeta\mathbf{w}\cdot\mathbf{n}_{\mathcal{L}(\Sigma)} & \mbox{in $\mathcal{L}(\Sigma)$}
\end{array}
\right.
\end{equation}
has no solutions in $C^2_{loc}\cap C^{1}_{1,\mathrm{H}}(\Sigma)$.
\end{cor}

\begin{proof}
We argue by contradiction. Take any sufficiently small $\epsilon$. Suppose that there were a $\zeta\in T^{\epsilon}_{\mathbf{v}}(\mathcal{K}_{\mathbf{v}}\setminus\{0\})$ so that the problem \eqref{AsympBoundaryEqn} has a solution $u\in C^2_{loc}\cap C^1_{1,\mathrm{H}}(\Sigma)$. There is a $\kappa\in\mathcal{K}_{\mathbf{v}}\setminus\{0\}$ so that $T^\epsilon_{\mathbf{v}}[\kappa]=\zeta$. Let $\tilde{\kappa}=\kappa\mathbf{v}\cdot\mathbf{n}_\Sigma$, so $\tilde{\kappa}\in\mathcal{K}\setminus\{0\}$. Thus, by the divergence theorem, letting $S_\rho=\Sigma\cap\partial B_\rho$,
\begin{equation} \label{BoundaryIntegralEqn}
e^{\frac{\rho^2}{4}}\int_{S_\rho} \tilde{\kappa}\partial_r u |\nabla_\Sigma r|^{-1} \, d\mathcal{H}^{n-1}=e^{\frac{\rho^2}{4}} \int_{S_\rho} u \partial_r\tilde{\kappa} |\nabla_\Sigma r|^{-1} \, d\mathcal{H}^{n-1}.
\end{equation}

On the one hand, using the Cauchy-Schwarz inequality gives
\begin{align*}
& e^{\frac{\rho^2}{4}}\int_{S_\rho} |\tilde{\kappa}\partial_r u| |\nabla_\Sigma r|^{-1} \, d\mathcal{H}^{n-1} \\
& \leq 2\left(\rho^{1-n}\int_{S_\rho} (\partial_r u)^2 \, d\mathcal{H}^{n-1}\right)^{\frac{1}{2}} \left(\rho^{n-1}e^{\frac{\rho^2}{2}}\int_{S_\rho} \tilde{\kappa}^2 \, d\mathcal{H}^{n-1}\right)^{\frac{1}{2}}.
\end{align*}
Hence,  Proposition \ref{AsymptoticKerProp} implies that there is a sequence of $\rho\to\infty$, for which the left hand side of \eqref{BoundaryIntegralEqn} converges to $0$.
On the other hand, we decompose
\begin{align*}
2\partial_r \tilde{\kappa}=(2\partial_r\tilde{\kappa}+r\tilde{\kappa})-r(\tilde{\kappa}-\mathscr{F}_{\mathbf{w}}[\tilde{\kappa}])-r\mathscr{F}_{\mathbf{w}}[\tilde{\kappa}].
\end{align*} 
Using the Cauchy-Schwarz inequality,
\begin{align*}
& e^{\frac{\rho^2}{4}}\int_{S_\rho} |u(2\partial_r\tilde{\kappa}+r\tilde{\kappa})| |\nabla_\Sigma r|^{-1} \, d\mathcal{H}^{n-1} \\
& \leq 2\left(\rho^{-n-1}\int_{S_\rho} u^2 \, d\mathcal{H}^{n-1}\right)^{\frac{1}{2}} \left(\rho^{n+1}e^{\frac{\rho^2}{2}}\int_{S_\rho} (2\partial_r\tilde{\kappa}+r\tilde{\kappa})^2 \, d\mathcal{H}^{n-1}\right)^{\frac{1}{2}},
\end{align*}
and likewise
\begin{align*}
& \rho e^{\frac{\rho^2}{4}}\int_{S_\rho} |u(\tilde{\kappa}-\mathscr{F}_{\mathbf{w}}[\tilde{\kappa}])| |\nabla_\Sigma r|^{-1} \, d\mathcal{H}^{n-1} \\
& \leq 2\left(\rho^{-n-1}\int_{S_\rho} u^2 \, d\mathcal{H}^{n-1}\right)^{\frac{1}{2}} \left(\rho^{n+3}e^{\frac{\rho^2}{2}}\int_{S_\rho} (\tilde{\kappa}-\mathscr{F}_{\mathbf{w}}[\tilde{\kappa}])^2 \, d\mathcal{H}^{n-1}\right)^{\frac{1}{2}}.
\end{align*}
Hence, Proposition \ref{AsymptoticKerProp} gives that the right hand side of \eqref{BoundaryIntegralEqn} converges to
$$
-\frac{1}{2} \int_{\mathcal{L}(\Sigma)} T_{\mathbf{v}}^\epsilon[\kappa]\mathrm{tr}_\infty^*[\tilde{\kappa}] \mathbf{w}\cdot\mathbf{n}_{\mathcal{L}(\Sigma)} \, d\mathcal{H}^{n-1}.
$$
As $\mathbf{w}$ is transverse to $\mathcal{L}(\Sigma)$, we may assume 
$$
\inf_{\mathcal{L}(\Sigma)} \mathbf{w}\cdot\mathbf{n}_{\mathcal{L}(\Sigma)}=\delta>0.
$$
Therefore, by Corollary \ref{MonomorphismCor} and injectivity of $\mathrm{tr}_\infty^*$, 
$$
\left\vert \int_{\mathcal{L}(\Sigma)} T_{\mathbf{v}}^\epsilon[\kappa]\mathrm{tr}_\infty^*[\tilde{\kappa}] \mathbf{w}\cdot\mathbf{n}_{\mathcal{L}(\Sigma)} \, d\mathcal{H}^{n-1} \right\vert
\geq  (\delta-\epsilon)\Vert\mathrm{tr}_\infty^*[\tilde{\kappa}]\Vert^2_{L^2}>0,
$$
which gives a contradiction.
\end{proof}

\section{Structure of the space of asymptotically conical self-expanders} \label{StructureSec}
For this section we fix an element $\Gamma\in\mathcal{ACH}^{k,\alpha}_n$. Let $\mathbf{f}\in\mathcal{ACH}^{k,\alpha}_n(\Gamma)$ be a $E$-stationary map, that is, $\Sigma=\mathbf{f}(\Gamma)$ is a self-expander. By Item \eqref{ACHItem} of Proposition \ref{AsympConicalProp}, $\Sigma\in\mathcal{ACH}^{k,\alpha}_n$ and $\mathcal{C}(\Sigma)=\mathcal{C}[\mathbf{f}](\mathcal{C}(\Gamma))$. Let $\mathbf{v}\in C^{k,\alpha}_0\cap C^k_{0,\mathrm{H}}(\Sigma;\mathbb{R}^{n+1})$ be a transverse section on $\Sigma$ so that $\mathbf{w}=\mathcal{C}[\mathbf{v}]$ is a transverse section on $\mathcal{C}(\Sigma)$ and satisfying \eqref{RegVectorEqn}. Next we define 
$$
\mathbf{H}\colon \mathcal{ACH}^{k,\alpha}_n(\Sigma)\to C^{k-2,\alpha}_{-1}(\Sigma;\mathbb{R}^{n+1})
$$ 
by $\mathbf{H}[\mathbf{g}](p)=\mathbf{H}_{\mathbf{g}(\Sigma)}(\mathbf{g}(p))$. Likewise, define $\mathbf{n}$ and $\mathbf{x}^\perp$. Finally, given a Banach space $X$, denote by $\mathcal{B}_R(p;X)$ the (open) ball in the space with center $p$ and radius $R$. 

\subsection{{Smooth} dependence theorem} \label{DependenceSubsec}
The goal here is to establish an analog of \cite[Theorem 3.2]{WhiteEI} in the asymptotically conical $E$-stationary setting.

\begin{thm} \label{SmoothDependThm}
There exist {smooth} maps
\begin{align*}
F_{\mathbf{v}}\colon \mathcal{U}_1\times\mathcal{U}_2 & \to \mathcal{ACH}_{n}^{k,\alpha}(\Sigma), \mbox{ and}, \\
G_{\mathbf{v}}\colon \mathcal{U}_1\times\mathcal{U}_2 & \to \mathcal{K}_{\mathbf{v}},
\end{align*}
where $\mathcal{U}_1$ is some neighborhood of $\mathbf{x}|_{\mathcal{L}(\Sigma)}$ in $C^{k,\alpha}(\mathcal{L}(\Sigma); \mathbb{R}^{n+1})$ and $\mathcal{U}_2$ is some neighborhood of $0$ in $\mathcal{K}_{\mathbf{v}}$, such that the following hold:
\begin{enumerate}
\item \label{TraceFItem} For $(\varphi,\kappa)\in\mathcal{U}_1\times\mathcal{U}_2$, $\mathrm{tr}_\infty^1[F_{\mathbf{v}}[\varphi,\kappa]]=\varphi$. 
\item \label{IdentityItem} $F_{\mathbf{v}}[\mathbf{x}|_{\mathcal{L}(\Sigma)},0]=\mathbf{x}|_\Sigma$. 
\item \label{EStationaryItem} For $(\varphi,\kappa)\in\mathcal{U}_1\times\mathcal{U}_2$, $F_{\mathbf{v}}[\varphi,\kappa]$ is $E$-stationary if and only if $G_{\mathbf{v}}[\varphi,\kappa]=0$. 
\item \label{FDiffItem} For $\kappa\in\mathcal{K}_{\mathbf{v}}$, $D_2F_{\mathbf{v}}(\mathbf{x}|_{\mathcal{L}(\Sigma)},0)\kappa=\kappa\mathbf{v}$.
\item \label{SubmanifoldItem} $G_{\mathbf{v}}^{-1}(0)$ is a {smooth} submanifold of codimension equal to $\dim \mathcal{K}_{\mathbf{v}}$. It contains $\{0\}\times\mathcal{K}_{\mathbf{v}}$ in its tangent space at $(\varphi,0)$. Equivalently, $D_1 G_{\mathbf{v}}(\mathbf{x}|_{\mathcal{L}(\Sigma)},0)$ is of rank equal to $\dim \mathcal{K}_{\mathbf{v}}$ and $D_2 G_{\mathbf{v}}(\mathbf{x}|_{\mathcal{L}(\Sigma)},0)=0$. 
\item \label{UniqueItem} Given $\epsilon>0$ there is a neighborhood $\mathcal{W}$ of $\mathbf{x}|_\Sigma\in \mathcal{ACH}^{k,\alpha}_n(\Sigma)$ such that for any $E$-stationary element $\mathbf{g}\in\mathcal{W}$ there is a $\kappa\in\mathcal{K}_{\mathbf{v}}$ with $\Vert \kappa \Vert_{k,\alpha}^* \leq \epsilon$ and a $C^{k,\alpha}$ diffeomorphism $\phi$ of $\Sigma$ with $\mathbf{x}|_\Sigma\circ\phi\in\mathcal{ACH}_{n}^{k,\alpha}(\Sigma)$ and $\mathrm{tr}_\infty^1[\mathbf{x}|_\Sigma\circ\phi]=\mathbf{x}|_{\mathcal{L}(\Sigma)}$ such that 
$$
\mathbf{g}=F_{\mathbf{v}}[\mathrm{tr}_\infty^1[\mathbf{g}],\kappa]\circ\phi. 
$$
That is, $\mathbf{g}\sim F_{\mathbf{v}}[\mathrm{tr}_\infty^1[\mathbf{g}],\kappa]$ in the sense of Section \ref{ACESubsec}.
\end{enumerate}

Furthermore, we can choose
$$
\mathcal{U}_1=\tilde{\mathcal{U}}_1\cap C^{k,\alpha}(\mathcal{L}(\Sigma); \mathbb{R}^{n+1}) \mbox{ and } \mathcal{W}=\tilde{\mathcal{W}}\cap C^{k,\alpha}_1(\Sigma;\mathbb{R}^{n+1}),
$$
where $\tilde{\mathcal{U}}_1$ and $\tilde{\mathcal{W}}$ are open sets of $C^{k,\tilde{\alpha}}(\mathcal{L}(\Sigma);\mathbb{R}^{n+1})$ and $C^{k,\tilde{\alpha}}_1(\Sigma;\mathbb{R}^{n+1})$, respectively, for $0<\tilde{\alpha}<\alpha$.
\end{thm}

To prove Theorem \ref{SmoothDependThm} we need several auxiliary lemmata. Let $g_{e}$ be the Euclidean metric on $\mathbb{R}^{n+1}$. If $\mathbf{h}\colon\Sigma\to\mathbb{R}^{n+1}$ is a $C^1$ embedding, then we denote by $g_{\mathbf{h}}=\mathbf{h}^*g_{e}$ the pull-back metric of the Euclidean one by $\mathbf{h}$.

\begin{lem} \label{SimpleExpanderMCLem}
If $\mathbf{h}\colon\Sigma\to\mathbb{R}^{n+1}$ is a $C^2$ embedding, then 
$$
\mathbf{v}\cdot\left(\mathbf{H}-\frac{\mathbf{x}^\perp}{2}\right)[\mathbf{h}]=\left(\mathbf{v}\cdot\mathbf{n}[\mathbf{h}]\right)\left(\mathscr{L}_\Sigma\mathbf{h}+\sum_{i,j=1}^n(g_{\mathbf{h}}^{-1}-g_\Sigma^{-1})^{ij}(\nabla_\Sigma^2\mathbf{h})_{ij}\right)\cdot\mathbf{n}[\mathbf{h}].
$$
\end{lem}

\begin{proof}
By a standard differential geometric fact (see, for instance, \cite[Appendix A]{Ecker}) 
\begin{equation} \label{PullbackMCEqn}
\begin{split}
\mathbf{H}[\mathbf{h}] & =\mathbf{n}[\mathbf{h}]\cdot\left(\sum_{i,j=1}^n(g_{\mathbf{h}}^{-1})^{ij}(\nabla_\Sigma^2\mathbf{h})_{ij}\right)\mathbf{n}[\mathbf{h}] \\
& =\mathbf{n}[\mathbf{h}]\cdot\left(\Delta_\Sigma\mathbf{h}+\sum_{i,j=1}^n (g_{\mathbf{h}}^{-1}-g_\Sigma^{-1})^{ij} (\nabla^2_\Sigma\mathbf{h})_{ij}\right)\mathbf{n}[\mathbf{h}].
\end{split}
\end{equation}
Observe that
$$
(\mathbf{x}\cdot\nabla_\Sigma\mathbf{h})(p)=(D\mathbf{h})_p(\mathbf{x}^\top)\in T_{\mathbf{h}(p)}\mathbf{h}(\Sigma).
$$
Thus, it follows that
\begin{equation} \label{PullbackDriftEqn}
\mathbf{n}[\mathbf{h}]\cdot\mathbf{h}=-\mathbf{n}[\mathbf{h}]\cdot\left(\mathbf{x}\cdot\nabla_\Sigma\mathbf{h}-\mathbf{h}\right).
\end{equation}
Hence, the lemma follows by combining \eqref{PullbackMCEqn} and \eqref{PullbackDriftEqn}.
\end{proof}

By Corollary \ref{ExtensionCor} we can define 
$$
\mathscr{E}_\Sigma\colon C^{k,\alpha}(\mathcal{L}(\Sigma); \mathbb{R}^{n+1}) \to C^{k,\alpha}_{1}\cap C^{k}_{1,\mathrm{H}}(\Sigma; \mathbb{R}^{n+1})
$$
such that $\mathscr{E}_\Sigma[\varphi]$ is the (unique) solution to $\mathscr{L}_\Sigma \mathbf{u}=\mathbf{0}$ with $\mathrm{tr}_\infty^1[\mathbf{u}]=\varphi$ given by Corollary \ref{ExtensionCor}. Moreover, $\mathscr{E}_\Sigma$ is an isomorphism.

\begin{lem} \label{ExpanderMCLem}
There is an $r_0>0$ sufficiently small, depending only on $\Sigma,\mathbf{v},n,k$ and $\alpha$, such that the map 
$$
\Xi_{\mathbf{v}} \colon \mathcal{B}_{r_0}(\mathbf{x}|_{\mathcal{L}(\Sigma)}; C^{k,\alpha}(\mathcal{L}(\Sigma); \mathbb{R}^{n+1}))\times \mathcal{B}_{r_0}(0; \mathcal{D}^{k,\alpha}(\Sigma)) \to C^{k-2,\alpha}_{-1}(\Sigma)
$$
defined by 
$$
\Xi_{\mathbf{v}}[\varphi,u]=\mathbf{v}\cdot\left(\mathbf{H}-\frac{\mathbf{x}^\perp}{2}\right)\left[\mathbf{x}|_\Sigma+\mathscr{E}_\Sigma[\varphi-\mathbf{x}|_{\mathcal{L}(\Sigma)}]+u\mathbf{v}\right]
$$
is a smooth map.
\end{lem}

\begin{proof}
First observe that there is an $\epsilon>0$ sufficiently small, depending on $\Sigma,\mathbf{v},n,k$ and $\alpha$, so that if 
\begin{equation} \label{SmallnessEqn}
\Vert\varphi-\mathbf{x}|_{\mathcal{L}(\Sigma)}\Vert_{k,\alpha}+\Vert u\Vert^*_{k,\alpha}<\epsilon,
\end{equation}
then 
$$
\mathbf{h}_{\varphi,u}=\mathbf{x}|_\Sigma+\mathscr{E}_\Sigma[\varphi-\mathbf{x}|_{\mathcal{L}(\Sigma)}]+u\mathbf{v} \in\mathcal{ACH}^{k,\alpha}_n(\Sigma).
$$

Fix any pair $(\varphi,u)$ satisfying \eqref{SmallnessEqn}. For simplicity, we write $\mathbf{h}=\mathbf{h}_{\varphi,u}$. By Item \eqref{ACHItem} of Proposition \ref{AsympConicalProp}, $\Lambda=\mathbf{h}(\Sigma)\in\mathcal{ACH}^{k,\alpha}_n$. As $\mathbf{n}[\mathbf{h}](p)=\mathbf{n}_\Lambda(\mathbf{h}(p))$, it follows from Item \eqref{ListProp4} of Proposition \ref{ListProp} that $\mathbf{n}[\mathbf{h}]\in C^{k-1,\alpha}_0(\Sigma;\mathbb{R}^{n+1})$. Similarly, one proves that $g_{\mathbf{h}}^{-1}-g_{\Sigma}^{-1}\in C^{k-1,\alpha}_0(\Sigma; T^{(0,2)}\Sigma)$. We show $\mathscr{L}_\Sigma\mathbf{h}\in C^{k-2,\alpha}_{-1}(\Sigma;\mathbb{R}^{n+1})$, which, together with Lemma \ref{SimpleExpanderMCLem} and Item \eqref{ListProp2} of Proposition \ref{ListProp}, implies that $\Xi_{\mathbf{v}}[\varphi,u]\in C^{k-2,\alpha}_{-1}(\Sigma)$. 

By our hypothesis that $\Sigma$ is a self-expander and the definition of $\mathscr{E}_\Sigma$,
$$
\mathscr{L}_\Sigma\left(\mathbf{x}|_\Sigma+\mathscr{E}_\Sigma[\varphi-\mathbf{x}|_{\mathcal{L}(\Sigma)}]\right)=\mathbf{0}.
$$
It remains only to show that $\mathscr{L}_\Sigma(u\mathbf{v})\in C^{k-2,\alpha}_{-1}(\Sigma;\mathbb{R}^{n+1})$. The product rule gives
$$
\mathscr{L}_\Sigma(u\mathbf{v})=\mathbf{v}\mathscr{L}_\Sigma u+2\nabla_\Sigma u\cdot\nabla_\Sigma\mathbf{v}+u\left(\mathscr{L}_\Sigma+\frac{1}{2}\right)\mathbf{v}.
$$
Thus the claim follows from that $u\in\mathcal{D}^{k,\alpha}(\Sigma)$, the hypotheses on $\mathbf{v}$ and Item \eqref{ListProp2} of Proposition \ref{ListProp}.

Finally, as $\mathbf{n}[\mathbf{h}]$ and $g^{-1}_{\mathbf{h}}-g^{-1}_\Sigma$ are rational functions of $\nabla_\Sigma\mathbf{h}$, the maps $\mathbf{h}\mapsto\mathbf{n}[\mathbf{h}]$ and $\mathbf{h}\mapsto (g_{\mathbf{h}}^{-1}-g_\Sigma^{-1})$ are smooth maps from $\mathcal{ACH}_n^{k,\alpha}(\Sigma)$ to, respectively, $C^{k-1,\alpha}_0(\Sigma;\mathbb{R}^{n+1})$ and $C^{k-1,\alpha}_0(\Sigma;T^{(0,2)}\Sigma)$, and so the smoothness of $\Xi_{\mathbf{v}}$ follows easily.
\end{proof}

We will need the following self-improving regularity theorem at several points in the argument. First, we prove a slight generalization of Theorem \ref{AsymptoticDirichletThm} where we allow the coefficient matrix of the Laplacian on $\Sigma$ to be a small perturbation of the metric. 

\begin{prop} \label{PerturbAsympDirichletProp}
There is an $\epsilon_1>0$ sufficiently small, depending on $\Sigma,n,k$ and $\alpha$, such that if $\mathbf{M}\in C^{k-2,\alpha}_0(\Sigma;T^{(0,2)}\Sigma)$ satisfies $\Vert\mathbf{M}\Vert_{k-2,\alpha}^{(0)}<\epsilon_1$, then for every $f\in C^{k-2,\alpha}_{-1}(\Sigma)$ there is a unique solution $u\in\mathcal{D}^{k,\alpha}(\Sigma)$ of 
$$
\mathscr{L}_{\Sigma,\mathbf{M}} u=\mathscr{L}_\Sigma u+\sum_{i,j=1}^n\mathbf{M}^{ij}(\nabla_\Sigma^2 u)_{ij}=f.
$$
Moreover, there is a constant $C_7>0$ depending on $\Sigma,n,k$ and $\alpha$ so that 
\begin{equation} \label{PerturbSchauderEstEqn}
\Vert u\Vert^*_{k,\alpha} \leq C_7 \Vert f\Vert^{(-1)}_{k-2,\alpha}.
\end{equation}
\end{prop}

\begin{proof}
When $\mathbf{M}$ vanishes identically this was proved in Theorem \ref{AsymptoticDirichletThm}. We can prove it for the general case by the method of continuity (see, for instance, \cite[Theorem 5.2]{GilbargTrudinger}), provided we can establish that the a priori estimate \eqref{PerturbSchauderEstEqn} holds.

First observe that by taking $\epsilon_1$ sufficiently small, depending only on $n$, we can ensure that the symbol of the operator $\mathscr{L}_{\Sigma,\mathbf{M}}$ is uniformly elliptic. Here we use the elementary fact that we may decompose $\mathbf{M}=\mathbf{M}_s+\mathbf{M}_a$ where $\mathbf{M}_s$ is symmetric and $\mathbf{M}_a$ is skew-symmetric, and $\Vert\mathbf{M}_s\Vert_{k-2,\alpha}\leq\Vert\mathbf{M}\Vert_{k-2,\alpha}$ and one has, due to the symmetry of $\nabla_\Sigma^2$, that $\mathscr{L}_{\Sigma,\mathbf{M}}=\mathscr{L}_{\Sigma,\mathbf{M}_s}$. Hence, by the elliptic maximum principle, the only $u\in\mathcal{D}^{k,\alpha}(\Sigma)$ in the kernel of $\mathscr{L}_{\Sigma,\mathbf{M}}$ is the zero function.

Next observe that, if $u\in\mathcal{D}^{k,\alpha}(\Sigma)$, then
$$
\sum_{i,j=1}^n\mathbf{M}^{ij}(\nabla_\Sigma^2 u)_{ij}\in C^{k-2,\alpha}_{-1}(\Sigma),
$$
and as $u$ solves
$$
\mathscr{L}_\Sigma u=f-\sum_{i,j=1}^n \mathbf{M}^{ij} (\nabla_\Sigma^2 u)_{ij},
$$
Theorem \ref{AsymptoticDirichletThm} implies 
\begin{align*}
 \Vert u\Vert_{k,\alpha}^* & \leq C_4 \Vert f-\sum_{i,j=1}^n \mathbf{M}^{ij} (\nabla^2_\Sigma u)_{ij}\Vert_{k-2,\alpha}^{(-1)} \\
& \leq C_4 \Vert f\Vert_{k-2,\alpha}^{(-1)}+2C_4\Vert\mathbf{M}\Vert_{k-2,\alpha}^{(0)}\Vert\nabla^2_\Sigma u\Vert^{(-1)}_{k-2,\alpha} \\
 & \leq C_4\Vert f\Vert_{k-2,\alpha}^{(-1)}+2C_4\epsilon_1\Vert u\Vert_{k,\alpha}^*.
\end{align*}
When $\epsilon_1=\frac{1}{4C_4}$, 
$$
\Vert u \Vert_{k,\alpha}^* \leq 2C_4 \Vert f\Vert_{k-2,\alpha}^{(-1)},
$$
and so the desired estimate holds with $C_7=2C_4$.
\end{proof}

\begin{lem} \label{ImproveRegLem}
Let $0<\tilde{\alpha}<\alpha$. There is an $r_1>0$ sufficiently small and a $C_8>0$, depending only on $\Sigma,\mathbf{v},n,k,\alpha$ and $\tilde{\alpha}$, with the following significance. Let $\mathbf{g}\in\mathcal{ACH}^{k,\alpha}_n(\Sigma)$ satisfy
$$
\Vert\mathbf{g}-\mathbf{x}|_\Sigma\Vert_{k-1,\alpha}^{(1)}<r_1 \mbox{ and } \mathscr{L}_\Sigma\mathbf{g}\in C^{k-2,\alpha}_{-1}(\Sigma;\mathbb{R}^{n+1}),
$$
and let 
$$
\Xi_{\mathbf{g},\mathbf{v}}\colon \mathcal{D}^{k,\tilde{\alpha}}(\Sigma)\cap\mathcal{B}_{r_1}(0; C^{k-1,\alpha}_{1}(\Sigma)) \to C_{-1}^{k-2,\tilde{\alpha}}(\Sigma)
$$
be the map defined by 
$$
\Xi_{\mathbf{g},\mathbf{v}}[u]=\mathbf{v}\cdot\left(\mathbf{H}-\frac{\mathbf{x}^\perp}{2}\right)[\mathbf{g}+u\mathbf{v}].
$$
The following statements hold:
\begin{enumerate}
\item \label{NonlinearItem} If $\Xi_{\mathbf{g},\mathbf{v}}[u]\in C^{k-2,\alpha}_{-1}(\Sigma)$, then $u\in\mathcal{D}^{k,\alpha}(\Sigma)$, with the estimate
$$
\Vert u \Vert_{k,\alpha}^{*}\leq C_8\left(\Vert\Xi_{\mathbf{g},\mathbf{v}}[u]\Vert_{k-2,{\alpha}}^{(-1)}+\Vert\mathbf{g}-\mathbf{x}|_\Sigma\Vert_{k,\alpha}^{(1)}+\Vert\mathscr{L}_\Sigma \mathbf{g}\Vert_{k-2,\alpha}^{(-1)}+\Vert u\Vert_{k-1,\alpha}^{(1)}\right).
$$
\item \label{LinearDiffItem} If $D\Xi_{\mathbf{g},\mathbf{v}}(0)u\in C^{k-2,\alpha}_{-1}(\Sigma)$ for $u\in\mathcal{D}^{k,\tilde{\alpha}}(\Sigma)$, then $u\in\mathcal{D}^{k,\alpha}(\Sigma)$.
\end{enumerate}
\end{lem}

\begin{proof}
First, we remark that all the constants below depend only on $\Sigma,\mathbf{v},n,k,\alpha$ and $\tilde{\alpha}$. By our hypotheses on $\mathbf{v}$, 
$$
\inf_{p\in\Sigma} |\mathbf{v}\cdot\mathbf{n}_\Sigma|(p)=2\delta>0.
$$ 
One may take $r_1>0$ sufficiently small so that if $\mathbf{g}\in\mathcal{ACH}^{k,\alpha}_n(\Sigma)$, $u\in\mathcal{D}^{k,\tilde{\alpha}}(\Sigma)$ and
$$
\Vert\mathbf{g}-\mathbf{x}|_\Sigma\Vert_{k-1,\alpha}^{(1)}+\Vert u\Vert_{k-1,\alpha}^{(1)}<r_1,
$$
then $\mathbf{h}=\mathbf{g}+u\mathbf{v}\in\mathcal{ACH}^{k,\tilde{\alpha}}_n(\Sigma)$, 
$$
\inf_{p\in\Sigma} |\mathbf{v}\cdot\mathbf{n}[\mathbf{h}]|(p)\geq\delta \mbox{ and } \Vert g_{\mathbf{h}}^{-1}-g_\Sigma^{-1}\Vert_{k-2,\alpha}^{(0)}<\epsilon_1(\Sigma,n,k,\alpha),
$$
where we used the fact that
\begin{equation} \label{PerturbMetricEqn}
\Vert g_{\mathbf{h}}^{-1}-g_\Sigma^{-1}\Vert^{(0)}_{k-2,\alpha} \leq C \Vert\mathbf{h}-\mathbf{x}|_\Sigma\Vert_{k-1,\alpha}^{(1)} \leq C \left(\Vert\mathbf{g}-\mathbf{x}|_\Sigma\Vert^{(1)}_{k-1,\alpha}+\Vert u\Vert^{(1)}_{k-1,\alpha}\right).
\end{equation}
Moreover, as shown in Lemma \ref{ExpanderMCLem}, $\mathscr{L}_\Sigma (u\mathbf{v})\in C^{k-2,\tilde{\alpha}}_{-1}(\Sigma;\mathbb{R}^{n+1})$ and, by our hypotheses on $\mathbf{g}$, so does $\mathscr{L}_\Sigma\mathbf{h}$. Hence, invoking Lemma \ref{SimpleExpanderMCLem} gives the map $\Xi_{\mathbf{g},\mathbf{v}}$ is well defined.

Let us set 
$$
\mathbf{M}_{\mathbf{g},\mathbf{v}}[u]=g^{-1}_{\mathbf{h}}-g_\Sigma^{-1} \mbox{ and } P_{\mathbf{g},\mathbf{v}}[u]=(\mathbf{v}\cdot\mathbf{n}[\mathbf{h}])^{-1}.
$$
By Lemma \ref{SimpleExpanderMCLem} and the product rule,
\begin{equation} \label{LinearExpanderEqn}
\mathscr{L}_\Sigma u+\sum_{i,j=1}^n \mathbf{M}_{\mathbf{g},\mathbf{v}}[u]^{ij} (\nabla_\Sigma^2 u)_{ij}=P_{\mathbf{g},\mathbf{v}}[u]^{2}\Xi_{\mathbf{g},\mathbf{v}}[u]-P_{\mathbf{g},\mathbf{v}}[u]Q_{\mathbf{g},\mathbf{v}}[u],
\end{equation}
where 
\begin{align*}
Q_{\mathbf{g},\mathbf{v}}[u] & =\left(\mathscr{L}_\Sigma\mathbf{g}+u\left(\mathscr{L}_\Sigma+\frac{1}{2}\right)\mathbf{v}+2\nabla_\Sigma u\cdot\nabla_\Sigma\mathbf{v}\right)\cdot\mathbf{n}[\mathbf{h}] \\
& \quad +\left(\sum_{i,j=1}^n \mathbf{M}_{\mathbf{g},\mathbf{v}}[u]^{ij}(\nabla_\Sigma^2(\mathbf{g}-\mathbf{x}|_\Sigma))_{ij}+\sum_{i,j=1}^n \mathbf{M}_{\mathbf{g},\mathbf{v}}[u]^{ij}(\nabla_\Sigma^2\mathbf{x}|_\Sigma)_{ij}\right)\cdot\mathbf{n}[\mathbf{h}] \\
& \quad +\left(\sum_{i,j=1}^n \mathbf{M}_{\mathbf{g},\mathbf{v}}[u]^{ij} \left(u(\nabla_\Sigma^2\mathbf{v})_{ij}+2(\nabla_\Sigma u)_i(\nabla_\Sigma\mathbf{v})_j\right)\right)\cdot\mathbf{n}[\mathbf{h}].
\end{align*}

By Items \eqref{ListProp3}-\eqref{ListProp4} of Proposition \ref{ListProp} and our hypotheses, 
$$
\Vert\mathbf{n}[\mathbf{h}]\Vert^{(0)}_{k-2,\alpha}+\Vert P_{\mathbf{g},\mathbf{v}}[u]\Vert_{k-2,\alpha}^{(0)} \leq C^\prime.
$$
Our hypotheses on $\mathbf{v}$ ensure that
$$
\mathbf{v}\in C^{k,\alpha}_0(\Sigma;\mathbb{R}^{n+1}) \mbox{ and } \left(\mathscr{L}_\Sigma+\frac{1}{2}\right)\mathbf{v}\in C^{k-2,\alpha}_{-2}(\Sigma;\mathbb{R}^{n+1}).
$$
As $\mathbf{g}\in C^{k,\alpha}_1(\Sigma;\mathbb{R}^{n+1})$, we use Item \eqref{ListProp2} of Proposition \ref{ListProp} and \eqref{PerturbMetricEqn} to estimate
$$
\Vert Q_{\mathbf{g},\mathbf{v}}[u]\Vert_{k-2,\alpha}^{(-1)} \leq C^{\prime\prime} \left(\Vert\mathscr{L}_\Sigma\mathbf{g}\Vert_{k-2,\alpha}^{(-1)}+\Vert\mathbf{g}-\mathbf{x}|_\Sigma\Vert^{(1)}_{k,\alpha}+\Vert u\Vert_{k-1,\alpha}^{(1)}\right).
$$
Hence, as we assume $\Xi_{\mathbf{g},\mathbf{v}}[u]\in C^{k-2,\alpha}_{-1}(\Sigma)$, the right hand side of \eqref{LinearExpanderEqn} is an element of $C^{k-2,\alpha}_{-1}(\Sigma)$, and so, by Proposition \ref{PerturbAsympDirichletProp}, $u\in\mathcal{D}^{k,\alpha}(\Sigma)$ and it satisfies the estimate claimed in Item \eqref{NonlinearItem}.

To obtain Item \eqref{LinearDiffItem}, we replace $u$ in \eqref{LinearExpanderEqn} by $su$ for $|s|<\epsilon$ to get that 
$$
s\left(\mathscr{L}_\Sigma u+\sum_{i,j=1}^n \mathbf{M}_{\mathbf{g},\mathbf{v}}[su]^{ij} (\nabla_\Sigma^2 u)_{ij}\right)=P_{\mathbf{g},\mathbf{v}}[su]^{2}\Xi_{\mathbf{g},\mathbf{v}}[su]-P_{\mathbf{g},\mathbf{v}}[su]Q_{\mathbf{g},\mathbf{v}}[su].
$$
By shown in Lemma \ref{ExpanderMCLem} and shrinking $r_1$ if needed, $\Xi_{\mathbf{g},\mathbf{v}}$ is smooth near $0\in\mathcal{D}^{k,\tilde{\alpha}}(\Sigma)$. And careful, but straightforward, analysis shows that our hypotheses ensure that $P_{\mathbf{g},\mathbf{v}}$ and $Q_{\mathbf{g},\mathbf{v}}$ are both smooth maps from a small neighborhood of $0\in \mathcal{D}^{k,\tilde{\alpha}}(\Sigma)$ into, respectively, $C^{k-2,\alpha}_0(\Sigma)$ and $C^{k-2,\alpha}_{-1}(\Sigma)$. Hence, differentiating the above equation at $s=0$ gives that
\begin{align*}
\mathscr{L}_\Sigma u+\sum_{i,j=1}^n\mathbf{M}_{\mathbf{g},\mathbf{v}}[0]^{ij} (\nabla^2_\Sigma u)_{ij} & =2\Xi_{\mathbf{g},\mathbf{v}}[0]P_{\mathbf{g},\mathbf{v}}[0]DP_{\mathbf{g},\mathbf{v}}(0)u+P_{\mathbf{g},\mathbf{v}}[0]^2D\Xi_{\mathbf{g},\mathbf{v}}(0)u \\
& \quad -Q_{\mathbf{g},\mathbf{v}}[0]DP_{\mathbf{g},\mathbf{v}}(0)u-P_{\mathbf{g},\mathbf{v}}[0]DQ_{\mathbf{g},\mathbf{v}}(0)u.
\end{align*}
Now, using that $\mathbf{g}\in C^{k,\alpha}_1(\Sigma;\mathbb{R}^{n+1})$ and $\mathscr{L}_\Sigma\mathbf{g}\in C^{k-2,\alpha}_{-1}(\Sigma;\mathbb{R}^{n+1})$, we compute,
$$
\Xi_{\mathbf{g},\mathbf{v}}[0]=P_{\mathbf{g},\mathbf{v}}[0]^{-1}\left(\mathscr{L}_{\Sigma}\mathbf{g}+\sum_{i,j=1}^n\mathbf{M}_{\mathbf{g},\mathbf{v}}[0]^{ij}(\nabla_\Sigma^2\mathbf{g})_{ij}\right)\cdot\mathbf{n}[\mathbf{g}]\in C^{k-2,\alpha}_{-1}(\Sigma).
$$
Hence, as we assume $D\Xi_{\mathbf{g},\mathbf{v}}(0)u\in C^{k-2,\alpha}_{-1}(\Sigma)$ we have
$$
\mathscr{L}_\Sigma u+\sum_{i,j=1}^n\mathbf{M}_{\mathbf{g},\mathbf{v}}[0]^{ij}(\nabla_\Sigma^2 u)_{ij} \in C^{k-2,\alpha}_{-1}(\Sigma),
$$
and the argument concludes by appealing to Proposition \ref{PerturbAsympDirichletProp}.
\end{proof}

Using the above lemmata we show that near any given asymptotically conical self-expander all other asymptotically conical self-expanders admit a natural parametrization.

\begin{lem} \label{ReparametrizeLem}
Given $\epsilon\in (0,1)$ there is an $r_2\in (0,1)$ depending only on $\Sigma,\mathbf{v},n,k,\alpha$ and $\epsilon$ so that for any $E$-stationary map
$$
\mathbf{g}\in\mathcal{ACH}^{k,\alpha}_n(\Sigma)\cap \mathcal{B}_{r_2}(\mathbf{x}|_\Sigma; C^{k,\alpha}_1(\Sigma;\mathbb{R}^{n+1}))
$$
there is a $u\in\mathcal{B}_\epsilon(0;\mathcal{D}^{k,\alpha}(\Sigma))$ so that $\mathbf{g}(\Sigma)$ can be reparametrized by the map 
$$
\tilde{\mathbf{g}}=\mathbf{x}|_\Sigma+\mathscr{E}_{\Sigma}[\mathrm{tr}_\infty^1[\mathbf{g}-\mathbf{x}|_\Sigma]]+u\mathbf{v}\in\mathcal{ACH}^{k,\alpha}_n(\Sigma).
$$
As such, $\tilde{\mathbf{g}}\sim \mathbf{g}$, where this is the equivalence relation of Section \ref{ACESubsec}.
\end{lem}

\begin{proof}
Let 
$$
\mathbf{h}=\mathbf{x}|_\Sigma+\mathscr{E}_\Sigma[\mathrm{tr}_\infty^1[\mathbf{g}-\mathbf{x}|_\Sigma]].
$$
And let $\Lambda=\mathbf{g}(\Sigma)$ and $\Upsilon=\mathbf{h}(\Sigma)$. If $\mathbf{g}$ lies in a sufficiently small neighborhood of $\mathbf{x}|_\Sigma\in C^{k,\alpha}_{1}(\Sigma; \mathbb{R}^{n+1})$, then, by Corollary \ref{ExtensionCor}, so does $\mathbf{h}$. Thus, $\mathbf{h}\in\mathcal{ACH}^{k,\alpha}_n(\Sigma)$, and so $\Upsilon\in\mathcal{ACH}^{k,\alpha}_n$ by Item \eqref{ACHItem} of Proposition \ref{AsympConicalProp}. Hence, if we set
$$
\hat{\mathbf{g}}=\mathbf{g}\circ\mathbf{h}^{-1} \mbox{ and } \hat{\mathbf{v}}=\mathbf{v}\circ\mathbf{h}^{-1},
$$ 
then, by Items \eqref{ACInverseItem} and \eqref{ACHCCOMPItem} of Proposition \ref{AsympConicalProp}, $\hat{\mathbf{g}}\in\mathcal{ACH}^{k,\alpha}_n(\Upsilon)$, and, by Item \eqref{ListProp4} of Proposition \ref{ListProp}, $\hat{\mathbf{v}}\in C^{k,\alpha}_{0}\cap C^{k}_{0,\mathrm{H}}(\Upsilon;\mathbb{R}^{n+1})$ and it is a transverse section on $\Upsilon$. Observe, that as $\mathrm{tr}_\infty^1[\mathbf{h}]=\mathrm{tr}_\infty^1[\mathbf{g}]$, $\mathcal{C}(\Upsilon)=\mathcal{C}(\Lambda)$ and $\mathrm{tr}_{\infty}^1[\hat{\mathbf{g}}]=\mathbf{x}|_{\mathcal{L}(\Upsilon)}$. Furthermore, by our hypotheses, Item \eqref{ListProp4} of Proposition \ref{ListProp}, Corollary \ref{ExtensionCor} and the triangle inequality, 
\begin{equation} \label{GHatEqn}
\begin{split}
\Vert\hat{\mathbf{g}}-\mathbf{x}|_\Upsilon\Vert_{k,\alpha}^{(1)} & =\Vert{\mathbf{g}}\circ\mathbf{h}^{-1}-\mathbf{h}\circ\mathbf{h}^{-1}\Vert_{k,\alpha}^{(1)}\leq \nu_1 \Vert\mathbf{g}-\mathbf{h}\Vert_{k,\alpha}^{(1)} \\
& \leq \nu_1 (1+CC_5)\Vert\mathbf{g}-\mathbf{x}|_{\Sigma}\Vert_{k,\alpha}^{(1)}\leq \nu_1 (1+CC_5) r_2,
\end{split}
\end{equation}
where $C$ depends on $\Sigma,n,k$ and $\alpha$. Thus, by taking $r_2$ sufficiently small, we can write
$$
\hat{\mathbf{g}}(p)=\pi_{\hat{\mathbf{v}}}(\hat{\mathbf{g}}(p))+\hat{\psi}(p) \hat{\mathbf{v}}(\pi_{\hat{\mathbf{v}}}(\hat{\mathbf{g}}(p))),
$$
so, by Item \eqref{ListProp4} of Proposition \ref{ListProp} and the fact that $\mathrm{tr}_\infty^1[\pi_{\hat{\mathbf{v}}}\circ\hat{\mathbf{g}}-\hat{\mathbf{g}}]=0$,
$$
\hat{\psi}(p)=\left(\hat{\mathbf{g}}(p)-\pi_{\hat{\mathbf{v}}}(\hat{\mathbf{g}}(p))\right)\cdot\hat{\mathbf{v}}(\pi_{\hat{\mathbf{v}}}(\hat{\mathbf{g}}(p)))\in C_{1}^{k,\alpha}\cap C^{k}_{1,0}(\Upsilon).
$$

Since 
$$
\pi_{\hat{\mathbf{v}}}\circ \hat{\mathbf{g}}-\mathbf{x}|_{\Upsilon}=\pi_{\hat{\mathbf{v}}}\circ \hat{\mathbf{g}}-\pi_{\hat{\mathbf{v}}}\circ \mathbf{x}|_{\Upsilon},
$$
it follows from Item \eqref{ListProp5} of Proposition \ref{ListProp} and \eqref{GHatEqn} that for any $\beta\in [0,\alpha)$, 
\begin{equation} \label{ProjectGHatEqn}
\Vert\pi_{\hat{\mathbf{v}}}\circ \hat{\mathbf{g}}-\mathbf{x}|_{\Upsilon}\Vert_{k,\beta}^{(1)} \leq \nu_2 \left(\Vert\hat{\mathbf{g}}-\mathbf{x}|_{\Upsilon}\Vert^{(1)}_{k,\beta}\right)^{\alpha-\beta} \leq \nu_1\nu_2(1+CC_5) r_2^{\alpha-\beta}.
\end{equation}
Hence, by Item \eqref{ListProp2} of Proposition \ref{ListProp}, the triangle inequality and by using \eqref{GHatEqn} and \eqref{ProjectGHatEqn},
\begin{equation} \label{PsiHatEqn}
\Vert\hat{\psi}\Vert_{k,\beta}^{(1)} \leq C^\prime \Vert\hat{\mathbf{g}}-\pi_{\hat{\mathbf{v}}}\circ\hat{\mathbf{g}}\Vert_{k,\beta}^{(1)} \leq C^\prime\nu_1(\nu_2+1)(1+CC_5)r_2^{\alpha-\beta},
\end{equation}
where $C^\prime$ depends on $\Sigma,\mathbf{v}, n,k$ and $\beta$. Moreover, in view of \eqref{ProjectGHatEqn}, for $r_2$ sufficiently small, $\pi_{\hat{\mathbf{v}}}\circ \hat{\mathbf{g}}\in \mathcal{ACH}^{k, \alpha}_n(\Upsilon)$ and so is its inverse by Item \eqref{ACInverseItem} of Proposition \ref{AsympConicalProp}. Thus, setting $u=\hat{\psi}\circ (\pi_{\hat{\mathbf{v}}}\circ\hat{\mathbf{g}})^{-1} \circ \mathbf{h}$ one has
$$
\tilde{\mathbf{g}}=\hat{\mathbf{g}}\circ (\pi_{\hat{\mathbf{v}}}\circ\hat{\mathbf{g}})^{-1} \circ \mathbf{h}=\mathbf{h}+ u \mathbf{v},
$$
and, by Item \eqref{ListProp4} of Proposition \ref{ListProp} and \eqref{PsiHatEqn}, $u\in C^{k,\alpha}_{1}\cap C^{k}_{1,0}(\Sigma)$ satisfies
\begin{equation}
\label{GraphWeightEstEqn}
\Vert u \Vert_{k, \beta}^{(1)}\leq \tilde{C} r_2^{\alpha-\beta}
\end{equation}
for some constant $\tilde{C}$ depending on $\Sigma,\mathbf{v},n,k,\alpha$ and $\beta$.

Next we claim that $u\in\mathcal{D}^{k,\beta}(\Sigma)$, and $\Vert u\Vert_{k,\beta}^* \to 0$ as $r_2 \to 0$. First observe that, by Corollary \ref{ExtensionCor} and \eqref{GraphWeightEstEqn}, for $r_2$ sufficiently small, $\tilde{\mathbf{g}}$ is sufficiently close to $\mathbf{x}|_\Sigma$ in the $C^{k,\beta}_1$ topology, and so $\tilde{\mathbf{g}}\in\mathcal{ACH}^{k,\beta}_{n}(\Sigma)$. Moreover, by Item \eqref{ListProp2} of Proposition \ref{ListProp},
\begin{equation} \label{PerturbNormalMetricEqn}
\Vert\mathbf{n}[\tilde{\mathbf{g}}]-\mathbf{n}_\Sigma\Vert_{k-1,\beta}^{(0)}+\Vert g_{\tilde{\mathbf{g}}}^{-1}-g_\Sigma^{-1}\Vert_{k-1,\beta}^{(0)} \leq C^{\prime\prime} \left(\Vert\mathbf{g}-\mathbf{x}|_\Sigma\Vert_{k,\beta}^{(1)}+\Vert u \Vert_{k,\beta}^{(1)}\right)
\end{equation}
where $C^{\prime\prime}$ depends on $\Sigma,\mathbf{v},n,k$ and $\beta$. It follows from the hypotheses on $\mathbf{v}$ that 
$$
\inf_{p\in\Sigma} |\mathbf{v}\cdot\mathbf{n}[\tilde{\mathbf{g}}]|(p)=\delta>0.
$$ 
If $\mathbf{g}$ is a $E$-stationary map, then so is $\tilde{\mathbf{g}}$. Hence, by Lemma \ref{SimpleExpanderMCLem} and the product rule,
$$
\mathscr{L}_\Sigma u  = -\frac{\mathbf{n}[\tilde{\mathbf{g}}]}{\mathbf{v}\cdot\mathbf{n}[\tilde{\mathbf{g}}]} \cdot \left(2\nabla_\Sigma u\cdot\nabla_\Sigma\mathbf{v}+u\left(\mathscr{L}_\Sigma+\frac{1}{2}\right)\mathbf{v}+\sum_{i,j=1}^n (g_{\tilde{\mathbf{g}}}^{-1}-g_\Sigma^{-1})^{ij} (\nabla_\Sigma^2\tilde{\mathbf{g}})_{ij}\right)
$$
where the right hand side is an element of $C^{k-2,\beta}_{-1}(\Sigma)$ and, by Items \eqref{ListProp3} and \eqref{ListProp2} of Proposition \ref{ListProp}, \eqref{GraphWeightEstEqn} and \eqref{PerturbNormalMetricEqn}, its $C^{k-2,\beta}_{-1}$ norm is bounded by $O(r_2^{\alpha-\beta})$. Hence, invoking Theorem \ref{AsymptoticDirichletThm}, there is a $\tilde{u}\in \mathcal{D}^{k,\beta}(\Sigma)$ satisfying that $\mathscr{L}_\Sigma\tilde{u}=\mathscr{L}_\Sigma u$ and $\Vert \tilde{u}\Vert_{k,\beta}^* \leq O(r_2^{\alpha-\beta})$. As $u,\tilde{u}\in C^{k,\beta}_{1,0}(\Sigma)$, by Corollary \ref{ExtensionCor}, $u=\tilde{u}$ and so $u\in\mathcal{D}^{k,\beta}(\Sigma)$ and $\Vert u\Vert_{k,\beta}^* \to 0$ as $r_2\to 0$.

Finally, as $\mathscr{L}_\Sigma\mathbf{h}=\Xi_{\mathbf{h},\mathbf{v}}[u]=0$, we can appeal to Item \eqref{NonlinearItem} of Lemma \ref{ImproveRegLem} to see that when $r_2$ is sufficiently small, $u\in\mathcal{D}^{k,\alpha}(\Sigma)$ and
$$
\Vert u\Vert_{k,\alpha}^* \leq C_8\left(\Vert\mathbf{h}-\mathbf{x}|_\Sigma\Vert_{k,\alpha}^{(1)}+\Vert u\Vert_{k,\beta}^*\right) \leq \hat{C} r_2^{\alpha-\beta},
$$
where $\hat{C}$ depends on $\Sigma,\mathbf{v},n,k,\alpha$ and $\beta$. Therefore, one completes the proof by taking $r_2$ sufficiently small so that $\hat{C} r_2^{\alpha-\beta}<\epsilon$.
\end{proof}

Now we are ready to prove Theorem \ref{SmoothDependThm} stated in the beginning of this section.

\begin{proof}[Proof of Theorem \ref{SmoothDependThm}]
First, we define
\begin{align*}
\mathcal{K}_{\mathbf{v}}^\perp & =\left\{f\in C^{k-2,\alpha}_{-1}(\Sigma)\colon \int_{\Sigma} f\kappa e^{\frac{r^2}{4}} \, d\mathcal{H}^n=0 \mbox{ for all $\kappa\in\mathcal{K}_{\mathbf{v}}$}\right\}, \mbox{ and}, \\
\mathcal{K}_{\mathbf{v},*}^\perp & =\left\{f\in\mathcal{D}^{k,\alpha}(\Sigma)\colon \int_{\Sigma} f\kappa e^{\frac{r^2}{4}}\, d\mathcal{H}^n=0 \mbox{ for all $\kappa\in\mathcal{K}_{\mathbf{v}}$}\right\}.
\end{align*}
By Lemma \ref{RegJacobiLem} the space $\mathcal{K}_{\mathbf{v}}$ is a finite dimensional subspace of $\mathcal{D}^{k,\alpha}(\Sigma)$, and for any $f\in C^{0}_{-1}(\Sigma)$ and $\kappa\in\mathcal{K}_{\mathbf{v}}$,
$$
\left|\int_{\Sigma} f\kappa e^{\frac{r^2}{4}} \, d\mathcal{H}^n\right| \leq C(\kappa)\Vert f\Vert_{0}^{(-1)}
$$
where, by Proposition \ref{AsymptoticKerProp}, 
$$
C(\kappa)=\int_{\Sigma} |\kappa|(1+r)^{-1} e^{\frac{r^2}{4}} \, d\mathcal{H}^n<\infty.
$$
Thus, we can define the orthogonal projection $\Pi_{\mathcal{K}_{\mathbf{v}}}\colon C^{k-2,\alpha}_{-1}(\Sigma)\to\mathcal{K}_{\mathbf{v}}$ with respect to the measure $e^{r^2/4}\mathcal{H}^n\lfloor\Sigma$, and $\Pi_{\mathcal{K}_{\mathbf{v}}}$ is a bounded linear map. We also consider the projection $\Pi_{\mathcal{K}^\perp_{\mathbf{v}}}\colon C^{k-2,\alpha}_{-1}(\Sigma)\to\mathcal{K}^\perp_{\mathbf{v}}$ defined by $\Pi_{\mathcal{K}^\perp_{\mathbf{v}}}=\mathrm{Id}-\Pi_{\mathcal{K}_{\mathbf{v}}}$,  is also a bounded linear map.

Consider the map
$$
\Theta_{\mathbf{v}}\colon \mathcal{B}_{r_0}(\mathbf{x}|_{\mathcal{L}(\Sigma)}; C^{k,\alpha}(\mathcal{L}(\Sigma); \mathbb{R}^{n+1}))\times\mathcal{B}_{\frac{r_0}{2}}(0; \mathcal{K}_{\mathbf{v}})\times\mathcal{B}_{\frac{r_0}{2}}(0; \mathcal{K}^\perp_{\mathbf{v},*}) \to \mathcal{K}^\perp_{\mathbf{v}}
$$
defined by 
$$
\Theta_{\mathbf{v}}[\varphi,\kappa,u]=\Pi_{\mathcal{K}^\perp_{\mathbf{v}}}\circ\Xi_{\mathbf{v}}[\varphi,\kappa+u].
$$
By Lemma \ref{ExpanderMCLem}, $\Theta_{\mathbf{v}}$ is a well defined, smooth map. Moreover, in view of Propositions \ref{1stVarProp} and \ref{2ndVarProp}, $D_3\Theta_{\mathbf{v}}(\mathbf{x}|_{\mathcal{L}(\Sigma)},0,0)$ is given by $L_{\mathbf{v}}$ restricted to $\mathcal{K}^\perp_{\mathbf{v},*}$. Thus, it follows from Theorem \ref{FredholmThm} and Corollary \ref{CokerCor} that $D_3\Theta_{\mathbf{v}}(\mathbf{x}|_{\mathcal{L}(\Sigma)},0,0)$ is an isomorphism between $\mathcal{K}^\perp_{\mathbf{v},*}$ and $\mathcal{K}^\perp_{\mathbf{v}}$. Hence, by the implicit function theorem, there is a neighborhood $\mathcal{U}=\mathcal{U}_1\times\mathcal{U}_2$ of $(\mathbf{x}|_{\mathcal{L}(\Sigma)},0)$ in $C^{k,\alpha}(\mathcal{L}(\Sigma); \mathbb{R}^{n+1})\times\mathcal{K}_{\mathbf{v}}$ and a {smooth} map $F_{\mathbf{v},*}\colon \mathcal{U}\to\mathcal{K}^\perp_{\mathbf{v},*}$ that gives the unique solution in a neighborhood of $0\in\mathcal{K}_{\mathbf{v},*}^\perp$ of 
$$
\Theta_{\mathbf{v}}[\varphi,\kappa,F_{\mathbf{v},*}[\varphi,\kappa]]=0.
$$

Now let 
\begin{align*}
F_{\mathbf{v}}[\varphi,\kappa] & =\mathbf{x}|_\Sigma+\mathscr{E}_\Sigma[\varphi-\mathbf{x}|_{\mathcal{L}(\Sigma)}]+\left(\kappa+F_{\mathbf{v},*}[\varphi,\kappa]\right)\mathbf{v}, \mbox{ and}, \\
G_{\mathbf{v}}[\varphi,\kappa] & =\Pi_{\mathcal{K}_{\mathbf{v}}}\circ\Xi_{\mathbf{v}}[\varphi,\kappa+F_{\mathbf{v},*}[\varphi,\kappa]].
\end{align*}
Then Items \eqref{TraceFItem} and \eqref{IdentityItem} follows immediately. By the definitions of $F_{\mathbf{v}}$ and $\Xi_{\mathbf{v}}$,
\begin{equation} \label{FDefnEqn}
\Pi_{\mathcal{K}^\perp_{\mathbf{v}}}\left[\mathbf{v}\cdot\left(\mathbf{H}-\frac{\mathbf{x}^\perp}{2}\right)[F_{\mathbf{v}}[\varphi,\kappa]]\right]=0.
\end{equation}
Thus, $F_{\mathbf{v}}[\varphi,\kappa]$ is $E$-stationary if and only if 
$$
\Pi_{\mathcal{K}_{\mathbf{v}}}\left[\mathbf{v}\cdot\left(\mathbf{H}-\frac{\mathbf{x}^\perp}{2}\right)[F_{\mathbf{v}}[\varphi,\kappa]]\right]=0,
$$
i.e., $G_{\mathbf{v}}[\varphi,\kappa]=0$. This proves Item \eqref{EStationaryItem}.

To establish Item \eqref{FDiffItem}, it suffices to show $D_2F_{\mathbf{v},*}(\mathbf{x}|_{\mathcal{L}(\Sigma)},0)=0$. Take any $\kappa\in\mathcal{K}_{\mathbf{v}}$. Observe that for $|s|<\epsilon$,
\begin{equation} \label{F*DiffEqn}
\Pi_{\mathcal{K}^\perp_{\mathbf{v}}}\left[\mathbf{v}\cdot\left(\mathbf{H}-\frac{\mathbf{x}^\perp}{2}\right)\left[\mathbf{x}|_\Sigma+\left(s\kappa+F_{\mathbf{v},*}[\mathbf{x}|_{\mathcal{L}(\Sigma)},s\kappa]\right)\mathbf{v}\right]\right]=0.
\end{equation}
Now, differentiating \eqref{F*DiffEqn} at $s=0$, we apply Proposition \ref{2ndVarProp} and the chain rule to get
$$
\Pi_{\mathcal{K}^\perp_{\mathbf{v}}}\circ L_{\mathbf{v}}\circ D_2F_{\mathbf{v},*}(\mathbf{x}|_{\mathcal{L}(\Sigma)},0)\kappa=0.
$$
Since $\mathcal{K}^\perp_{\mathbf{v}}=\mathrm{Im}(L_{\mathbf{v}})$ by Theorem \ref{FredholmThm} and Corollary \ref{CokerCor}, it follows that 
$$
D_2F_{\mathbf{v},*}(\mathbf{x}|_{\mathcal{L}(\Sigma)},0)\kappa\in\mathcal{K}_{\mathbf{v}}\cap\mathcal{K}^\perp_{\mathbf{v},*}=\{0\}.
$$ 
Thus, the claim follows immediately from the arbitrariness of $\kappa\in\mathcal{K}_{\mathbf{v}}$.

Next, by the definition of $G_{\mathbf{v}}$ and \eqref{FDefnEqn}, we have
$$
G_{\mathbf{v}}[\varphi,\kappa]=\mathbf{v}\cdot\left(\mathbf{H}-\frac{\mathbf{x}^\perp}{2}\right)[F_{\mathbf{v}}[\varphi,\kappa]].
$$
Thus, it follows from Proposition \ref{2ndVarProp} and Item \eqref{FDiffItem} that for all $\kappa\in\mathcal{K}_{\mathbf{v}}$,
$$
D_2G_{\mathbf{v}}(\mathbf{x}|_{\mathcal{L}(\Sigma)},0)\kappa=(\mathbf{v}\cdot\mathbf{n}_\Sigma)L_\Sigma[D_2F_{\mathbf{v}}(\mathbf{x}|_{\mathcal{L}(\Sigma)},0)\kappa\cdot\mathbf{n}_\Sigma]=L_{\mathbf{v}}\kappa=0.
$$
This shows $D_2G_{\mathbf{v}}(\mathbf{x}|_{\mathcal{L}(\Sigma)},0)=0$. Next, take any $\zeta\in T_{\mathbf{v}}^{\epsilon_1}(\mathcal{K}_{\mathbf{v}})\setminus\{0\}$. Then
$$
D_1 G_{\mathbf{v}}(\mathbf{x}|_{\mathcal{L}(\Sigma)},0)[\zeta\mathbf{w}]=(\mathbf{v}\cdot\mathbf{n}_\Sigma)L_\Sigma[D_1 F_{\mathbf{v}}(\mathbf{x}|_{\mathcal{L}(\Sigma)},0)[\zeta\mathbf{w}]\cdot\mathbf{n}_\Sigma].
$$
By the linearity of $\mathrm{tr}_\infty^1$ and Item \eqref{TraceFItem},
$$
\mathrm{tr}_\infty^1[D_1 F_{\mathbf{v}}(\mathbf{x}|_{\mathcal{L}(\Sigma)},0)[\zeta\mathbf{w}]\cdot\mathbf{n}_\Sigma]=\zeta\mathbf{w}\cdot\mathbf{n}_{\mathcal{L}(\Sigma)}.
$$
As $\Sigma\in\mathcal{ACH}^{k,\alpha}_n$, it follows from Corollary \ref{NonExistAsympDirichletCor} that 
$$
D_1 G_{\mathbf{v}}(\mathbf{x}|_{\mathcal{L}(\Sigma)},0)[\zeta\mathbf{w}]\neq 0.
$$
Thus, we appeal to Corollary \ref{MonomorphismCor} to see that $D_1 G_{\mathbf{v}}(\mathbf{x}|_{\mathcal{L}(\Sigma)},0)$ is of rank equal to $\dim \mathcal{K}_{\mathbf{v}}$, which completes the proof of Item \eqref{SubmanifoldItem}. Finally, Item \eqref{UniqueItem} follows from Lemma \ref{ReparametrizeLem} and the implicit function theorem.

It remains only to prove the last statement (``Furthermore, we can choose..."). Fix an $\tilde{\alpha}\in(0,\alpha)$. First observe that, by Lemma \ref{RegJacobiLem}, $\mathcal{K}_{\mathbf{v}}$ is a finite dimensional subspace of $\mathcal{D}^{k,\alpha}(\Sigma)$, so any norms on $\mathcal{K}_{\mathbf{v}}$ are equivalent. Let $\tilde{\mathcal{K}}_{\mathbf{v}}^\perp$ and $\tilde{\mathcal{K}}_{\mathbf{v},*}^\perp$ be the orthogonal complements of $\mathcal{K}_{\mathbf{v}}$ in $C^{k-2,\tilde{\alpha}}_{-1}(\Sigma)$ and $\mathcal{D}^{k,\tilde{\alpha}}(\Sigma)$, respectively. We further let $\Pi_{\tilde{\mathcal{K}}_{\mathbf{v}}^\perp}$ be the orthogonal projection of $C^{k-2,\tilde{\alpha}}_{-1}(\Sigma)$ onto $\tilde{\mathcal{K}}_{\mathbf{v}}^\perp$. As $\Sigma\in\mathcal{ACH}^{k,\alpha}_n\subset\mathcal{ACH}^{k,\tilde{\alpha}}_n$, applying the theorem with $\tilde{\alpha}$ replacing $\alpha$, there is a neighborhood $\tilde{\mathcal{U}}=\tilde{\mathcal{U}}_1\times\tilde{\mathcal{U}}_2$ of $(\mathbf{x}|_{\mathcal{L}(\Sigma)},0)$ in $C^{k,\tilde{\alpha}}(\mathcal{L}(\Sigma); \mathbb{R}^{n+1})\times\mathcal{K}_{\mathbf{v}}$ together with a smooth map $\tilde{F}_{\mathbf{v},*}\colon\tilde{\mathcal{U}}\to \tilde{\mathcal{K}}_{\mathbf{v},*}^\perp$ giving the unique solution in some neighborhood of $0\in\tilde{\mathcal{K}}_{\mathbf{v},*}^\perp$ of 
$$
\tilde{\Theta}_{\mathbf{v}}[\varphi,\kappa,\tilde{F}_{\mathbf{v},*}[\varphi,\kappa]]=\Pi_{\tilde{\mathcal{K}}_{\mathbf{v}}^\perp}\left(\mathbf{v}\cdot\left(\mathbf{H}-\frac{\mathbf{x}^\perp}{2}\right)[\tilde{F}_{\mathbf{v}}[\varphi,\kappa]]\right)=0,
$$
where
$$
\tilde{F}_{\mathbf{v}}[\varphi,\kappa]=\mathbf{x}|_\Sigma+\mathscr{E}_\Sigma[\varphi-\mathbf{x}|_{\mathcal{L}(\Sigma)}]+\left(\kappa+\tilde{F}_{\mathbf{v},*}[\varphi,\kappa]\right)\mathbf{v}.
$$

If $\varphi\in\tilde{\mathcal{U}}_1\cap C^{k,\alpha}(\mathcal{L}(\Sigma); \mathbb{R}^{n+1})$, then $\mathscr{E}_\Sigma[\varphi-\mathbf{x}|_{\mathcal{L}(\Sigma)}]\in C^{k,\alpha}_1(\Sigma; \mathbb{R}^{n+1})$. Let
$$
\mathbf{g}=\mathbf{x}|_\Sigma+\mathscr{E}_{\Sigma}[\varphi-\mathbf{x}|_{\mathcal{L}(\Sigma)}].
$$
Thus, shrinking $\tilde{\mathcal{U}}_1$ if necessary, $\mathbf{g}\in\mathcal{ACH}^{k,\alpha}_n(\Sigma)$, and $\mathscr{L}_\Sigma \mathbf{g}=
\mathbf{0}$ as $\Sigma$ is a self-expander. By the definition of $\tilde{F}_{\mathbf{v},*}$, 
$$
\Xi_{\mathbf{g},\mathbf{v}}[\kappa+\tilde{F}_{\mathbf{v},*}[\varphi,\kappa]]\in\mathcal{K}_{\mathbf{v}}\subset C^{k-2,\alpha}_{-1}(\Sigma).
$$
Hence, shrinking $\tilde{\mathcal{U}}$ so we may apply Item \eqref{NonlinearItem} of Lemma \ref{ImproveRegLem} to get $\tilde{F}_{\mathbf{v},*}[\varphi,\kappa]\in\mathcal{D}^{k,\alpha}(\Sigma)$. Denote by
$$
\mathcal{U}^\prime=\tilde{\mathcal{U}}\cap (C^{k,\alpha}(\mathcal{L}(\Sigma);\mathbb{R}^{n+1})\times\mathcal{K}_{\mathbf{v}}).
$$
Therefore, the map $F_{\mathbf{v},*}^\prime\colon\mathcal{U}^\prime\to\mathcal{K}^\perp_{\mathbf{v},*}$ given by the restriction of $\tilde{F}_{\mathbf{v},*}$ to $\mathcal{U}^\prime$ is well defined. Likewise, define $F_{\mathbf{v}}^\prime\colon\mathcal{U}^\prime\to\mathcal{ACH}^{k,\alpha}_n(\Sigma)$ to be the restriction of $\tilde{F}_{\mathbf{v}}$.

Given $(\varphi,\kappa)\in\mathcal{U}^\prime$, let $u=F^\prime_{\mathbf{v},*}[\varphi,\kappa]$ and $\mathbf{h}=F^\prime_{\mathbf{v}}[\varphi,\kappa]$. Then $D_3\tilde{\Theta}_{\mathbf{v}}(\varphi,\kappa,u)$ is an isomorphism of $\tilde{\mathcal{K}}_{\mathbf{v},*}^\perp$ and $\tilde{\mathcal{K}}_{\mathbf{v}}^\perp$. Observe that $\mathbf{h}\in\mathcal{ACH}^{k,\alpha}_n(\Sigma)$ with $\mathscr{L}_\Sigma\mathbf{h}\in C^{k-2,\alpha}_{-1}(\Sigma;\mathbb{R}^{n+1})$. Since $\Xi_{\mathbf{h},\mathbf{v}}$ restricts to a {smooth} map from a neighborhood of $0\in\mathcal{D}^{k,\alpha}(\Sigma)$ to $C^{k-2,\alpha}_{-1}(\Sigma)$, we have 
$$
D_3\tilde{\Theta}_{\mathbf{v}}(\varphi,\kappa,u)|_{\mathcal{K}^\perp_{\mathbf{v},*}}=\Pi_{\mathcal{K}_{\mathbf{v}}^\perp}\circ D\Xi_{\mathbf{h},\mathbf{v}}(0)|_{\mathcal{K}^\perp_{\mathbf{v},*}}.
$$
Thus, Item \eqref{LinearDiffItem} of Lemma \ref{ImproveRegLem} implies $D_3\tilde{\Theta}_{\mathbf{v}}(\varphi,\kappa,u)$ is an isomorphism of $\mathcal{K}_{\mathbf{v},*}^\perp$ and $\mathcal{K}^\perp_{\mathbf{v}}$. Moreover, by the implicit function theorem, there is a smooth map $F$ from a neighborhood of $(\varphi,\kappa)\in C^{k,\alpha}(\mathcal{L}(\Sigma);\mathbb{R}^{n+1})\times\mathcal{K}_{\mathbf{v}}$ to $\mathcal{K}_{\mathbf{v}}^\perp$ that gives the unique solution in a neighborhood of $u$ of 
$$
\tilde{\Theta}_{\mathbf{v}}[\varphi^\prime,\kappa^\prime,F[\varphi^\prime,\kappa^\prime]]=0.
$$
Hence, $F^\prime_{\mathbf{v},*}$ coincides with $F$ near $(\varphi,\kappa)$. Therefore, $F^\prime_{\mathbf{v},*}$ and $F^\prime_{\mathbf{v}}$ are both smooth maps.

Let $\tilde{G}_{\mathbf{v}}\colon\tilde{\mathcal{U}}\to\mathcal{K}_{\mathbf{v}}$ be given by
$$
\tilde{G}_{\mathbf{v}}[\varphi,\kappa]=\mathbf{v}\cdot\left(\mathbf{H}-\frac{\mathbf{x}^\perp}{2}\right)[\tilde{F}_{\mathbf{v}}[\varphi,\kappa]],
$$
and let $G_{\mathbf{v}}^\prime=\tilde{G}_{\mathbf{v}}|_{\mathcal{U}^\prime}$. Then $\tilde{G}_{\mathbf{v}}$ and $G^\prime_{\mathbf{v}}$ are {smooth} maps. By the theorem (replacing $\alpha$ by $\tilde{\alpha}$), $\tilde{G}_{\mathbf{v}}^{-1}(0)$ is a smooth submanifold of $\tilde{\mathcal{U}}$ at each point of which we may assume $D\tilde{G}_{\mathbf{v}}$ is surjective. Observe, that
$$
(G^\prime_{\mathbf{v}})^{-1}(0)=\tilde{G}^{-1}_{\mathbf{v}}(0)\cap (C^{k,\alpha}(\mathcal{L}(\Sigma);\mathbb{R}^{n+1})\times\mathcal{K}_{\mathbf{v}}).
$$
As $C^{k,\alpha}(\mathcal{L}(\Sigma))$ is dense in $C^{k,\tilde{\alpha}}(\mathcal{L}(\Sigma))$, $DG^\prime_{\mathbf{v}}(\varphi,\kappa)$ is surjective for $(\varphi,\kappa)\in (G^\prime_{\mathbf{v}})^{-1}(0)$. Thus, by the implicit function theorem, $(G^\prime_{\mathbf{v}})^{-1}(0)$ is a smooth submanifold of $\mathcal{U}^\prime$. 

Finally, by the preceding discussions and the fact that the norms $\Vert\cdot\Vert_{k,\alpha}^*$ and $\Vert\cdot\Vert_{k,\tilde{\alpha}}^*$ on $\mathcal{K}_{\mathbf{v}}$ are equivalent, the $\mathcal{W}$ in Item \eqref{UniqueItem} can be taken to be $\tilde{\mathcal{W}}\cap C^{k,\alpha}_1(\Sigma;\mathbb{R}^{n+1})$ for $\tilde{\mathcal{W}}$ open in $C^{k,\tilde{\alpha}}_1(\Sigma;\mathbb{R}^{n+1})$, completing the proof of the last statement.
\end{proof}

\subsection{Global structure theorem} \label{StructureSubsec}
Following the strategy of the proof of \cite[Theorem 3.3]{WhiteEI}, we apply Theorem \ref{SmoothDependThm} to prove Theorem \ref{StructureThm}.

\begin{proof}[Proof of Theorem \ref{StructureThm}]
First, by Item \eqref{FDiffItem} of Theorem \ref{SmoothDependThm}, one has that for any $\kappa\in\mathcal{K}_{\mathbf{v}}\setminus\{0\}$, there is a bounded smooth $n$-form $\omega$ on a neighborhood of $\Sigma\subset\mathbb{R}^{n+1}$ so that 
$$
{\frac{d}{ds}\vline}_{s=0} \int_\Sigma e^{-\frac{r^2}{4}} F_{\mathbf{v}}[\mathbf{x}|_{\mathcal{L}(\Sigma)},s\kappa]^*\omega \neq 0.
$$
Furthermore, one can choose such $n$-forms, $\omega_1,\ldots,\omega_{M}$, where $M=\dim \mathcal{K}_{\mathbf{v}}$, so that the linear map $\mathcal{K}_{\mathbf{v}}\to\mathbb{R}^M$ defined by
$$
\kappa\mapsto{\frac{d}{ds}\vline}_{s=0} \left(\int_\Sigma e^{-\frac{r^2}{4}}F_{\mathbf{v}}[\mathbf{x}|_{\mathcal{L}(\Sigma)},s\kappa]^*\omega_1,\ldots, \int_\Sigma e^{-\frac{r^2}{4}}F_{\mathbf{v}}[\mathbf{x}|_{\mathcal{L}(\Sigma)},s\kappa]^*\omega_M\right)
$$
is an isomorphism. 

Recall, that $\mathbf{f}\in\mathcal{ACH}^{k,\alpha}_n(\Gamma)$ and $\Sigma=\mathbf{f}(\Gamma)$ is a self-expander. Define 
$$
\Phi_{\mathbf{f}}\colon\mathcal{ACH}^{k,\alpha}_n(\Gamma) \to C^{k,\alpha}(\mathcal{L}(\Gamma);\mathbb{R}^{n+1})\times\mathbb{R}^M
$$
given by
$$
\Phi_{\mathbf{f}}(\mathbf{g})=\left(\mathrm{tr}_\infty^1[\mathbf{g}], \int_\Sigma e^{-\frac{r^2}{4}} (\mathbf{g}\circ\mathbf{f}^{-1})^*\omega_1,\ldots,\int_\Sigma e^{-\frac{r^2}{4}} (\mathbf{g}\circ\mathbf{f}^{-1})^*\omega_M\right).
$$
Clearly, $\Phi_{\mathbf{f}}$ is a smooth map. We also define 
$$
C_{\mathbf{f}}\colon\mathcal{ACH}_n^{k,\alpha}(\Sigma)\to\mathcal{ACH}^{k,\alpha}_n(\Gamma), \, C_{\mathbf{f}}[\mathbf{h}]=\mathbf{h}\circ\mathbf{f}.
$$
By Items \eqref{ACInverseItem} and \eqref{ACHCCOMPItem} of Proposition \ref{AsympConicalProp}, $C_{\mathbf{f}}$ is well defined and it is an isomorphism with inverse $C_{\mathbf{f}^{-1}}$.

Now, consider 
$$
\Phi_{\mathbf{f}}\circ C_{\mathbf{f}}\circ F_{\mathbf{v}}\colon \mathcal{U}_1\times\mathcal{U}_2\to C^{k,\alpha}(\mathcal{L}(\Gamma);\mathbb{R}^{n+1})\times\mathbb{R}^{M}
$$
where $\mathcal{U}_1=\tilde{\mathcal{U}}_1\cap C^{k,\alpha}(\mathcal{L}(\Sigma);\mathbb{R}^{n+1})$,  $\tilde{\mathcal{U}}_1$ a neighborhood of $\mathbf{x}|_{\mathcal{L}(\Sigma)} \in C^{k,\tilde{\alpha}}(\mathcal{L}(\Sigma);\mathbb{R}^{n+1})$ and $\mathcal{U}_2$ is a neighborhood of $0\in\mathcal{K}_{\mathbf{v}}$. By our definitions,
\begin{align*}
\Phi_{\mathbf{f}}\circ C_{\mathbf{f}}\circ F_{\mathbf{v}}[\varphi,\kappa]=\left(|\mathrm{tr}_\infty^1[\mathbf{f}]|\varphi\left(\frac{\mathrm{tr}_\infty^1[\mathbf{f}]}{|\mathrm{tr}_\infty^1[\mathbf{f}]|}\right), \int_\Sigma e^{-\frac{r^2}{4}}F_{\mathbf{v}}[\mathbf{x}|_{\mathcal{L}(\Sigma)},\kappa]^*\omega_1,\right. & \\
\left. \ldots, \int_\Sigma e^{-\frac{r^2}{4}}F_{\mathbf{v}}[\mathbf{x}|_{\mathcal{L}(\Sigma)},\kappa]^*\omega_M\right). &
\end{align*}
Thus, by the inverse function theorem, shrinking $\tilde{\mathcal{U}}_1,\tilde{\mathcal{U}}_2$ if necessary, $\Phi_{\mathbf{f}}\circ C_{\mathbf{f}}\circ F_{\mathbf{v}}$ is a smooth diffeomorphism onto its image. By Item \eqref{SubmanifoldItem} of Theorem \ref{SmoothDependThm}, $G_{\mathbf{v}}^{-1}(0)$ is a smooth submanifold of $C^{k,\alpha}(\mathcal{L}(\Sigma);\mathbb{R}^{n+1})\times\mathcal{K}_{\mathbf{v}}$, and so $\mathcal{S}_{\mathbf{f}}=\Phi_{\mathbf{f}}\circ C_{\mathbf{f}}\circ F_{\mathbf{v}}(G_{\mathbf{v}}^{-1}(0))$ is a smooth submanifold of $C^{k,\alpha}(\mathcal{L}(\Gamma);\mathbb{R}^{n+1})\times\mathbb{R}^{M}$ near $\Phi_{\mathbf{f}}[\mathbf{f}]$. Hence, in view of Items \eqref{EStationaryItem} and \eqref{UniqueItem} of Theorem \ref{SmoothDependThm}, $\Phi_{\mathbf{f}}$ sets up a one-to-one correspondence between 
$$
\mathcal{O}_{\mathbf{f}}=\left\{[\mathbf{g}]\in\mathcal{ACE}^{k,\alpha}_n(\Gamma)\colon \mathbf{g}\in\mathcal{W}\right\},
$$
where $\mathcal{W}=\tilde{\mathcal{W}}\cap \mathcal{ACH}^{k,\alpha}_n(\Gamma)$ for $\tilde{\mathcal{W}}$ a neighborhood of $\mathbf{f}\in C^{k,\tilde{\alpha}}_1(\Gamma;\mathbb{R}^{n+1})$, and a neighborhood of $\Phi_{\mathbf{f}}[\mathbf{f}] \in \mathcal{S}_{\mathbf{f}}$. Moreover, if $\mathbf{f}^\prime\in\mathcal{ACH}^{k,\alpha}_n(\Gamma)$ and $\mathbf{f}^\prime(\Gamma)$ is a self-expander, then 
$$
\Phi_{\mathbf{f}^\prime}\circ (\Phi_{\mathbf{f}}|_{\mathcal{O}_{\mathbf{f}}})^{-1}=\Phi_{\mathbf{f}^\prime}\circ C_{\mathbf{f}}\circ F_{\mathbf{v}}\circ (\Phi_{\mathbf{f}}\circ C_{\mathbf{f}} \circ F_{\mathbf{v}})^{-1}|_{\mathcal{S}_{\mathbf{f}}}
$$
is a smooth map. Therefore, the collection of such $(\mathcal{O}_{\mathbf{f}},\Phi_{\mathbf{f}})$ form a smooth atlas for $\mathcal{ACE}^{k,\alpha}_n(\Gamma)$ from which one obtains a countable subcover as $C^{k,\alpha}_1\cap C^{k}_{1,\mathrm{H}}(\Gamma;\mathbb{R}^{n+1})$ and $C^{k,\alpha}(\mathcal{L}(\Gamma);\mathbb{R}^{n+1})$ are separable with respect to the $C^{k,\tilde{\alpha}}_1$ and $C^{k,\tilde{\alpha}}$ topology, respectively. This proves Item \eqref{ManifoldItem}. 

To obtain Items \eqref{ProjectionItem} -- \eqref{KernelItem}, we observe that $\Pi\circ(\Phi_{\mathbf{f}}|_{\mathcal{O}_{\mathbf{f}}})^{-1}$ is the usual projection of $C^{k,\alpha}(\mathcal{L}(\Gamma); \mathbb{R}^{n+1})\times\mathbb{R}^M$ onto its first component. We invoke Item \eqref{SubmanifoldItem} of Theorem \ref{SmoothDependThm} to see that $\mathrm{codim}\, \mathcal{S}_{\mathbf{f}}=\dim \mathcal{K}_{\mathbf{v}}=M$ and $\{0\}\times\mathbb{R}^M$ is contained in the tangent space of $\mathcal{S}_{\mathbf{f}}$ at $\Phi_{\mathbf{f}}[\mathbf{f}]$. Hence, the claims in Items \eqref{ProjectionItem} -- \eqref{KernelItem} follow easily from these observations.
\end{proof}

\subsection*{Acknowledgements} The second author would like to thank David Hoffman, Bing Wang and Brian White for helpful discussions.

\appendix

\section{Estimates for the heat equation on $\mathbb{R}^n$} \label{EstHeatEqnApp}
We first state a well-known maximum principle for parabolic equations on $\mathbb{R}^n$. 

\begin{prop} \label{MaxPrincipleProp}
If $w\in C^0(\mathbb{R}^n\times [0,T])$ has continuous spatial derivatives up to the second order and continuous time derivative and satisfies
$$
\left\{
\begin{array}{cc}
\partial_t w-\sum\limits_{i,j=1}^n a^{ij}\partial_{x_ix_j}^2 w-\sum\limits_{i=1}^n b^i\partial_{x_i}w-cw\leq 0 & \mbox{in $\mathbb{R}^n\times (0,T)$} \\
w(x,0)\leq 0 & \mbox{for $x\in\mathbb{R}^n$}
\end{array}
\right.
$$
where for some $\lambda>0$,
$$
\sum_{i,j=1}^n a^{ij}(x,t)\xi_i\xi_j \geq 0 \mbox{ and } \sum_{i,j=1}^n |a^{ij}(x,t)|+\sum_{i=1}^n |b^i(x,t)|+|c(x,t)| \leq \lambda,
$$
then $w\leq 0$ on $\mathbb{R}^n\times [0,T]$. 
\end{prop}

\begin{proof}
For each $l>0$ we define 
$$
w_l(x,t)=e^{-2\lambda t} w(x,t)-l(1+|x|^2)^{\frac{1}{2}}.
$$
Then, by direct computations, 
$$
\partial_t w_l-\sum_{i,j=1}^n a^{ij}\partial_{x_ix_j}^2 w_l-\sum_{i=1}^n b^i \partial_{x_i} w_l+(2\lambda-c) w_l \leq C(n,\lambda) l.
$$
Clearly, $w_l(x,0)\leq 0$, and as $w\in C^0(\mathbb{R}^n\times [0,T])$, $w_l(x,t)\to -\infty$ when $(x,t)$ approaches infinity. 

Thus, there is a point $(x_l,t_l)\in \mathbb{R}^n\times [0,T]$ such that 
$$
\sup_{(x,t)\in\mathbb{R}^n\times [0,T]} w_l(x,t)=w_l(x_l,t_l). 
$$
If $t_l>0$, then 
$$
\partial_t w_l(x_l,t_l)\geq 0, \, \partial_{x_i} w_l(x_l,t_l)=0, \mbox{ and } \sum_{i,j=1}^n a^{ij}(x_l,t_l)\partial_{x_ix_j}^2 w_l(x_l,t_l)\leq 0.
$$
As $|c|\leq\lambda$, it follows that $w_l(x_l,t_l)\leq C l \lambda^{-1}$. If $t_l=0$, then $w_l(x_l,t_l)\leq 0$. Hence, 
$$
\sup_{(x,t)\in\mathbb{R}^n\times [0,T]} w_l(x,t)\leq C l \lambda^{-1}.
$$
Now passing $l\to 0$, we get 
$$
\sup_{(x,t)\in\mathbb{R}^n\times [0,T]} w(x,t)\leq 0,
$$
proving the lemma.
\end{proof}

We now prove Lemma \ref{CauchyProbLem} for the heat equation on $\mathbb{R}^n$.

\begin{prop} \label{SpecialCauchyProbProp}
Let $\beta\in (0,1)$. Given $h\in C^0((0,1); C^\beta(\mathbb{R}^n))$ the Cauchy problem
\begin{equation} \label{SpecialCauchyProbEqn}
\left\{
\begin{array}{cc}
\partial_t w-\Delta w=h & \mbox{in $\mathbb{R}^n\times (0,1)$} \\
\displaystyle \lim_{(x^\prime,t)\to (x,0)} w(x^\prime,t)=0 & \mbox{for $x\in\mathbb{R}^n$}
\end{array}
\right.
\end{equation}
has a unique solution in $C^0((0,1); C^{2,\beta}(\mathbb{R}^n))$. Moreover, $w$ satisifes
$$
\sup_{0<t<1} \sum_{i=0}^2 t^{\frac{i-2}{2}} \Vert \nabla^i w(\cdot,t)\Vert_{\beta} \leq \nu(n,\beta) \sup_{0<t<1} \Vert h(\cdot,t)\Vert_\beta.
$$
\end{prop}

\begin{proof}
On $\mathbb{R}^n\times (0,1)$ we define
$$
w(x,t)=\int_0^t \int_{\mathbb{R}^n} h(y,s) \Phi(x-y,t-s) \, dyds,
$$
where $\Phi(x-y,t-s)=(4\pi (t-s))^{-\frac{n}{2}} e^{-\frac{|x-y|^2}{4(t-s)}}$. It follows from a direct calculation -- cf. \cite[pp. 263--264]{LSU} -- that $\partial_t w-\Delta w=h$. 

Observe that $|w(x,t)|\leq t\Vert h\Vert_0$, and that
\begin{align*}
|w(x,t)-w(x^\prime,t)| & \leq \int_0^t \int_{\mathbb{R}^n} |h(x-y,s)-h(x^\prime-y,s)| \Phi(y,t-s) \, dyds \\
& \leq t \sup_{0<s<1} [h(\cdot,s)]_\beta |x-x^\prime|^\beta.
\end{align*}
Thus,
$$
\sup_{0<t<1} t^{-1} \Vert w(\cdot,t)\Vert_{\beta} \leq \sup_{0<t<1}\Vert h(\cdot,t)\Vert_\beta.
$$

Next we use \cite[Chapter 4, (2.5)]{LSU} to estimate
$$
|\partial_{x_i} w(x,t)| \leq \Vert h\Vert_0\int_0^t \int_{\mathbb{R}^n} |\partial_{x_i}\Phi(x-y,t-s)| \, dyds \leq C(n) \sqrt{t} \Vert h\Vert_0.
$$
Moreover, 
\begin{align*}
|\partial_{x_i} w(x,t)-\partial_{x_i} w(x^\prime,t)| & \leq \int_0^t \int_{\mathbb{R}^n} |h(x-y,s)-h(x^\prime-y,s)| |\partial_{y_i}\Phi(y,t-s)| \, dyds \\
& \leq C(n)\sqrt{t}\sup_{0<s<1} [h(\cdot,s)]_\beta |x-x^\prime|^\beta.
\end{align*}
Thus,
$$
\sup_{0<t<1} t^{-\frac{1}{2}} \Vert \partial_{x_i} w(\cdot,t)\Vert_{\beta} \leq C(n) \sup_{0<t<1} \Vert h(\cdot,t)\Vert_{\beta}.
$$

Next we use \cite[Chapter 4, (1.9) and (2.5)]{LSU} to estimate
\begin{align*}
|\partial_{x_ix_j}^2 w(x,t)| & \leq \int_0^t \int_{\mathbb{R}^n} |h(y,s)-h(x,s)| |\partial_{x_ix_j}^2\Phi(x-y,t-s)| \, dyds \\
& \leq \sup_{0<s<1} [h(\cdot,s)]_\beta \int_0^t \int_{\mathbb{R}^n} |x-y|^\beta |\partial_{x_ix_j}^2\Phi(x-y,t-s)| \, dyds \\
& \leq C^{\prime}(n,\beta)\sup_{0<s<1} [h(\cdot,s)]_\beta.
\end{align*}
This together with the H\"{o}lder estimate \cite[pp. 276--277]{LSU} of $\partial_{x_ix_j}^2 w$ gives
$$
\sup_{0<t<1} \Vert\partial_{x_ix_j}^2 w(\cdot,t)\Vert_\beta \leq C^{\prime}(n,\beta) \sup_{0<t<1} \Vert h(\cdot,t)\Vert_\beta.
$$

Hence, we have proved that $w$ is a classical solution to the problem \eqref{SpecialCauchyProbEqn} satisfying 
$$
\sup_{0<t<1} \sum_{i=0}^2 t^{\frac{i-2}{2}} \Vert \nabla^i w(\cdot,t)\Vert_\beta \leq C^{\prime\prime}(n,\beta) \sup_{0<t<1} \Vert h(\cdot,t)\Vert_\beta.
$$
Moreover, observe that for $(x,t)\in\mathbb{R}^n\times (0,1)$ and $0<\delta<1-t$, 
\begin{align*}
w(x,t+\delta)-w(x,t) & =\int_0^t \int_{\mathbb{R}^n} (h(y,s+\delta)-h(y,s))\Phi(x-y,t-s) \, dyds \\
& \quad +\int_{-\delta}^0 \int_{\mathbb{R}^n} h(y,s+\delta) \Phi(x-y,t-s) \, dyds.
\end{align*}
By the preceding discussions, if $h\in C^0((0,1); C^\beta(\mathbb{R}^n))$, then $w\in C^0((0,1); C^{2,\beta}(\mathbb{R}^n))$. The uniqueness follows from Proposition \ref{MaxPrincipleProp}.
\end{proof}

\end{document}